\documentclass{amsart}[12pt]
\usepackage[hidelinks]{hyperref}
\usepackage{verbatim}
\usepackage{yfonts} %
\usepackage{amssymb} %
\usepackage{amsthm}
\usepackage{array}
\usepackage{booktabs}%
\usepackage{hhline}%
\usepackage{xy} %
\usepackage{epsfig}%
\usepackage{color}%
\usepackage{upgreek}
\usepackage[english]{babel}
\usepackage{epigraph}%
\usepackage{fancybox}%
\usepackage{shadow}
\usepackage{afterpage}
\usepackage{mathrsfs}
\usepackage{enumitem}
\usepackage{tabularx}
\usepackage{subcaption}
\usepackage{graphicx}
\usepackage{type1cm}
\usepackage{eso-pic}
\usepackage{color}
\usepackage{upgreek}
\usepackage{bigints}

\newtheorem{theorem}{Theorem}

\newtheorem{proposition}{Proposition}
\theoremstyle{definition}
\newtheorem{definition}{Definition}
\newtheorem{remark}{Remark}
\newtheorem{example}{Example}

\theoremstyle{plain}

\newcommand{\otoprule}{\midrule[\heavyrulewidth]}

\newcommand{\vt}{\vspace{.1cm}}
\newcommand{\vtt}{\vspace{.2cm}}
\newcommand{\R}{\mathbb{R} }
\newcommand{\q}{\mathbb{Q} }

\newcommand{\h}{\mathbb{H}}

\newcommand{\hr}{\h^n\times\R}

\newcommand{\s}{\mathbb{S}}

\renewcommand{\rho}{\varrho}
\renewcommand{\theta}{\varTheta}
\renewcommand{\Theta}{\varTheta}
\renewcommand{\Sigma}{\varSigma}
\renewcommand{\Omega}{\varOmega}
\renewcommand{\Lambda}{\varLambda}
\renewcommand{\tau}{\uptau}
\usepackage{amsmath}% http://ctan.org/pkg/amsmath

\newcommand{\overbar}[1]{\mkern 1.5mu\overline{\mkern-1.5mu#1\mkern-1.5mu}\mkern 1.5mu}

\newcommand{\qr}{\q_\epsilon^n\times\R}

\newcommand{\transv}{\mathrel{\text{\tpitchfork}}}
\makeatletter
\newcommand{\tpitchfork}{%
  \vbox{
    \baselineskip\z@skip
    \lineskip-.52ex
    \lineskiplimit\maxdimen
    \m@th
    \ialign{##\crcr\hidewidth\smash{$-$}\hidewidth\crcr$\pitchfork$\crcr}
  }%
}
\makeatother

\begin{document}

\title[]
{Translators to Higher Order Mean Curvature Flows in $\R^n\times\R$ and
$\h^n\times\R$}
\author{Ronaldo F. de Lima \and Giuseppe Pipoli.}

\address[A1]{Departamento de Matem\'atica - Universidade Federal do Rio Grande do Norte}
\email{ronaldo.freire@ufrn.br}
\address[A2]{Department of Information Engineering, Computer Science and Mathematics, Università degli Studi
dell’Aquila.}
\email{giuseppe.pipoli@univaq.it}

\maketitle

\begin{abstract}
We consider translators to the extrinsic flows in  $\mathbb R^n\times\mathbb R$
and $\mathbb H^n\times\mathbb R$ (called $r$-mean curvature flows or $r$-MCF, for short)
whose velocity functions are  the higher order mean curvatures $H_r.$
We show that there exist rotational bowl-type and catenoid-type
translators to $r$-MCF in both $\mathbb R^n\times\mathbb R$ and
$\mathbb H^n\times\mathbb R,$ and also that there exist
parabolic and hyperbolic catenoid-type translators
to $r$-MCF in $\mathbb H^n\times\mathbb R.$ In addition, we show that
there exist grim reaper-type translators to Gaussian flow ($n$-MCF) in
$\mathbb R^n\times\mathbb R$ and $\mathbb H^n\times\mathbb R$.
We also establish the uniqueness  of
all these translators (together with certain cylinders) among those which
are invariant by either rotations or translations
(Euclidean, parabolic or hyperbolic). We apply this uniqueness result to classify
the translators to $r$-MCF in $\mathbb R^n\times\R$ and $\mathbb H^n\times\mathbb R$
whose $r$-th mean curvature is constant, as well as those which are isoparametric.
Our results  extend to the context of $r$-MCF in
$\mathbb R^n\times\mathbb R$ and $\mathbb H^n\times\mathbb R$
the existence and uniqueness theorems by  Altschuler--Wu (of the bowl soliton) and
Clutterbuck--Schn\"urer--Schulze (of the translating catenoids) in Euclidean space.

\vspace{.15cm}

\vspace{.1cm}

\end{abstract}

\section{Introduction}

Extrinsic geometric flows of hypersurfaces in Riemannian manifolds is
a most prominent topic in submanifold theory. Such a flow is generated by
a hypersurface moving  in the direction of its normal vector
with speed  given by a smooth symmetric function  of its principal curvatures.
When this movement constitutes a continuous translation
in a fixed direction, such a hypersurface is called a \emph{translator} to the flow.
The mean curvature flow (MCF, for short), that is, the extrinsic flow determined
by the mean curvature function,  is certainly the most studied extrinsic flow.
Indeed, there is a vast literature on  MCF in Euclidean
spaces and, in particular, on translators to MCF. In this context,
it is well known that translators appear naturally as type II singularities (cf. \cite{huisken-sinestrari}).

The rotationally symmetric translators to MCF in Euclidean space $\R^{n+1}=\R^n\times\R$ are completely classified.
They constitute the entire graph obtained by Altschuler and Wu~\cite{altschuler-wu} known as the
\emph{bowl soliton} or \emph{translating paraboloid}, and the one-parameter family  of annuli obtained by
Clutterbuck, Schn\"urer, and Schulze~\cite{schulzeetal} known as  \emph{translating catenoids}
(see~\cite[Section 13.1]{andrewsetal} and the references therein for a detailed account on translators to
MCF in Euclidean spaces).

Translators are naturally conceived in product spaces $M\times\R,$ where $M$ is a Riemannian manifold.
On this matter, Bueno~\cite{bueno} managed to construct bowl-type and catenoid-type
translators to MCF in $\h^2\times\R.$ Also, in~\cite{lira-martin}, Lira and Martín considered translators to MCF in products
$M\times\R,$ where $M$ is a Hadamard manifold endowed with a rotationally invariant metric.
There, they constructed bowl-type and  catenoid-type rotational translators, as well as
translators which are invariant by either parabolic or hyperbolic
horizontal  translations of $M\times\R$.
They also classified translators to MCF which are invariant
by either rotations or translations. However, their list of translators
having this property seems to be incomplete (cf.~Remark~\ref{rem-liramartin} in Section~\ref{sec-uniqueness}).
We add that, in~\cite{delima}, the first author proved the existence of
graphical translators to flows by powers of the Gaussian curvature in $\h^n\times\R$ and $\s^n\times\R,$ and that,
in~\cite{pipoli1,pipoli2}, the second author classified all translators
to MCF  in the  solvable group ${\rm Sol}_3$,
as well as in Heisenberg group ${\rm Nil}_3$, which are invariant by some
one-parameter group of ambient isometries.

In this paper, we consider translators to the extrinsic flows in  $\mathbb R^n\times\mathbb R$
and $\mathbb H^n\times\mathbb R$ (called $r$-mean curvature flows or $r$-MCF, for short)
whose velocity functions are  the higher order mean curvatures $H_r$.
Recall that the (nonnormalized) $r$-mean curvature of a hypersurface is the
homogeneous polynomial of degree $r$ of its principal curvatures, so that
$H_1$ is the mean curvature and $H_n$ is the Gaussian curvature.

More precisely, we address the problem of constructing  and classifying  translators to $r$-MCF which are
invariant by either rotations or horizontal translations (Euclidean, parabolic or hyperbolic).
We were motivated by the fact that $r$-mean curvature flows are particular examples of an importante and large class of fully
nonlinear extrinsic flows (cf.~\cite[Chapter 18]{andrewsetal}). Yet,
except for the cases $r=1$ and $r=n$, such translators have never been considered,
not even in Euclidean space.

We show that there exist rotational bowl-type and catenoid-type
translators to $r$-MCF in both $\R^n\times\R$ and $\h^n\times\R,$ and also that there exist
parabolic and hyperbolic catenoid-type translators to $r$-MCF in $\h^n\times\R.$
In addition,  we show that there exist grim reaper-type translators to MCF and to Gaussian curvature flow ($n$-MCF) in
$\mathbb R^n\times\mathbb R$ and $\mathbb H^n\times\mathbb R$.

Regarding our technique, we obtain the above $r$-translators by considering them as  graphs whose level
sets are parallel umbilical hypersurfaces of $\R^n$ or $\h^n$. In this way, by imposing on these graphs the condition
of being $r$-translators, we obtain  ordinary differential equations whose solutions yield the height
functions of the $r$-translators we aim to construct. We add that these equations are nonlinear and,  for $r>1$,
they  are more involved than the ones for $r=1$. Besides,
for $r>1$ odd, the catenoid-type translators to $r$-MCF have nonempty singular sets of null measure.
We also remark that our method of using graphs built on parallel hypersurfaces is new in this theory.
In fact, when applied to the case $r=1$, it gives
simpler proofs of  the aforementioned existence results by
Altschuler--Wu and Clutterbuck--Schn\"urer--Schulze, as well as
the ones by Bueno.

We  establish the uniqueness  of the translators to $r$-MCF we obtain here, which we call
\emph{fundamental translators}, among those which are invariant by either rotations or translations.
Then, we classify the translators to $r$-MCF whose $r$-th mean
curvature is constant, as well as those which are isoparametric.
We also characterize the non-cylindrical translators to MCF
whose angle function is constant along their horizontal sections as those which are
local graphs foliated by isoparametric hypersurfaces. In addition, we verify an interesting
phenomenon; up to an ambient isometry,
two distinct fundamental translators to $r$-MCF are asymptotic to each other, regardless
the groups of isometries fixing them (cf.~Remark~\ref{rem-asymptotic} in Section~\ref{sec-uniqueness}).

We should mention that, at first, our intention was to consider $r$-translators in $\s^n\times\R$ as well. However,
in this case, the associated differential equations behave quite differently from the ones
we have when the ambient space is $\R^n\times\R$ or $\h^n\times\R$, making a
unified treatment impossible. Hence, to avoid the paper becoming lengthy,
we have chosen to consider $r$-translators in $\s^n\times\R$ in a forthcoming work.

The paper is organized as follows. In Section~\ref{sec-preliminaries}, we set some notation
and introduce the notion of vertical graph in $\R^n\times\R$ and $\h^n\times\R$  whose level
sets are parallel hypersurfaces.
In Section~\ref{sec-translators}, we discuss $r$-mean curvature flows in
$\R^n\times\R$ and $\h^n\times\R$, establishing
some fundamental results. In Section~\ref{sec-rotationaltranslators}, we prove the
existence of  rotational bowl-type and  catenoid-type
translators to $r(<\hspace{-.1cm}n)$-MCF in $\R^n\times\R$ and $\h^n\times\R$.
The parabolic and hyperbolic versions of these
results for translators in $\h^n\times\R$
are obtained in Sections~\ref{sec-parabolic} and~\ref{sec-hyperbolic},
respectively. In Section~\ref{sec-gaussian}, we consider translators
to Gaussian curvature flow, proving the existence of bowl-type and
grim reaper-type ones.  Finally, in Section~\ref{sec-uniqueness}, we
establish the  uniqueness results for fundamental translators we mentioned above.

\section{Preliminaries} \label{sec-preliminaries}

\subsection{Isoparametric hypersurfaces} \label{sec-isoparametric}
Let  $\mathbb M^n$ be a Riemannian manifold. Given an open interval $I\subset\R$,
one says that a one-parameter family
$$f_s\colon M^{n-1}\rightarrow\mathbb M^n, \,\,\, s\in I,$$
of immersions is \emph{parallel} if, for a fixed $s_0\in I$, one has
\begin{equation} \label{eq-immersion}
f_s(p):=\exp_p(s\eta_{s_0}(p)), \,\,\, p\in M, \, s\in I,
\end{equation}
where $\exp$ is the exponential map of $\mathbb M^n$ and $\eta_{s_0}$ is the unit
normal of $f_{s_0}$.  In this setting, the hypersurfaces $M_s:=f_s(M)$, $s\in I$, are
also called \emph{parallel}.

A family of parallel hypersurfaces
$$\{M_s\subset\mathbb M^{n}\,;\, s\in I\subset\R\}$$ of
a Riemannian manifold $\mathbb M^{n}$
is called \emph{isoparametric}
if each hypersurface $M_s$ has constant mean curvature
(possibly depending on $s$). If so, each hypersurface $M_s$ is also called
\emph{isoparametric}. The  isoparametric hypersurfaces of $\R^{n}$, as well as those of
$\h^n$, are totally  classified. Indeed, any such  hypersurface is
necessarily an open set of either a umbilical hypersurface or
a tube over a totally geodesic submanifold
of codimension greater than one
(cf. \cite[Theorems 3.12 and 3.14]{cecil-ryan}).

\subsection{Hypersurfaces of $\qr$}
We shall consider oriented hypersurfaces in the product
$\qr$ endowed with its standard product metric,
where $\q_\epsilon^n$ denotes the simply connected space form of
constant sectional curvature $\epsilon\in\{0,-1\},$ i.e., Euclidean space
$\R^n$ or hyperbolic space $\h^n.$

Given an oriented hypersurface $\Sigma$ of $\qr,$
set $N$ for its unit normal field  and $A$ for its shape operator  with respect to
$N,$ so that
\[
AX=-\overbar\nabla_XN,  \,\, X\in T\Sigma,
\]
where $\overbar\nabla$ is the Levi-Civita connection of $\qr,$ and $T\Sigma$ is the tangent bundle of $\Sigma$.
The principal curvatures of $\Sigma,$ that is, the eigenvalues of the shape operator
$A,$ will be denoted by $k_1\,, \dots ,k_n$.

We define the \emph{height function} $\phi$ and the   \emph{angle function} $\Theta$
of $\Sigma$ as:
\[
\phi:=\pi_{\scriptscriptstyle\R}|_\Sigma \quad\text{and}\quad \Theta:=\langle N,\partial_t\rangle,
\]
where $\partial _t$ denotes the gradient of the projection $\pi_{\scriptscriptstyle\R}$ of
$\qr$ on its second factor $\R.$ Notice that $\partial_t$ is a parallel field on $\qr.$ So,
denoting by $\nabla $ the gradient  on $C^\infty(\Sigma)$ and writing
$T:=\nabla\phi,$ we have that the identities
\begin{equation} \label{eq-T}
T=\partial_t-\theta N\quad\text{and}\quad AT=-\nabla\theta
\end{equation}
hold everywhere on $\Sigma.$ From the first of them, one has:
\[
\|T\|^2=1-\theta^2.
\]

Given an integer $r\in\{1,\dots,n\},$ recall that the (non normalized)
$r$-th \emph{mean curvature} $H_r$ of a hypersurface $\Sigma$ of $\qr$
is the function:
\[
H_r:=\sum_{i_1<\cdots <i_r}k_{i_1}\dots k_{i_r}\,.
\]
Notice that $H_1$ and $H_n$ are  the  non normalized mean curvature  and
the Gaussian curvature of $\Sigma,$ respectively.

\subsection{Graphs on parallel hypersurfaces}
Let $\mathscr F:=\{M_s\subset\q_\epsilon^n\,;\, s\in I\}$
be a family of parallel hypersurfaces of $\q_\epsilon^n,$ where $I\subset\R$ is an open interval.
Given a smooth function $\phi$ on $I,$ let
$$f\colon M_{s_0}\times I\rightarrow\qr, \,\,\, s_0\in I,$$
be the immersion given by
\begin{equation} \label{eq-immersion}
f(p,s):=(\exp_p(s\eta_{s_0}(p)),\phi(s)), \,\,\, (p,s)\in M_{s_0}\times I,
\end{equation}
where $\exp$ denotes the exponential map of $\q_\epsilon^n,$ and $\eta_{s_0}$ is the unit
normal of $M_{s_0}.$ The hypersurface $\Sigma=f(M_{s_0}\times I)$ is a vertical graph
over an open set of $\q_\epsilon^n$ whose level hypersurfaces
are the parallels $M_s$ to $M_{s_0}.$

\begin{definition}
With the above notation, we shall call $\Sigma$ an $(M_s,\phi)$-\emph{graph}.
\end{definition}

As proved in~\cite{delima-manfio-santos},  the unit normal $N$ of $\Sigma$
(when endowed with the metric induced by $f$) at a point
$(p,s)\in M_{s_0}\times I$ is
\begin{equation} \label{eq-unitnormal}
N=-\rho(s)\eta_s(p)+\theta\partial_t,
\end{equation}
where $\rho$ is the function defined by
\begin{equation}\label{eq-rho}
\rho:=\frac{\phi'}{\sqrt{1+(\phi')^2}} \,\, \left( \Leftrightarrow \phi'=\frac{\rho(s)}{\sqrt{1-\rho^2(s)}}\right),
\end{equation}
and $\theta$ is the angle function of $\Sigma$.
With this orientation, the principal curvatures
$k_i=k_i(p,s)$ of an $(M_s,\phi)$-graph $\Sigma$
at a point $(p,s)\in M_{s_0}\times I$ are:
\begin{equation} \label{eq-principalcurvatures}
k_i=-\rho(s)k_i^s(p), \,\,\, i=1,\dots, n-1, \quad\text{and}\quad k_n=\rho'(s),
\end{equation}
where $k_i^s(p)$ is the principal curvature function of the parallel $M_s$ at
$\exp_p(s\eta_{s_0}(p)),$ $p\in M_{s_0}.$

By integrating~\eqref{eq-rho}, we conclude that  $\rho$ determines the height function
$\phi$ up to a constant.
More precisely:
\begin{equation}\label{eq-phi10}
\phi(s)=\int_{s_0}^{s}\frac{\rho(u)}{\sqrt{1-\rho^2(u)}}du+\phi(s_0), \,\,\, s_0, s\in I.
\end{equation}
It can also be proved that the equality
\begin{equation} \label{eq-theta2}
\rho^2+\theta^2=1
\end{equation}
holds everywhere on  $\Sigma$
(see Section 3 of~\cite{delima-manfio-santos} for details and proofs).

From equalities~\eqref{eq-rho} and~\eqref{eq-theta2}, we have the following relation between
$\phi'$ and $\theta$:
\begin{equation} \label{eq-thetasquared}
\theta^2=\frac{1}{1+(\phi')^2}\cdot
\end{equation}

\subsection{Umbilical hypersurfaces of $\q_\epsilon^n$}
The $(M_s,\phi)$-graphs we shall consider here are those whose parallel level
hypersurfaces are  umbilical. Recall that the umbilical hypersurfaces
of $\q_\epsilon^n$ are:

\begin{itemize}[parsep=1ex]
  \item The totally geodesic hyperplanes $\q_\epsilon^{n-1}\subset \q_\epsilon^{n}.$
  \item The geodesic spheres $S_s^{n-1}\subset \q_\epsilon^{n}$ of radius $s>0.$
  \item The horospheres $\mathcal H^{n-1}$ of $\h^{n}.$
  \item The equidistant hypersurfaces $\mathcal E^{n-1}$ to totally geodesic hyperplanes of $\h^{n}.$
\end{itemize}

\begin{table}[thb]
\centering
\begin{tabular}{ccc}
\toprule
   {{\small\rm Function}} & $\epsilon=0$ &  $\epsilon=-1$ \\\otoprule
$\cos_\epsilon (s)$       & $1$          &  $\cosh s$     \\\midrule
$\sin_\epsilon (s)$       & $s$          &  $\sinh s$     \\\bottomrule
\end{tabular}
\vtt
\caption{Definition of $\cos_\epsilon$ and $\sin_\epsilon.$}
\label{table-trigfunctions}
\vspace{-.5cm}
\end{table}

Defining $\cos_\epsilon$ and $\sin_\epsilon$  as in
Table \ref{table-trigfunctions}, and setting
\[
\tan_\epsilon=\frac{\sin_\epsilon}{\cos_\epsilon}\quad\text{and}\quad \cot_\epsilon=\frac{1}{\tan_\epsilon}\,,
\]
we have that the principal curvatures of the umbilical hypersurfaces
of $\q_\epsilon^n$ (endowed with the outward orientation) are as indicated in
Table \ref{table-principalcurvatures}.
\begin{table}[thb]
\centering
\begin{tabular}{cc}
\toprule
Hypersurface $M_s$                  &  Principal curvatures  ($i=1,\dots, n-1$)                                     \\\otoprule
$\q_\epsilon^{n-1}$                 &  $k_i^s=0$                   \\\midrule
$S_s^{n-1}$                           &  $k_i^s=-\cot_\epsilon(s)$          \\\midrule
$\mathcal H^{n-1}$                  &  $k_i^s=-1$       \\\midrule
$\mathcal E^{n-1}$                  &  $k_i^s=-\tanh(s)$      \\\bottomrule
\end{tabular}
\vtt
\caption{Principal curvatures of the umbilical hypersurfaces of $\q_\epsilon^n$.}
\label{table-principalcurvatures}
\vspace{-.5cm}
\end{table}

\subsection{Invariant hypersurfaces of $\qr$} \label{subsec-symmetric}
In hyperbolic space $\h^n$,   there are three special
types of one-parameter families of
isometries. They are  the rotations around a fixed point (\emph{elliptic isometries}),
the translations along  horocycles sharing the same point at infinity (\emph{parabolic isometries}),
and the translations along a fixed geodesic (\emph{hyperbolic isometries}).
In Euclidean space, the one-parameter
groups of rotations, as well as of translations in a fixed direction, are well known.

Observe that each one of the aforementioned groups of isometries of $\q_\epsilon^n$
fixes a family of umbilical hypersurfaces. Indeed, an elliptic isometry fixes
a family of parallel spheres, whereas a parabolic (resp.~hyperbolic) translation
fixes a family of parallel horospheres (resp.~equidistant hypersurfaces). A translation
in a fixed direction in $\R^n$ fixes a family of parallel hyperplanes.

These isometries of $\q_\epsilon^n$ extend naturally to isometries of $\q_\epsilon^n\times\R$
(which we call \emph{horizontal}) by  fixing the factor $\R$ pointwise.
Therefore, a hypersurface $\Sigma\subset\qr$ which is invariant by such a group
of  isometries is necessarily foliated by vertical
translations of its corresponding family of umbilical hypersurfaces.
We shall call such a $\Sigma$ an \emph{invariant} hypersurface of $\qr.$
An invariant hypersurface of $\hr$ will be called \emph{parabolic} (resp.~\emph{hyperbolic})
if it is invariant by horizontal parabolic translations (resp.~hyperbolic translations).

\section{Translators to the $r$-th Mean Curvature Flow} \label{sec-translators}

Given positive integers $n\ge 2$ and $r\in\{1,\dots, n\},$
we say that an oriented hypersurface $\Sigma$ of $\qr$
\emph{moves under} $H_r$-flow if there exists a one-parameter
family of immersions $F\colon \Sigma_0\times[0,u_0)\rightarrow\qr,$ $u_0\le+\infty,$
such that
\begin{equation} \label{eq-Hrflow}
\left\{
\begin{array}{l}
\frac{\partial F}{\partial u}^\perp(p,u)=H_r(p,u)N(p,u).\\[1ex]
F(\Sigma_0\,, 0)=\Sigma,
\end{array}
\right.
\end{equation}
where $N(p,u)$ is the inward
unit normal to the hypersurface $F_u:=F(.\,, u),$
$H_r(p,u)$ is the $r$-th mean curvature
of $F_u$ with respect to $N_u:=N(.\,, u),$ and
$\frac{\partial F}{\partial u}^\perp$ denotes the normal component of
$\frac{\partial F}{\partial u},$  i.e.,
\[
\frac{\partial F}{\partial u}^\perp=\left\langle\frac{\partial F}{\partial u},N_u\right\rangle N_u\,.
\]

In particular, the first equality in \eqref{eq-Hrflow} is equivalent to
\begin{equation}  \label{eq-condition}
\left\langle\frac{\partial F}{\partial u}(p,u),N(p,u)\right\rangle=H_r(p,u).
\end{equation}
We call such a map $F$  an $H_r$-\emph{flow} in $\qr.$

Denote by $\exp$ the exponential map of $\qr$ and
consider an isometric immersion $F_0\colon \Sigma_0\rightarrow\qr$.
Define then the map
\[
F(p,u):=\exp_{F_0(p)}(u\partial_t), \,\,(p,u)\in\Sigma_0\times [0,+\infty),
\]
and notice that, for each $u\in (0,+\infty),$ the hypersurface  $F(\Sigma_0\,, u)$ is nothing
but an upwards vertical translation of $\Sigma:=F(\Sigma_0,0).$ Since vertical translations are isometries of
$\qr$,
we have that $\Sigma$ and $F(\Sigma_0\,, u)$
are congruent, so that  their angle functions and $r$-th
mean curvature functions coincide, that is,
\begin{equation} \label{eq-angles}
\theta(p,u)=\theta(p,0) \quad\text{and}\quad H_r(p,u)=H_r(p,0)\,\,\, \forall (p,u)\in\Sigma_0\times [0,u_0).
\end{equation}

Now, differentiating $F$ with respect to $u,$  we have
\begin{equation} \label{eq-partialt}
\frac{\partial F}{\partial u}(p,u)=(d\exp_{F_0(p)})(u\partial_t)\partial_t=\partial_t\,.
\end{equation}

From \eqref{eq-angles} and \eqref{eq-partialt}, we have that
$F$ satisfies \eqref{eq-condition} if and only if
the  equality
\[
\theta (p,0)=H_r(p,0)
\]
holds for all  $p\in\Sigma_0.$ This fact motivates the following concept.

\begin{definition}
Given positive integers $n\ge 2$ and $r\in\{1,\dots, n\},$
we say that a hypersurface $\Sigma$ of $\qr$ is a \emph{translator} (or a
\emph{translating soliton})
to the $r$-th mean curvature flow ($r$-MCF, for short), if the equality
$H_r=\theta$ holds everywhere on $\Sigma.$
We shall also call a translator to $r$-MCF an $r$-\emph{translator}.
\end{definition}

\begin{example} \label{exem-stationary}
Let $\q_\epsilon^{n-1}\subset\q_\epsilon^n$ be a totally geodesic hyperplane of
$\q_\epsilon^n.$ Then, $\Sigma=\q_\epsilon^{n-1}\times\R$ is a totally geodesic hypersurface
of $\qr$ which we call a \emph{vertical hyperplane.} On such a $\Sigma,$
$H_r=\theta=0,$ which implies that
$\Sigma$ is an $r$-translator for all $r\in\{1,\dots,n\}.$ In addition,
from the first equality in \eqref{eq-Hrflow},  $\Sigma$ is
stationary under $r$-MCF. More generally, if $\Gamma\subset\q_\epsilon^n$ is an
$r$-minimal hypersurface for some $r\in\{1,\dots, n-1\}$ (i.e., the $r$-th mean curvature of
$\Gamma$ vanishes everywhere), then the cylinder $\Gamma\times\R$ is a stationary translator
to $r$-MCF in $\qr.$ The same holds for $r=n$ if $\Gamma$ is \emph{any} hypersurface of $\q_\epsilon^n.$
\end{example}

\begin{remark}
We shall  consider $n$-submanifolds $\Sigma$ of
$\qr$ which are of class at least $C^2,$ except on a set of null  measure $\Lambda\subset\Sigma,$
where $\Sigma$ is of class $C^1.$
In this case, we shall say that $\Sigma$ is $C^2$-\emph{singular} on $\Lambda.$ If the equality
$H_r=\theta$ holds on $\Sigma-\Lambda,$ by abuse of terminology,
we still call $\Sigma$ an $r$-{translator}
(see Remark~\ref{rem-catenoidsmove} in Section~\ref{sec-rotationaltranslators}).
We add that, in a similar fashion, some rotational hypersurfaces of
constant $r$-th mean curvature constructed in~\cite{nelli-pipoli-russo} have
$C^2$-singular sets.
\end{remark}

\begin{remark} \label{rem-r-meancurvaturevector}
Setting $\mathbf{k}:=(k_1,\dots ,k_n),$ it is easily seen that the $r$-th mean curvature
function $H_r$ satisfies $H_r(-\mathbf{k})=-H_r(\mathbf{k})$ when $r$ is odd. In this case,
given an orientable hypersurface $\Sigma\subset\qr$ and a unit normal field  $N$
on $\Sigma,$ the $r$-\emph{mean curvature vector}
$\mathbf{H_r}:=H_rN$
is well defined, that is, it is independent of the orientation $N$, and so we can write~\eqref{eq-Hrflow}
as
\[
\left\{
\begin{array}{l}
\frac{\partial F}{\partial u}^\perp(p,u)={\mathbf{H_r}}(p,u),\\[1ex]
F(\Sigma_0\,, 0)=\Sigma.
\end{array}
\right.
\]
On the other hand, for $r$ even, one has
\begin{equation} \label{eq-hreven}
H_r(-\mathbf{k})=H_r(\mathbf{k}),
\end{equation}
so that the $r$-mean curvature vector is not defined.
\end{remark}

\begin{remark} \label{rem-reflectedtranslator}
Let $\Phi$ be the reflection over a horizontal hyperplane $\Pi_t:=\q_\epsilon^n\times\{t\}$
in $\qr.$ Suppose that $\Sigma$ is an $r$-translator in $\qr$ with unit normal $N,$
and call $\overbar\Sigma$ the hypersurface $\Phi(\Sigma)$  with unit normal
$\overbar N\circ\Phi:=-\Phi_*N.$ Then, if $r$ is even, $\overbar\Sigma$ is
an $r$-translator as well. Indeed, in this case, we have from~\eqref{eq-hreven} that
the $r$-mean curvature function is invariant by
change of orientation. This, together with the fact that $\Phi$ is an isometry, gives that the
$r$-mean curvature $H_r$ of $\Sigma$ at a point $p$ coincides with the
$r$-mean curvature $\overbar H_r$ of $\overbar\Sigma$ at $\Phi(p).$ Therefore,
\[
\overbar H_r\circ\Phi=H_r=\langle N,\partial_t\rangle=
\langle \Phi_*N,\Phi_*\partial_t\rangle=\langle\overbar N\circ\Phi,\partial_t\rangle,
\]
which gives that $\overbar\Sigma$ is an $r$-translator.
\end{remark}

\subsection{Graphs on parallels as translators}
Let $\{M_s\,;\, s\in I\}$ be a family of parallel umbilical hypersurfaces of
$\q_\epsilon^n.$  With the notation as in Table \ref{table-principalcurvatures}, set
\begin{equation} \label{eq-alpha}
\alpha(s)=-k_i^s, \quad i=1,\dots, n-1.
\end{equation}

Considering the identities \eqref{eq-principalcurvatures} and writing
\[
H_r = \sum_{i_1<\cdots <i_r\ne n}k_{i_1}\dots k_{i_r}+\sum_{i_1<\cdots <i_{r-1}}k_{i_1}\dots k_{i_{r-1}}k_{n}\,,
\]
we have that the $r$-th mean curvature of an $(M_s,\phi)$-graph $\Sigma$ is
a function of $s$ alone which is given by
\[
H_r={{n-1}\choose{r}}(\alpha\rho)^r+{{n-1}\choose{r-1}}(\alpha\rho)^{r-1}\rho'.
\]

This last equality, together with \eqref{eq-theta2}, gives the following result.
\begin{proposition} \label{prop-main}
Let $\{M_s\,;\, s\in I\}$ be a family of parallel umbilical hypersurfaces of
\,$\q_\epsilon^n,$ and let $\alpha$ be as in \eqref{eq-alpha}. Then, an $(M_s,\phi)$-graph $\Sigma$ in
$\qr$ is an $r$-translator if and only if
its associated $\rho$-function satisfies:
\begin{equation} \label{eq-rMCFodegeneral}
{{n-1}\choose{r}}(\alpha\rho)^r+{{n-1}\choose{r-1}}(\alpha\rho)^{r-1}\rho'=\sqrt{1-\rho^2}.
\end{equation}
\end{proposition}

As a first application of Proposition~\ref{prop-main}, we shall recover
a classical translator to MCF in Euclidean space.

\begin{figure}[hbt]
 \centering
 \includegraphics[scale=.8]{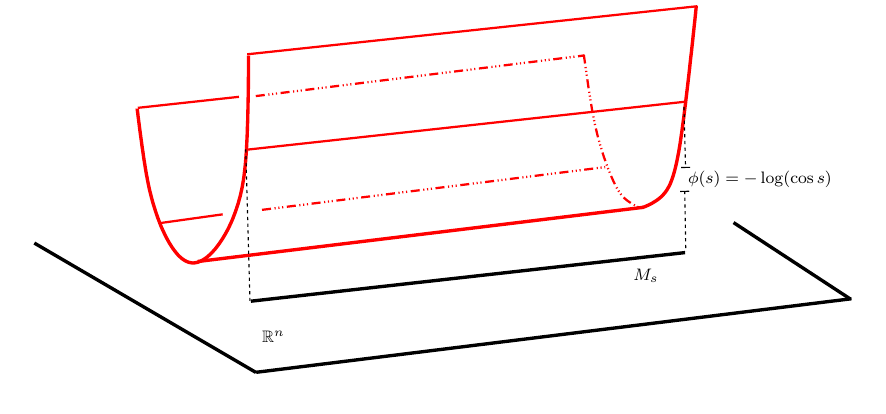}
 \caption{\small The grim reaper.}
 \label{fig-grimreaper}
\end{figure}

\begin{example}[\emph{grim reaper}] \label{exam-grimreaper}
Let $\mathscr F:=\{M_s\,;\, s\in\R\}$ be a family of parallel
totally geodesic hyperplanes in $\R^n.$ Then,
considering Proposition~\ref{prop-main} for $\epsilon=0,$ $r=1,$  and $\mathscr F$,
one has that~\eqref{eq-rMCFodegeneral} becomes
$\rho'=\sqrt{1-\rho^2},$
which gives $\rho(s)=\sin(s).$ Consequently, the height function
of the corresponding $(M_s,\phi)$-graph $\Sigma$ is (assuming $\phi(0)=0$):
\[
\phi(s)=\int_{0}^{s}\frac{\rho(u)}{\sqrt{1-\rho^2(u)}}du=\int_{0}^{s}\tan(u)du=-\log(\cos s), \,\,\, s\in(-\pi/2,\pi/2),
\]
so that  $\Sigma$ is the  solution to MCF in $\R^{n}\times\R$ known as the \emph{grim reaper} (Fig.~\ref{fig-grimreaper}).

Notice that, for  $r>1,$ \eqref{eq-rMCFodegeneral} reduces to
$\sqrt{1-\rho^2}=0,$ giving that $\rho(s)=1$ for all $s\in(-\infty,+\infty).$ In this case, the
corresponding $(M_s,\phi)$-graph degenerates into a vertical totally geodesic hyperplane  of
$\R^{n}\times\R$ (see Example~\ref{exem-stationary}).
\end{example}

Considering the above example, we shall exclude the families
of parallel hyperplanes of $\R^n$ in the discussion that follows, i.e.,
the function $\alpha$ will be $\cot_\epsilon$, $\tanh$ or the constant $1$
(cf.~Table~\ref{table-principalcurvatures}). Under this hypothesis,
setting $\tau=\rho^r,$ the ODE~\eqref{eq-rMCFodegeneral} assumes the form
\begin{equation} \label{eq-ODEtau}
\tau'(s)=C(\alpha(s))^{1-r}\sqrt{1-\tau^{2/r}(s)}-(n-r)\alpha(s)\tau(s),
\end{equation}
where $C=C(n,r)=r{{n-1}\choose{r-1}}^{-1}$. This equality  suggests the consideration of the following Cauchy problem:
\begin{equation} \label{eq-generalIVP}
\left\{
\begin{array}{l}
y'(s)=F(s,y(s))\\[1ex]
y(s_0)=y_0,
\end{array}
\right.
\end{equation}
where $(s_0,y_0)\in\Omega:=(0,+\infty)\times[-1,1]$  and $F=F_{(n,r,\alpha)}$ is the function
\begin{equation} \label{eq-generalF}
F(s,y):=C(\alpha(s))^{1-r}\sqrt{1-y^{2/r}}-(n-r)\alpha(s)y, \,\,\,\, (s,y)\in\Omega.
\end{equation}

Since $F$ is $C^\infty$ in the interior of $\Omega$,
the orbits of the slope field determined by $F$ constitute
a foliation of $\Omega$ by the graphs of the solutions of~\eqref{eq-generalIVP}.
Consequently, the endpoints of such graphs are necessarily boundary points of $\Omega.$

In what follows, we establish the qualitative behavior of the solutions
to~\eqref{eq-generalIVP} as suggested in Figure~\ref{fig-generalplots}. We shall consider first the case  $r<n$.
The case $r=n$ will be treated in Section~\ref{sec-gaussian}.

\begin{figure}[hbt]
 \centering
  \includegraphics[width=6cm]{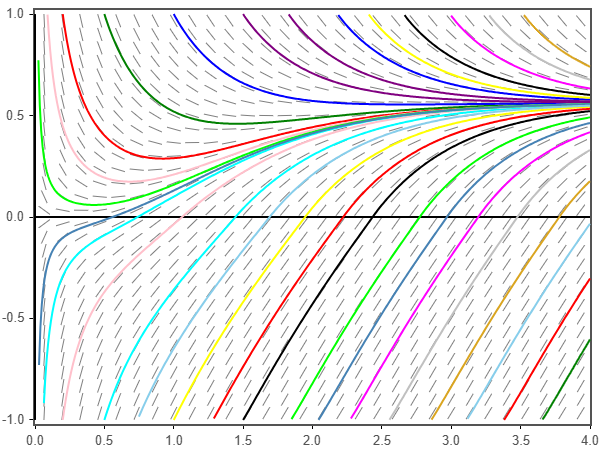}
  \hfill
 \includegraphics[width=6cm]{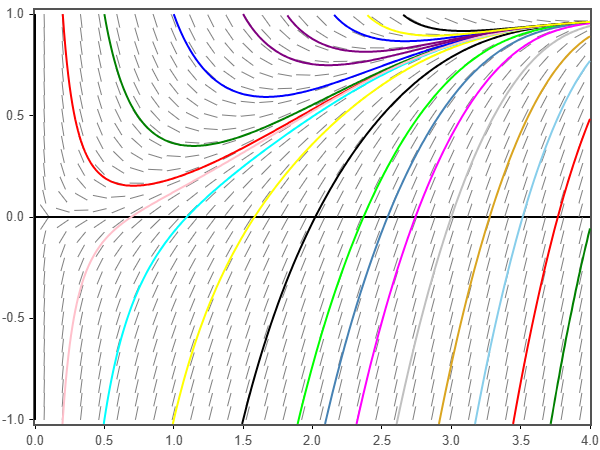}
 \caption{\small Graphs of solutions to~\eqref{eq-generalIVP} for $(n,r,\alpha)=(4,3,\coth)$ (left) and
 $(n,r,\alpha)=(4,2,\cot_0)$ (right).}
 \label{fig-generalplots}
\end{figure}

\begin{definition} \label{def-limitconstant&limitangle}
Given integers $\epsilon\in\{0,-1\}$,  $n\ge 2$, and $r\in\{1,\dots, n\},$
define the \emph{limit constant} $L=L(\epsilon,n,r)\in (0,1]$ as
the only positive number satisfying
\[
C\sqrt{1-L^{2/r}}+\epsilon(n-r)L=0, \quad C=C(n,r)=r{{n-1}\choose{r-1}}^{-1},
\]
and the \emph{limit angle} $\theta_L\in[0,1)$ as
\[
\theta_L:=\sqrt{1-L^{2/r}}.
\]
In particular, $L=1$ (and $\theta_L=0$) if and only if $\epsilon=0$ or $n=r$.
\end{definition}

\begin{proposition} \label{prop-qualitative}
Let $r\in\{1,\dots, n-1\}.$ Given $s_0>0,$ denote by $\tau_{s_0}^-$ and \,$\tau_{s_0}^+$ the solutions of~\eqref{eq-generalIVP}
 for $y_0=-1$ and $y_0=1,$ respectively. Then, $\tau_{s_0}^{\pm}$  are both defined in $[s_0,+\infty)$ and
 have the following properties:
 \begin{itemize}[parsep=1ex]
  \item[\rm i)] $\tau_{s_0}^-$ has one and only one zero $s_1\in(s_0,+\infty)$, and
  its derivative is positive if the function $\alpha$ is either $\cot_\epsilon$ or $1$.
  For $\alpha=\tanh$, $\tau_{s_0}^-$ has at most one critical
  point $s_*>s_1$, which is necessarily a maximum.
%%%%%%%%%%%%%%%%
  \item[\rm ii)] $\tau_{s_0}^+$ is positive in $[s_0,+\infty)$
  and has at most one critical point, which is necessarily a minimum.
\end{itemize}
In addition, the following equalities hold:
 \begin{equation} \label{eq-L}
\lim_{s\rightarrow+\infty}\tau_{s_0}^-(s)=\lim_{s\rightarrow+\infty}\tau_{s_0}^+(s)=L,
 \end{equation}
 where $L$ is the limit constant (cf.~Definition~{\rm \ref{def-limitconstant&limitangle}}).
\end{proposition}
\begin{proof}
From the hypothesis,  $\tau_{s_0}^-$ satisfies~\eqref{eq-ODEtau} and $\tau_{s_0}^-(s_0)=-1.$
Hence,
\[
(\tau_{s_0}^-)'(s_0)=(n-r)\alpha(s_0)>0,
\]
which implies that $\tau_{s_0}^-$ is strictly increasing near $s_0.$
It is clear  from~\eqref{eq-ODEtau} that $\tau_{s_0}^-$
will be strictly increasing as long as it stays negative. So, we have two possibilities, $\tau_{s_0}^-$ vanishes at some point
$s_1>s_0,$ or it is defined in $[s_0,+\infty),$ being negative and strictly
increasing in this interval. Assuming the latter,
we have that the graph of $\tau_{s_0}$ has a horizontal asymptotic line, so that
\[
\lim_{s\to+\infty}(\tau_{s_0}^-)'(s)=0.
\]
However, from~\eqref{eq-ODEtau} and the properties of $\alpha,$ we also have
\[
0<\lim_{s\rightarrow+\infty}(\tau_{s_0}^-)'(s)<+\infty,
\]
which is clearly a contradiction. Hence, $\tau_{s_0}^-(s_1)=0$ for some $s_1>s_0.$

Considering~\eqref{eq-ODEtau}, we see that
$(\tau_{s_0}^-)'(s)>0$ for all  $s$ such $\tau_{s_0}^-(s)=0$, from which we
conclude that  $s_1$ is the only zero of $\tau_{s_0}^-.$
Moreover, since $F(s,1)<0$ for all $s>0$ ($F$ as in~\eqref{eq-generalF}),
we have that there is no $s\in(s_0,+\infty)$ such that
$\tau_{s_0}^-(s)=1,$ which implies that $\tau_{s_0}^-$ is defined in $[s_0,+\infty).$

We claim that $\tau_{s_0}^-$ has at most one critical point.
Indeed, by the above considerations, any critical point
$s_*$ of $\tau_{s_0}^-$ is necessarily larger than $s_1.$ In particular,
$\tau_{s_0}^-(s_*)>0.$

Let us consider first the case $\alpha=1.$ Then, we have
\[
F(s,y)=C\sqrt{1-y^{2/r}}-(n-r)y,
\]
which implies that the constant function $\tau_L(s)=L$ is a solution to~\eqref{eq-generalIVP}
satisfying $y(s_0)=L,$ where $L$ is the limit constant.
If there is $s_*>s_1$ such that $(\tau_{s_0}^-)'(s_*)=0$, then
$F(s_*,\tau_{s_0}^-(s_*))=0$, which yields $\tau_{s_0}^-(s_*)=L.$
Thus, by uniqueness of solutions, we have that $\tau_{s_0}^-$ coincides with
the constant function $\tau_L$ on $[s_0,+\infty),$ which is clearly an absurd.
Therefore, $\tau_{s_0}^-$ has no critical points if $\alpha=1.$ In particular,
$(\tau_{s_0}^-)'>0$ in $[s_0,+\infty).$

Now, let $\alpha$ be either $\cot_\epsilon$ or $\tanh$,  and assume that
$s_*>s_1$ is a critical point of $\tau_{s_0}^-$.
In this setting, we have from~\eqref{eq-ODEtau} that
\begin{equation} \label{eq-secondderivative}
(\tau_{s_0}^-)''(s_*)=-\alpha'(s_*)[C(r-1)(\alpha(s_*))^{-r}\sqrt{1-\tau^{2/r}(s_*)}+(n-r)\tau(s_*)].
\end{equation}

Since the derivatives of $\cot_\epsilon$ and $\tanh$ are negative and positive in $(0,+\infty),$ respectively,
we have from~\eqref{eq-secondderivative} that $(\tau_{s_0}^-)''(s_*)$ is positive if $\alpha=\cot_\epsilon$  (so that $s_*$ is a local minimum)
and negative if $\alpha=\tanh$ (so that $s_*$ is a local maximum). In any of these cases, $s_*$ is the only
possible critical point of $(\tau_{s_0}^-),$ proving our claim. However,
$\tau_{s_0}^-$ is strictly increasing in a neighborhood of $s_1,$ so that its smaller critical point
could not be a local minimum. Hence, for $\alpha=\cot_\epsilon$, $\tau_{s_0}^-$ has no critical points, that is,
$(\tau_{s_0}^-)'>0$ in $[s_0,+\infty).$

It follows from the above considerations that, for any $s_0>0$,
\begin{equation} \label{eq-Ls0}
L_{s_0}^-:=\lim_{s\rightarrow+\infty}\tau_{s_0}^-(s)
\end{equation}
is well defined and satisfies $0<L_{s_0}^-\le 1$,
which implies that $(\tau_{s_0}^-)'(s)\rightarrow 0$ as $s\rightarrow+\infty.$
Therefore,
\begin{equation} \label{eq-limgeneralF}
\lim_{s\rightarrow+\infty}F(s,\tau_{s_0}^-(s))=0.
\end{equation}

We have that,  as $s\to+\infty$,
$\alpha(s)\to-\epsilon$ (resp.~$\alpha(s)\to 1$) if $\alpha=\cot_\epsilon$
(resp.~$\alpha=1$ or $\alpha=\tanh$).
In any of these cases, it follows from~\eqref{eq-limgeneralF} that
\[
C\sqrt{1-(L_{s_0}^-)^{2/r}}+\epsilon(n-r)L_{s_0}^-=0,
\]
which implies that $L_{s_0}^-=L.$

Regarding $\tau_{s_0}^+,$  we have from~\eqref{eq-ODEtau} that
$(\tau_{s_0}^+)'(s_0)=-(n-r)\alpha(s_0)<0,$
giving that  $\tau_{s_0}^+$ is decreasing  near $s_0.$
Also, as we have seen, the graph of a solution to~\eqref{eq-generalIVP} cannot cross the $s$-axis
from the positive side, so that $\tau_{s_0}^+$ is positive.
From this point, reasoning as in the preceding paragraphs, one concludes that
 $\tau_{s_0}^+$  is defined in $[s_0,+\infty)$ and has at most
 one critical point in this interval, which is necessarily a minimum.
 Therefore, the number
 \begin{equation}
L_{s_0}^+:=\lim_{s\rightarrow+\infty}\tau_{s_0}^+(s)
\end{equation}
is well defined and satisfies $0<L_{s_0}^+\le 1.$
As it was for $L_{s_0}^-,$ this last equality yields
$L_{s_0}^+=L.$ This finishes the proof.
\end{proof}

\section{Rotational Translators to $r(<\hspace{-.08cm}n)$-MCF} \label{sec-rotationaltranslators}
This section concerns
$r(<\hspace{-.11cm}n)$-translators in $\qr$ which are invariant by rotations.
More precisely, such a translator
will be considered as an $(M_s,\phi)$-graph $\Sigma$, where
$\{M_s=S_s^{n-1}\,;\, s\in I\subset(0,+\infty)\}$ is a family
of concentric geodesic spheres of $\q_\epsilon^n$ centered
at a point $o\in\q_\epsilon^n.$

\begin{figure}[h]
 \centering
 \includegraphics[scale=.29]{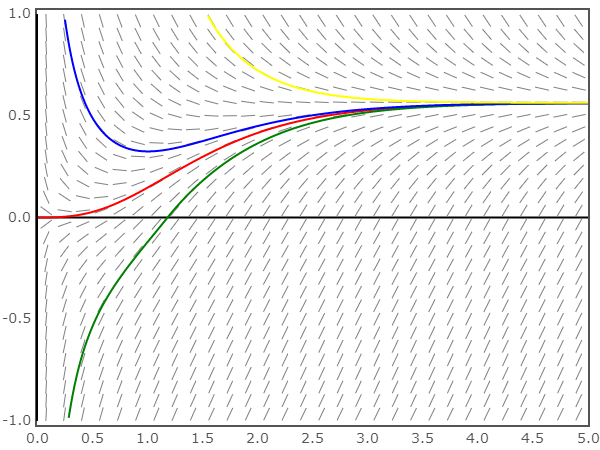}
 \hfill
 \includegraphics[scale=.29]{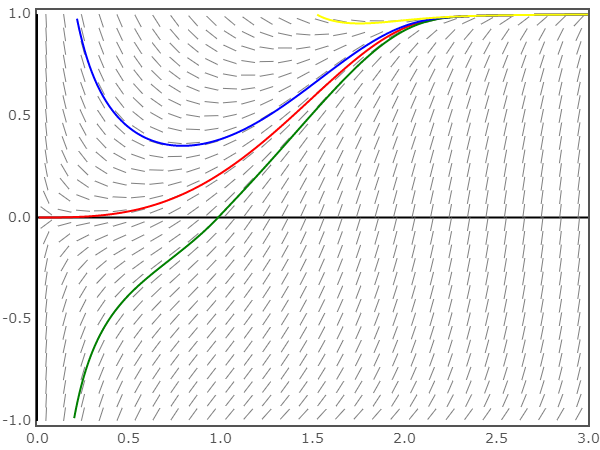}
 \caption{\small Orbits of the slope field of $F_{(\epsilon, n,r)}$ for $(\epsilon, n,r)=(-1,4,3)$ (left),
 and $(\epsilon, n,r)=(0,4,3)$ (right).}
 \label{fig-plot}
\end{figure}

In this setting, with the notation of the preceding section, we
have $\alpha=\cot_\epsilon,$ so that equation~\eqref{eq-ODEtau} becomes
\begin{equation} \label{eq-MCFode2}
\tau'(s)=C\sqrt{1-\tau^{2/r}(s)}\tan_\epsilon^{r-1}(s)-(n-r)\cot_\epsilon(s)\tau(s),
\end{equation}
and the Cauchy problem~\eqref{eq-generalIVP} becomes
\begin{equation} \label{eq-IVProtational}
\left\{
\begin{array}{l}
y'(s)=F(s,y(s))\\[1ex]
y(s_0)=y_0,
\end{array}
\right.
\end{equation}
where $(s_0,y_0)\in\Omega:=(0,+\infty)\times[-1,1],$  and $F=F_{(\epsilon,n,r)}$ is the function:
\begin{equation} \label{eq-F}
F(s,y):=C\sqrt{1-y^{2/r}}\tan_\epsilon^{r-1}(s)-(n-r)\cot_\epsilon(s)y, \,\,\, 1\le r<n,  \,\,\, (s,y)\in\Omega.
\end{equation}

In the next propositions we establish that, besides the solutions $\tau_{s_0}^{\pm}$ defined
in Proposition~\ref{prop-qualitative}, the Cauchy problem~\eqref{eq-IVProtational} has a solution
$\tau_0$ defined in $[0,+\infty)$ which satisfies $\tau_0(0)=0.$ Moreover, the functions
$\tau_0$ and $\tau_{s_0}^{\pm}$ are the only solutions to~\eqref{eq-generalIVP} defined in a maximal
interval (Fig.~\ref{fig-plot}).

\begin{proposition} \label{prop-solutionatzero}
Given $s_0>0,$  let $\tau_{s_0}^-,$  $\tau_{s_0}^+,$ and $L\in(0,1]$ be as in
Proposition~{\rm \ref{prop-qualitative}.}
Then, there exists a solution $\tau_0:[0,+\infty)\rightarrow[0,L)$
of~\eqref{eq-IVProtational} such that:
\begin{itemize}[parsep=1ex]
  \item [\rm i)] $\tau_0(0)=0.$
  \item [\rm ii)] $\tau_0$ and $\tau_0'$ are both positive in $(0,+\infty).$
  \item [\rm iii)] $\displaystyle \lim_{s\rightarrow+\infty}\tau_0(s)=L.$
  \item [\rm iv)] For any $s_0>0,$ the inequalities $\tau_{s_0}^-<\tau_0<\tau_{s_0}^+$ hold on $[s_0,+\infty).$
\end{itemize}
\end{proposition}
\begin{proof}
Given $s_1>0$, let $\tau_{s_1}$ be the solution to~\eqref{eq-IVProtational}
satisfying the initial condition $\tau_{s_1}(s_1)=0.$ Proceeding as in the
proof of Proposition~\ref{prop-qualitative}, we conclude that $\tau_{s_1}$ is defined
in $[s_1,+\infty)$, is strictly increasing, and satisfies
\[
\lim_{s\rightarrow+\infty}\tau_{s_1}(s)=L.
\]

Now, define the function $\psi_{s_1}$ on $(0,+\infty)$ by
\[
\psi_{s_1}(s):=\tau_{s_1}(s_1+s).
\]
For any $s_1>0,$  $\psi_{s_1}$ is strictly increasing with lowest upper bound $L$.
Hence, as $s_1\to 0$, the functions $\psi_{s_1}$ converge uniformly to
the function $\tau_0\colon(0,+\infty)\to (0,L)$ given by
\[
\tau_0(s):=\lim_{s_1\to 0}\psi_{s_1}(s).
\]
In addition, one has
\[
\psi_{s_1}'(s)=\tau_{s_1}'(s_1+s)=F(s_1+s,\tau_{s_1}(s_1+s))=F(s_1+s,\psi_{s_1}(s)),
\]
which implies that, as $s_1\to 0,$ the derivatives $\psi_{s_1}'$ converge uniformly to the function
$s\mapsto F(s,\tau_0(s))$ in any compact interval  $[a,b]\subset(0,+\infty).$ Therefore,
$\tau_0$ is differentiable in $(0,+\infty)$ and satisfies (cf.~\cite[Theorem 4.7.8]{field})
\[
\tau_0'(s)=\lim_{s_1\to 0}\psi_{s_1}'(s)=\lim_{s_1\to 0}F(s_1+s,\psi_{s_1}(s))=F(s,\tau_0(s)),
\]
so that  $\tau_0$ is a solution to~\eqref{eq-IVProtational}
on $(0,+\infty).$ Besides, it is easily checked that
$\tau_0$ extends smoothly to $s=0$ and satisfies
\[
\tau_0(0)=0 \quad\text{and}\quad \tau_0'(0)=\left\{
\begin{array}{ccl}
1 & \text{if} & r=1,\\
0 & \text{if} & 1<r\le n-1.
\end{array}
\right.
\]

The proofs of (ii) and (iii) are analogous to the ones given in Proposition~\ref{prop-qualitative} for
the functions $\tau_{s_0}^-.$
Finally, denoting by $\mathcal G_0, \mathcal G_{s_0}^-,$ and $\mathcal G_{s_0}^+$
the graphs  of $\tau_0, \, \tau_{s_0}^-,$ and $\tau_{s_0}^+,$ respectively, we have that
$\mathcal G_0$ separates $\Omega$ into two connected components,
one below $\mathcal G_0,$ say $\Omega^-,$ and one above $\mathcal G_0,$ $\Omega^+.$
By the uniqueness of solutions to~\eqref{eq-IVProtational} with given initial conditions,
the graphs of two distinct solutions never intersect. Hence,
for any $s_0>0,$ one has $\mathcal G_{s_0}^-\subset\Omega^-$ and
$\mathcal G_{s_0}^+\subset\Omega^+.$ This clearly implies (iv) and finishes
our proof.
\end{proof}

\begin{proposition} \label{prop-noblow}
Let $\tau_0$ be  as in Proposition~{\rm\ref{prop-solutionatzero}}.
Then, for $\rho_0=\tau_0^{1/r},$ one has that the limits
\[
L_1:=\lim_{s\rightarrow 0}(\cot_\epsilon(s)\rho_0(s))
\quad\text{and}\quad L_2:=\lim_{s\rightarrow 0}\rho_0'(s)
\]
are both finite.
\end{proposition}

\begin{proof}
Since $\tau_0$ is a solution to~\eqref{eq-MCFode2}, we have that~\eqref{eq-rMCFodegeneral} holds
for $\alpha=\cot_\epsilon$ and  $\rho=\rho_0.$ Hence,
$L_1$ and $L_2$ cannot be both infinite, for
$\sqrt{1-\rho_0^2(s)}\rightarrow 1$ as $s\rightarrow 0.$ In addition,
\[
\lim_{s\rightarrow 0}(\cot_\epsilon(s)\rho_0(s))=\lim_{s\rightarrow 0}\left(\frac{\rho_0(s)}{\tan_\epsilon(s)}\right)
=\lim_{s\rightarrow 0}(\rho_0'(s)\cos_\epsilon^2(s)),
\]
which clearly implies that $L_1$ and $L_2$ are both finite.
\end{proof}

\begin{figure}
 \centering
 \includegraphics[scale=1.4]{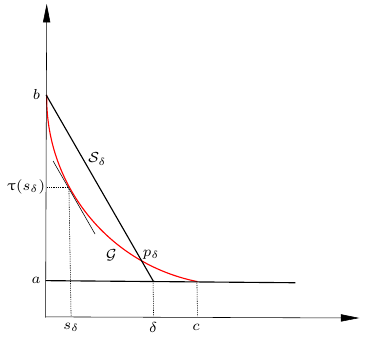}
 \caption{\small Graph $\mathcal G$ of the function $\tau$ considered in the proof of Proposition~\ref{prop-uniquenesssolutions}.}
 \label{fig-impossiblegraph}
\end{figure}

\begin{proposition} \label{prop-uniquenesssolutions}
The only solutions to the Cauchy problem~\eqref{eq-IVProtational}
which are defined in a maximal interval
are the functions $\tau_{s_0}^{\pm}$ of Proposition~{\rm\ref{prop-qualitative}}, and the
function $\tau_0$ of Proposition~{\rm\ref{prop-solutionatzero}}.
\end{proposition}

\begin{proof}
It suffices to prove that there is no solution to~\eqref{eq-IVProtational}
whose graph has an endpoint
of the form $p:=(0,b)$ with $b\ne 0.$ Assume, by contradiction,
that such a solution exists and call it $\tau$. Assuming also that
$b>0,$ we have that $F(s,b)\to-\infty$ as $s\to 0$. Then,  if we extend $\tau$ to $0$
by making $\tau(0)=b$, the graph $\mathcal G$ of $\tau$ is tangent
to the $y$-axis at $p$ (Fig.~\ref{fig-impossiblegraph}).

Now, choose a small  $c>0$ such that $\tau'<0$ on $(0,c)$,
and set $a:=\tau(c)>0.$
Given a positive $\delta<c$, let $\mathcal S_\delta$ be the line segment from $p=(0,b)$ to $(\delta,a).$
It is clear that $\mathcal S_\delta$  intersects $\mathcal G$ at  a single point $p_\delta.$
Then, by Rolle's Theorem,  there exists a point $q_\delta:=(s_\delta,\tau(s_\delta))$ in the
open arc of $\mathcal G$ from $p$ to $p_\delta$ such that the tangent line to
$\mathcal G$ at $q_\delta$ is parallel to
$\mathcal S_\delta$. In particular, $\tau'(s_\delta)=-(b-a)/\delta$.
Thus, by~\eqref{eq-MCFode2},
\[
-C\sqrt{1-\tau^{2/r}(s_\delta)}\tan_\epsilon^{r-1}(s_\delta)+(n-r)\cot_\epsilon(s_\delta)\tau(s_\delta)=\frac{b-a}{\delta}\cdot
\]

Since $0<s_\delta<\delta$, we have that $\delta\cot_\epsilon(s_\delta)>s_\delta\cot_\epsilon(s_\delta)\ge 1.$
This, together with the last equality above, yields
\begin{equation} \label{eq-b-a}
b-a>-\delta C\sqrt{1-\tau^{2/r}(s_\delta)}\tan_\epsilon^{r-1}(s_\delta)+(n-r)\tau(s_\delta).
\end{equation}
Letting $\delta\to0$ on both sides of~\eqref{eq-b-a} gives $b-a\ge (n-r)b\ge b$, which is
a contradiction. In the same way, we derive a contradiction if we assume $b<0.$ Therefore, except for the
function $\tau_0$ of Proposition~\ref{prop-solutionatzero}, no graph of
a solution to~\eqref{eq-IVProtational} has an endpoint at the $y$-axis, as we wished to prove.
\end{proof}

\begin{figure}[htb]
 \centering
  \includegraphics[width=5.5cm]{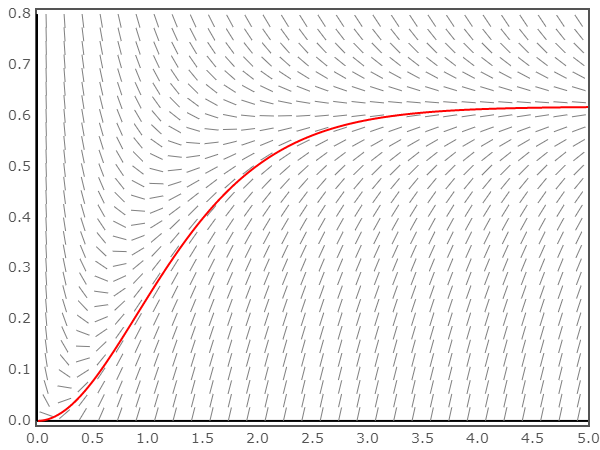}
  \hfill
 \includegraphics[width=6.1cm]{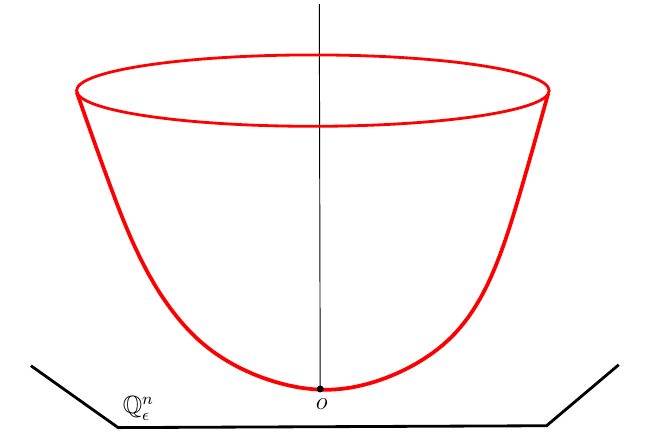}
 \caption{\small The graph of $\tau_0$  (left) and the
 $r$-bowl soliton obtained from it (right).}
 \label{fig-bowlsoliton}
\end{figure}

Now, we are in position to state and prove our first main result.

\begin{theorem} \label{th-existence}
Given integers $n\ge 2$ and $r\in\{1,\dots, n-1\}$, the following  hold:
\begin{itemize}[parsep=1ex]
  \item[\rm i)] There exists a rotational strictly convex  $r$-translator $\Sigma_0$
  in $\qr$ (to be called the $r$-\emph{bowl soliton}) which is an entire vertical graph
  contained in the closed half-space $\q_\epsilon^n\times[0,+\infty)$ with unbounded height
  (Fig.~{\rm\ref{fig-bowlsoliton}}).

  \item[\rm ii)] If \,$r$ is odd, there exists a one-parameter family
  $\mathscr C_r=\{\Sigma_\lambda\,;\, \lambda>0\}$ of
  properly embedded annular rotational $r$-translators  in $\qr$
  (to be called $r$-\emph{translating catenoids}) with the following properties
  (Fig.~{\rm\ref{fig-wing}}):
  %%%%%%%%%%%%%%%%%%%%
  \subitem $\bullet$ For each $\lambda>0,$  $\Sigma_\lambda$ is the union of two graphs  $\Sigma_\lambda^-$ and
  $\Sigma_\lambda^+$ over the complement of the ball $B_\lambda(o)\subset\q_\epsilon^n$ which have unbounded height
  and satisfy $\partial\Sigma_\lambda^{\pm}=\partial B_\lambda(o).$
  %%%%%%%%%%%%%%%%%%%%
  \subitem $\bullet$ Each $r$-translating catenoid $\Sigma_\lambda\in\mathscr C_r$ is contained in a half-space  of \,$\qr,$ and its set of
  points of minimal height is an $(n-1)$-sphere centered at the axis of rotation which is contained in
  a horizontal hyperplane $\Pi_t:=\mathbb Q_\epsilon^n\times\{t\}$, $t<0$.
  %%%%%%%%%%%%%%%%%%%%
  \subitem $\bullet$ For $r>1,$ any  $r$-translating catenoid $\Sigma_\lambda\in\mathscr C_r$
  is $C^2$-singular along its $(n-1)$-sphere of minimal height.
  %%%%%%%%%%%%%%%%%%%%
  \subitem $\bullet$ For any $\lambda>0,$ the graphs $\Sigma_\lambda^-$ and
  $\Sigma_\lambda^+$ have the same asymptotic behavior of
  the $r$-bowl soliton $\Sigma_0.$ More precisely, the angle functions
  $\theta^-,$ $\theta^+,$ and\, $\theta_0,$ of\,  $\Sigma_\lambda^-,$
  $\Sigma_\lambda^+$ and  $\Sigma_0,$ respectively, satisfy:
  \begin{equation} \label{eq-thetas}
  \lim_{s\rightarrow+\infty}\theta^-(s)=\lim_{s\rightarrow+\infty}\theta^+(s)=\lim_{s\rightarrow+\infty}\theta_0(s)=\theta_L,
  \end{equation}
  where $\theta_L$ is the limit angle (cf.~Definition~{\rm\ref{def-limitconstant&limitangle}}).
  %%%%%%%%%%%%%%%%%%%%%%%%%%%%%%%%%%%%%%%%%%%%%%%
  %%%%%%%%%%%%%%%%%%%%%%%%%%%%%%%%%%%%%%%%%%%%%%%
  \item[\rm iii)] If \,$r$ is even, there  are two one-parameter families
  $\mathscr C_r^i=\{\Sigma_\lambda^i\,;\, \lambda>0\},$ $i=1,2,$ of
  properly embedded annular rotational $r$-translators  in $\qr$
  (to be called $r$-\emph{translating catenoids}) with nonempty boundary.
  In addition, one has that (Fig.~{\rm\ref{fig-translatingcatenoid}}):
  %%%%%%%%%%%%%%%%%%%%
  \subitem $\bullet$ For each $\lambda>0,$  $\Sigma_\lambda^i$ is an unbounded graph
  in the half-space  $\q_\epsilon^n\times[0,+\infty)$
  on the complement of a ball $B\subset\q_\epsilon^n\times\{0\}$ centered at the rotation axis
  and of radius $R=R(\lambda)>0.$
  %%%%%%%%%%%%%%%%%%%%
  \subitem $\bullet$ Along their boundaries, the $r$-translators in $\mathscr C_r^1$ are tangent to
  the horizontal hyperplane $\Pi_0$, whereas those in $\mathscr C_r^2$ are orthogonal to $\Pi_0$.
  %%%%%%%%%%%%%%%%%%%%
  \subitem $\bullet$ For any $\lambda>0,$  the angle functions
  $\theta_\lambda^i$  and\, $\theta_0,$  of the graphs $\Sigma_\lambda^i$
  and the $r$-bowl soliton  $\Sigma_0,$ respectively, satisfy:
  \begin{equation} \label{eq-thetaseven}
  \lim_{s\rightarrow+\infty}\theta_\lambda^i(s)=\lim_{s\rightarrow+\infty}\theta_0(s)=\theta_L.
  \end{equation}
\end{itemize}
\end{theorem}

\begin{proof}
Let $\rho_0$ be as in Proposition~\ref{prop-noblow}. Then, by
Proposition~\ref{prop-main}, the rotational
entire graph $\Sigma_0$ with $\rho$-function $\rho_0$ and height function
\[
\phi_0(s)=\int_{0}^{s}\frac{\rho_0(u)}{\sqrt{1-\rho_0^2(u)}}du, \,\,\, s\in[0,+\infty),
\]
is an $r$-translator in $\qr$ (Fig.~\ref{fig-bowlsoliton}).
Setting $\{o\}\times\R,$ $o\in\q_\epsilon^n$, for the axis of rotation of
$\Sigma_0,$ we have from~\eqref{eq-principalcurvatures} (for $k_i^s=-\cot_\epsilon(s)$)
and Proposition~\ref{prop-noblow} that  the principal curvatures
of  $\Sigma_0$ at $o$ are well defined, so that $\Sigma_0$ is
$C^2$ at $o.$ Since $0=\phi_0(0)<\phi_0(s)$ for all $s>0,$ we also have that
$\Sigma_0$ is contained in the half-space $\q_\epsilon^n\times [0,+\infty),$
and is tangent to $\q_\epsilon^n\times\{0\}$ at $o.$ In particular,
$\Sigma_0$ is strictly convex at $o.$ In addition, its height function $\phi_0$ is unbounded.
Indeed, since $\tau_0,$ and so $\rho_0,$ is increasing in $(0,+\infty),$ for any $a>0,$ one has
\[
\phi_0(s)>\int_{a}^{s}\frac{\rho_0(u)}{({1-(\rho_0(u))^2})^{1/2}}du\ge
\int_{a}^{s}\rho_0(u)du\ge \rho_0(a)(s-a),
\]
which clearly implies that $\phi_0$ is unbounded. Finally,
it follows from~\eqref{eq-principalcurvatures} (for $k_i^s=-\cot_\epsilon(s)$)
and Proposition~\ref{prop-solutionatzero} that all principal
curvatures $k_i(s)$  of $\Sigma_0$ are positive for $s>0,$ which gives
that $\Sigma_0$ is strictly convex. This proves (i).

To prove (ii), set $s_0=\lambda>0$ and let $\tau_\lambda^-$
and $\tau_\lambda^+$ be as in Proposition~\ref{prop-qualitative}.
Denote by $\Sigma_\lambda^-$ and  $\Sigma_\lambda^+$  the rotational graphs whose $\rho$-functions
are $\rho_\lambda^-=(\tau_\lambda^-)^{1/r}$ and $\rho_\lambda^+=(\tau_\lambda^+)^{1/r},$ and
whose height functions are
\[
\phi_\lambda^-(s)=\int_{\lambda}^{s}\frac{\rho_\lambda^-(u)}{({1-(\rho_\lambda^-(u))^2})^{1/2}}du \quad\text{and}\quad
\phi_\lambda^+(s)=\int_{\lambda}^{s}\frac{\rho_\lambda^+(u)}{({1-(\rho_\lambda^+(u))^2})^{1/2}}du,
\]
respectively.

Assume $r$ odd. By Proposition~\ref{prop-qualitative}-(i), there exists a unique $s(\lambda)>0$
at which $\rho_\lambda^-$ vanishes, so that
$\rho_\lambda^-$ is negative in the interval $(\lambda,s(\lambda)),$  and positive in $(s(\lambda),+\infty).$
Then, $\phi_\lambda^-(s)$ is decreasing in $(\lambda,s(\lambda)),$
and increasing in $(s(\lambda),+\infty).$ Furthermore, the tangent spaces of
the closure of $\Sigma_\lambda^-$ in $\qr$ along its boundary are all well defined and vertical, for
$\rho_\lambda^-(\lambda)=-1.$

Since
$\tau_{\lambda}^-(s(\lambda))=0<(\tau_{\lambda}^-)'(s(\lambda))$ and
$$(\rho_{\lambda}^-)'(s)=\frac{1}{r}(\tau_{\lambda}^-(s))^{\frac{1-r}{r}}(\tau_{\lambda}^-)'(s),$$
 if $r>1,$ one has   $(\rho_{\lambda}^-)'(s)\rightarrow +\infty$ as $s\rightarrow s(\lambda).$
From this and~\eqref{eq-principalcurvatures},
we conclude that, for $r>1,$  the second fundamental form of $\Sigma_\lambda^-$ blows up at all points
of its $(n-1)$-sphere of (minimal) height $\phi(s(\lambda)),$ i.e., $\Sigma_\lambda^-$ is
$C^2$-singular at these points.

\begin{figure}[hbt]
 \centering
  \includegraphics[width=5.5cm]{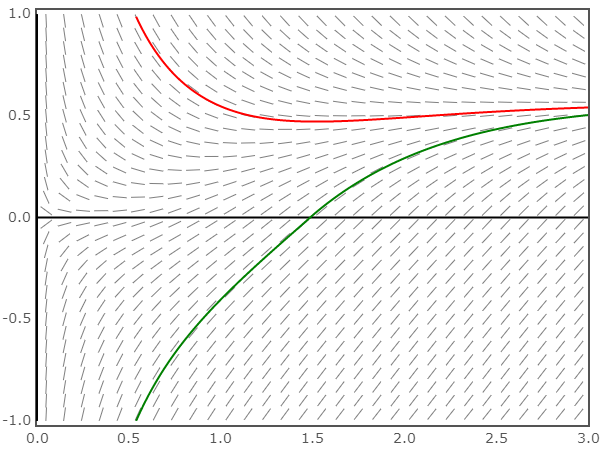}
  \hfill
 \includegraphics[width=6.1cm]{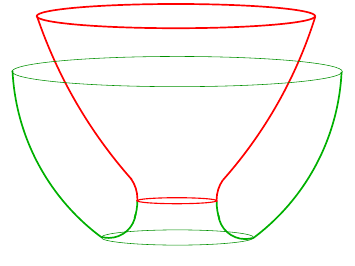}
 \caption{\small The graphs of $\tau_\lambda^-$ and $\tau_\lambda^+$ (left) and the
 $r$(odd)-translating catenoid $\Sigma_\lambda$ obtained from them (right). For $r>1,$ $\Sigma_\lambda$ is $C^2$-singular
 on the horizontal $(n-1)$-sphere of minimal height.}
 \label{fig-wing}
\end{figure}

Since $\tau_\lambda^+$, and so $\rho_\lambda^+$, is positive in $[\lambda,+\infty),$
the same is true for $\phi_\lambda^+.$ Besides, analogously to $\Sigma_\lambda^-$, the tangent spaces of
the closure of $\Sigma_\lambda^+$ in $\qr$ along its boundary are all well defined and vertical.
However, the boundaries of $\Sigma_\lambda^-$ and $\Sigma_\lambda^+$ coincide with $\partial B_\lambda(o)$,
and so we have that
$$\Sigma_\lambda:={\rm closure}\,(\Sigma_\lambda^-)\cup{\rm closure}\,(\Sigma_\lambda^+)$$ is
an $r$-translator (Fig.~\ref{fig-wing}). Furthermore, the argument we gave to prove
that $\phi_0$ is unbounded apply
to  $\phi_\lambda^-$ and $\phi_\lambda^+,$ so that these functions are both unbounded
as well. We also point out that $\Sigma_\lambda$ is $C^2$-smooth on the common  boundary
$\partial\Sigma_\lambda^\pm$ of $\Sigma_\lambda^\pm.$ To see this, first observe that
the principal curvatures on $\partial\Sigma_\lambda^-$ and
$\partial\Sigma_\lambda^+$ are well defined. Indeed, recall that
$\rho_\lambda^-(\lambda)=-\rho_\lambda^+(\lambda)=-1$ and that $r$ is odd. So, on
$\partial\Sigma_\lambda^-$, the principal curvatures $k_i,\dots, k_n$ are given by
$$k_i=\alpha(\lambda)\rho_\lambda^-(\lambda)=
-\alpha(\lambda)=-\cot_\epsilon(\lambda)<0, \,\,\, i=1,\dots, n-1,$$ and
$$k_n=(\rho_\lambda^-)'(\lambda)=
\frac{1}{r}(\tau_\lambda^-(\lambda))^{\frac{1-r}{r}}(\tau_\lambda^-)'(\lambda)=
\frac{1}{r}(\tau_\lambda^-)'(\lambda)>0,$$
whereas on  $\partial\Sigma_\lambda^+$
$$k_i=\alpha(\lambda)\rho_\lambda^+(\lambda)=\alpha(\lambda)=\cot_\epsilon(\lambda)>0,$$
and
$$k_n=(\rho_\lambda^+)'(\lambda)=\frac{1}{r}(\tau_\lambda^+(\lambda))^{\frac{1-r}{r}}(\tau_\lambda^+)'(\lambda)=\frac{1}{r}(\tau_\lambda^+)'(\lambda)<0.$$
Moreover, since $\tau_\lambda^+$ and $\tau_\lambda^-$ are the solutions to~\eqref{eq-IVProtational}
satisfying $y(\lambda)=1$ and $y(\lambda)=-1,$ respectively, we have that
$(\tau_\lambda^+)'(\lambda)=-(\tau_\lambda^-)'(\lambda).$ Hence, after a change of orientation of either
$\Sigma_\lambda^-$ or $\Sigma_\lambda^+$ (see Remark~\ref{rem-r-meancurvaturevector}),
we conclude from the above equalities that $\Sigma_\lambda$ is $C^2$-smooth on
$\partial\Sigma_\lambda^\pm.$

Now, by Proposition \ref{prop-solutionatzero}-(iv), the inequalities
$\rho_\lambda^-(s)<\rho_0(s)<\rho_\lambda^+(s)$
hold for all $s\in[\lambda,+\infty).$ Thus,
$$\phi_\lambda^-(s)<\phi_0(s)<\phi_\lambda^+(s) \,\,\, \forall s\in[\lambda,+\infty).$$
In particular, $\Sigma_\lambda$ is properly embedded.

To conclude the proof of (ii), we observe that
equality~\eqref{eq-thetas} follows directly from the relation
$\theta^2=1-\rho^2,$  equality~\eqref{eq-L}, and Proposition~\ref{prop-solutionatzero}-(iii).

To prove (iii), assume that $r$ is even. In this case, keeping the notation above,
we have to disregard the negative part of $\tau_\lambda^-,$ since
$(\rho_\lambda^-)^r\ge0.$ Then, we have
\begin{equation} \label{eq-rho007}
\rho_\lambda^-=(\hat\tau_\lambda^-)^{1/r}\quad\text{and}\quad \rho_\lambda^+=(\tau_\lambda^+)^{1/r},
\end{equation}
where $\hat\tau_\lambda^-:=\tau_\lambda^-|_{[s(\lambda),+\infty)}\ge 0$ and $s(\lambda)>0$ satisfies
$\tau_\lambda^-(s(\lambda))=0.$

Now, denote by $\Sigma_\lambda^1$ and  $\Sigma_\lambda^2$  the rotational graphs whose $\rho$-functions
are $\rho_\lambda^-$ and $\rho_\lambda^+,$  and
whose height functions are
\begin{equation} \label{eq-phi007}
\phi_\lambda^1(s)=\int_{s(\lambda)}^{s}\frac{\rho_\lambda^-(u)}{({1-(\rho_\lambda^-(u))^2})^{1/2}}du \quad\text{and}\quad
\phi_\lambda^2(s)=\int_{\lambda}^{s}\frac{\rho_\lambda^+(u)}{({1-(\rho_\lambda^+(u))^2})^{1/2}}du,
\end{equation}
respectively. Then, proceeding as before, we conclude that each element of the family
$\mathscr C_r^i:=\{\Sigma_\lambda^i\,, \lambda>0\},$ $i=1,2,$ is an unbounded graph in $\q_\epsilon^n\times[0,+\infty)$
as asserted. Also, since $\rho_\lambda^-(s(\lambda))=0$ and $\rho_\lambda^+(\lambda)=1,$
we have that the graphs $\Sigma_\lambda^1$ are tangent to $\q_\epsilon^n\times\{0\}$ along their boundaries,
whereas the graphs $\Sigma_\lambda^2$ are orthogonal to $\q_\epsilon^n\times\{0\}$  (Fig.~\ref{fig-translatingcatenoid}).

Finally, as it was for~\eqref{eq-thetas}, equality~\eqref{eq-thetaseven}
follows from the relation $\theta^2=1-\rho^2,$  equality~\eqref{eq-L}, and Proposition~\ref{prop-solutionatzero}-(iv).
This shows (iii) and finishes our proof.
\end{proof}

\begin{figure}
 \centering
  \includegraphics[width=5cm]{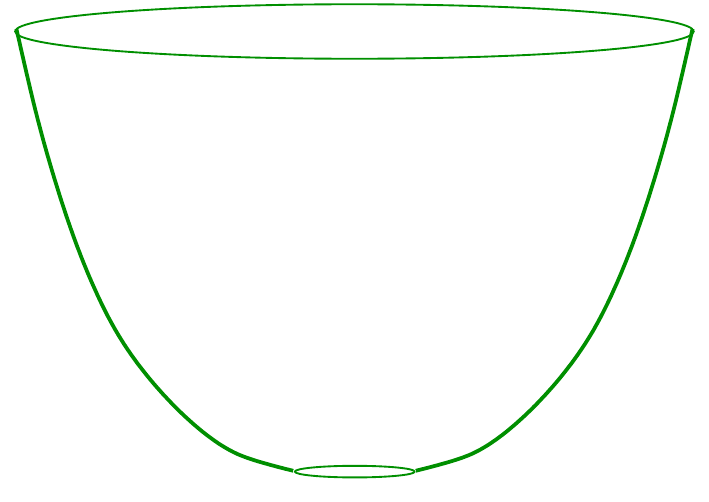}
  \hspace{.6cm}
 \includegraphics[width=5cm]{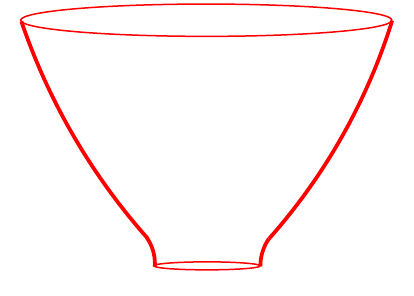}
 \caption{\small $r$(even)-translating catenoids with boundary, where the one on the left
 belongs to $\mathscr C_r^1$, whereas the one on the right belongs to $\mathscr C_r^2$}.
 \label{fig-translatingcatenoid}
\end{figure}

\begin{remark} \label{rem-catenoidsmove}
Regarding  Theorem \ref{th-existence}-(ii), the $r(>\hspace{-.1cm}1)$-th mean curvature $H_r$ of any translating catenoid
$\Sigma_\lambda\in\mathscr C_r$ extends $C^1$-smoothly to
the $C^2$-singular $(n-1)$-sphere of minimal height, and equals $1$ there. Hence, despite the fact
that $r(>\hspace{-.1cm}1)$-translating catenoids have  $C^2$-singular sets when $r$ is odd, they move under
$H_r$-flow, that is, they are genuine $r$-translators.
The same goes for the other translators with
$C^2$-singular sets we shall obtain in the next sections.
\end{remark}

\begin{remark} \label{rem-catenoids}
In the above setting, it is not hard to prove that, for  any $s_0>0,$
the solution $\tau_{s_0}$ to~\eqref{eq-IVProtational} has a local minimum
$s_*=s_*(s_0)$, and that  $(s_*(s_0),\tau(s_*(s_0))$ converges to $(0,0)$ as
$s_0\to 0$ (see Fig.~\ref{fig-convergence}). As a consequence,
the restriction of $\tau_{s_0}^+$ to $(s_*(s_0),+\infty)$ converges uniformly to the solution
$\tau_0$ as $s_0\rightarrow 0.$ By the  definition of $\tau_0,$ the same is true
for the restriction of $\tau_{s_0}^-$ to the interval where it is positive.
Thus, writing $\lambda=s_0$, the subsets of the translating catenoids
$\Sigma_\lambda$ generated by these restrictions of $\tau_{s_0}^{\pm}$
converge (in compact sets) to the bowl soliton  $\Sigma_0$ as $\lambda\rightarrow 0$.
\end{remark}

\begin{figure}[hbt]
 \centering
  \includegraphics[scale=.3]{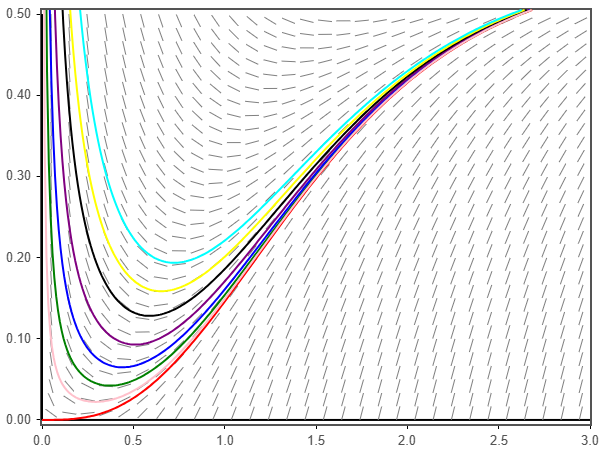}
 \caption{\small Graphs of solutions $\tau_{s_0}^+$ to~\eqref{eq-IVProtational}. As $s_0\to 0$, the points
 of minimal height converge to $(0,0)$.}
 \label{fig-convergence}
\end{figure}

\begin{remark}  \label{rem-congruentreven}
For $r$ even, $-\rho$ is a solution to \eqref{eq-rMCFodegeneral} whenever $\rho$ is a solution.
Hence, in \eqref{eq-rho007}, we could have chosen the negative functions
$\rho_\lambda^-=-(\hat\tau_\lambda^-)^{1/r}$ and $\rho_\lambda^+=-(\hat\tau_\lambda^-)^{1/r}.$
However, the corresponding graphs $\overbar\Sigma_\lambda^i$ obtained from the functions $\phi_\lambda^i$ as in
\eqref{eq-phi007} would be congruent to the ones in $\mathscr C_r^i.$ Indeed, as one can easily check,
$\overbar\Sigma_\lambda^i$ is nothing but the reflection of $\Sigma_\lambda^i$ about
the horizontal hyperplane $\q_\epsilon^n\times\{0\}.$ (Recall that,
as we pointed out in Remark \ref{rem-reflectedtranslator},
$\overbar\Sigma_\lambda^i$ is an $r$-translator when it is properly oriented.)
These considerations  apply to the $r$(even)-translators  with boundary we shall
obtain in the next sections.
\end{remark}

We close this section with the following uniqueness result.

\begin{proposition} \label{prop-uniquenessgraphrotational}
Let $\Sigma$ be a connected rotational $r(<\hspace{-.1cm}n)$-translator in $\qr$
which is a vertical graph over an open set of \,$\q_\epsilon^n.$
Then, $\Sigma$ is an open set of either an $r$-bowl soliton or an
$r$-translating catenoid.
\end{proposition}

\begin{proof}
Since $\Sigma$ is rotational, it constitutes an $(M_s,\phi)$-graph such that the parallels
$M_s$ are geodesic spheres of $\q_\epsilon^n.$ Hence, by Proposition~\ref{prop-main},
the $\rho$-function of $\Sigma$ satisfies~\eqref{eq-rMCFodegeneral}
for $\alpha=\cot_\epsilon$, which implies that $\tau:=\rho^r$ is a solution
to~\eqref{eq-IVProtational} defined in an interval $I\subset[0,+\infty).$
Thus, from Proposition~\ref{prop-uniquenesssolutions}, $\tau$ is the restriction to
$I$ of either the function $\tau_0$ defined in Proposition~{\rm\ref{prop-solutionatzero}}, or
one of the solutions $\tau_{s_0}^{\pm}$ defined in Proposition~{\rm\ref{prop-qualitative}}.
It follows then by (the proof of) Theorem~\ref{th-existence} that $\Sigma$ is an open set of either
the $r$-bowl soliton or an $r$-translating catenoid.
\end{proof}

\section{Parabolic Translators to $r(<\hspace{-.1cm}n)$-MCF in $\h^n\times\R$} \label{sec-parabolic}

We shall consider now parabolic $r(<\hspace{-.1cm}n)$-translators
in $\h^n\times\R$, that is, those which are invariant by horizontal
parabolic translations.
\begin{figure}[hbt]
 \centering
 \includegraphics[scale=.3]{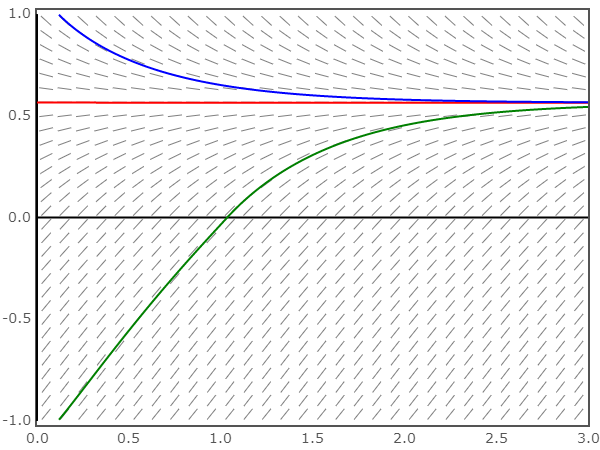}
 \caption{\small Graphs of solutions to the Cauchy problem~\eqref{eq-IVP-parabolic}.}
 \label{fig-parabolicplot1}
\end{figure}
More precisely, such a translator
will be obtained from  $(\mathcal H_s,\phi)$-graphs, where
$$\mathscr H:=\{\mathcal H_s\,;\, s\in I\subset(-\infty,+\infty)\}$$
is a family of horospheres of $\h^n$ centered
at a fixed point $p_\infty$ of the ideal boundary $\partial_\infty\h^n$ of
$\h^n.$

For the family $\mathscr H$, we have that $\alpha$ is the constant function
$1$ (notation as in Section~\ref{sec-translators}). So, equation~\eqref{eq-ODEtau} becomes
\begin{equation} \label{eq-parabolicMCFode2}
\tau'(s)=C\sqrt{1-\tau^{2/r}(s)}-(n-r)\tau(s),
\end{equation}
and the associated Cauchy problem is:
\begin{equation} \label{eq-IVP-parabolic}
\left\{
\begin{array}{l}
y'(s)=F(y(s))\\[1ex]
y(s_0)=y_0,
\end{array}
\right.
\end{equation}
where $(s_0,y_0)\in\Omega:=(-\infty,+\infty)\times[-1,1],$  and $F=F_{(n,r)}$ is the function:
\begin{equation} \label{eq-Fparabolic}
F(y):=C\sqrt{1-y^{2/r}}-(n-r)y, \,\,\, 1\le r<n,  \,\,\, y\in[-1,1].
\end{equation}

Figure~\ref{fig-parabolicplot1} shows the graphs of some solutions
to~\eqref{eq-IVP-parabolic}.
As we pointed out in the proof of
Proposition~\ref{prop-qualitative}, the constant function
$\tau_L=L$ is a solution, where
$L$ is the limit constant (red curve).
The blue and green curves are the graphs of  solutions of the type $\tau_{s_0}^{\pm}$,
described in Proposition~\ref{prop-qualitative}.

\begin{theorem} \label{th-existence-parabolic}
Given integers $n\ge 2$ and $r\in\{1,\dots, n-1\}$, the following  hold:
\begin{itemize}[parsep=1ex]
  \item[\rm i)] There exists a parabolic convex $r$-translator
  $\Sigma_L$  in $\h^n\times\R$ (to be called the \emph{parabolic}
  $r$-\emph{bowl soliton}) which is an entire vertical graph
  with unbounded height function, from above and from below (Fig.~{\rm\ref{fig-parabolicbowl}}).
  In addition, one has:
  %%%%%%%%%%%%%%%%%%%%
  \subitem $\bullet$ the angle function $\theta$ of $\Sigma_L$ is constant
  and satisfies $\theta=\theta_L$.
  %%%%%%%%%%%%%%%%%%%%
  \subitem $\bullet$ all  principal curvatures of $\Sigma_L$ are  constant.
  %%%%%%%%%%%%%%%%%%%%
  \subitem $\bullet$ $\Sigma_L$ is isoparametric.

%%%%%%%%%%%%%%%%%%%%%%%%%%%%%%%%%%%%%%%%%%%%%%%%%%%%%%%%%%%%%%%%%%%%%%%%%%%%%%%%%%%%%%%%%%%

  \item[\rm ii)] If \,$r$ is odd, there is a
  properly embedded parabolic $r$-translator $\Sigma$  in $\h^n\times\R$
  (to be called the \emph{parabolic} $r$-\emph{translating catenoid})
  which is homeomorphic to Euclidean space $\R^n.$  In addition, the following
  assertions hold (Fig.~{\rm\ref{fig-parabolicwing}}):
  %%%%%%%%%%%%%%%%%%%%
  \subitem $\bullet$ $\Sigma$ is the union of two graphs  $\Sigma^-$ and
  $\Sigma^+$ over the complement of the horoball bounded by the horosphere $\mathcal H_0\subset\h^n,$ both unbounded from above,
   such that $\partial\Sigma^{\pm}=\mathcal H_0.$
  %%%%%%%%%%%%%%%%%%%%
  \subitem $\bullet$ $\Sigma$
  is contained in a half-space  of \,$\h^n\times\R,$
  and its set of points of minimal height is a horosphere
  in a horizontal hyperplane $\Pi_t$, $t<0$.
  %%%%%%%%%%%%%%%%%%%%
  \subitem $\bullet$ For $r>1,$ $\Sigma$
  is $C^2$-singular along its horosphere of minimal height.
  %%%%%%%%%%%%%%%%%%%%
  \subitem $\bullet$ The graphs $\Sigma^-$ and
  $\Sigma^+$ are asymptotic to
  the constant angle parabolic $r$-bowl soliton $\Sigma_L.$ More precisely, the angle functions
  $\theta^-$ and $\theta^+$ of  $\Sigma^-$ and
  $\Sigma^+$, respectively, satisfy:
  \[
  \lim_{s\rightarrow+\infty}\theta^-(s)=\lim_{s\rightarrow+\infty}\theta^+(s)=\theta_L\,.
  \]
%%%%%%%%%%%%%%%%%%%%%%%%%%%%%%%%%%%%%%%%%%%%%%%%%%%%%%%

\item[\rm iii)] If \,$r$ is even, there  are two
  properly embedded parabolic $r$-translators  $\Sigma_1$ and $\Sigma_2$   in $\h^n\times\R$
  (to be called \emph{parabolic} $r$-\emph{translating catenoids}),
   both with nonempty boundary and homeomorphic to the half-space $\R^{n-1}\times[0,+\infty)$.
  In addition, one has that (Fig.~{\rm\ref{fig-parabolictranslatingcatenoid}}):
  %%%%%%%%%%%%%%%%%%%%
  \subitem $\bullet$
  $\Sigma_1$ and $\Sigma_2$ are both  unbounded graphs
  in the half-space  $\h^n\times[0,+\infty)$
  on the complement of a horoball in $\Pi_0$.
  %%%%%%%%%%%%%%%%%%%%
  \subitem $\bullet$ Along its boundary, the $r$-translator $\Sigma_1$ is tangent to
  $\Pi_0$, whereas $\Sigma_2$ is orthogonal to $\Pi_0$.
  %%%%%%%%%%%%%%%%%%%%
  \subitem $\bullet$ Denoting by $\theta$ the angle function
  of either $\Sigma_1$ or  $\Sigma_2$, one has:
  \[
  \lim_{s\rightarrow+\infty}\theta(s)=\theta_L.
  \]
\end{itemize}
\end{theorem}

\begin{figure}[hbt]
 \centering
  \includegraphics[width=6cm]{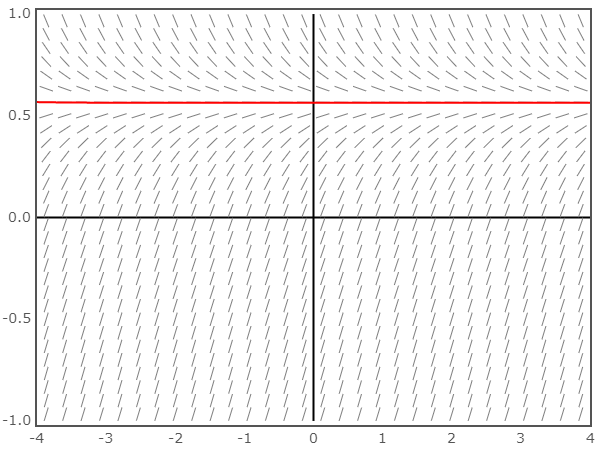}
  \hspace{1cm}
 \includegraphics[width=4.5cm]{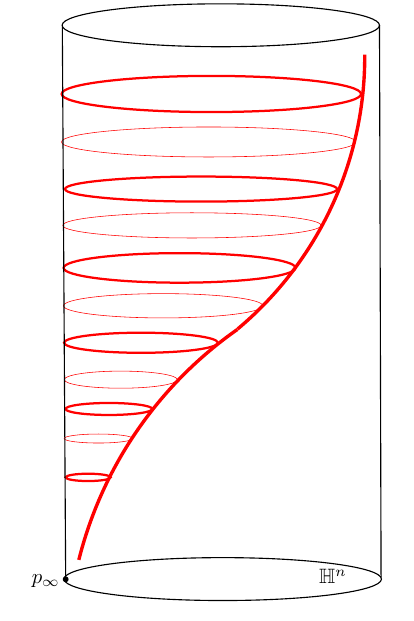}
 \caption{\small The graph of the constant function $\tau_L\equiv L$  (left) and the
 parabolic $r$-bowl soliton obtained from it (right).}
 \label{fig-parabolicbowl}
\end{figure}

\begin{proof}
Let $\tau_L=L$ be the constant solution to~\eqref{eq-IVP-parabolic}
and set $\rho_L=\tau_L^{1/r}$.
Then, by Proposition~\ref{prop-main},
the $(\mathcal H_s,\phi)$-graph $\Sigma_L$ with $\rho$-function $\rho_L$
and height function
\begin{equation} \label{eq-phiL}
\phi_L(s)=\int_{0}^{s}\frac{\rho_L}{\sqrt{1-\rho_L^2}}du=
\frac{\rho_L}{\sqrt{1-\rho_L^2}}s, \,\,\, s\in(-\infty,+\infty),
\end{equation}
is an $r$-translator in $\h^n\times\R.$

Since $\phi_L$ is a linear function on  $\R$,
$\Sigma_L$ is an entire graph over $\Pi_0$ whose height function
is unbounded from above and from below (Fig.~\ref{fig-parabolicbowl}).
Moreover, the angle function of $\Sigma_L$ is constant and coincides with
the limit angle $\theta_L$, for
$$\theta_L^2=1-L^{2/r}=1-\rho_L^2.$$
\begin{figure}[hbt]
 \centering
  \includegraphics[width=5.5cm]{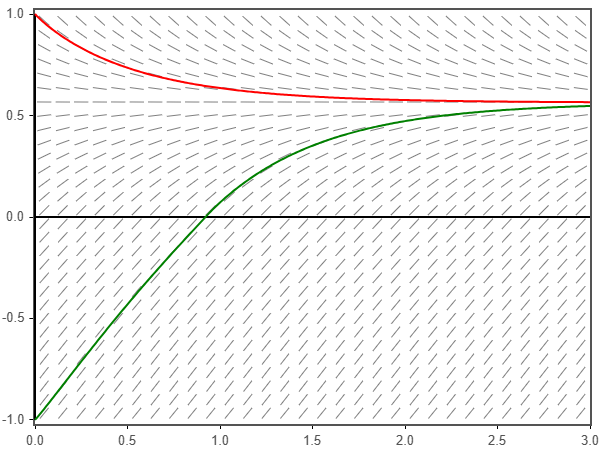}
  \hfill
 \includegraphics[width=5cm]{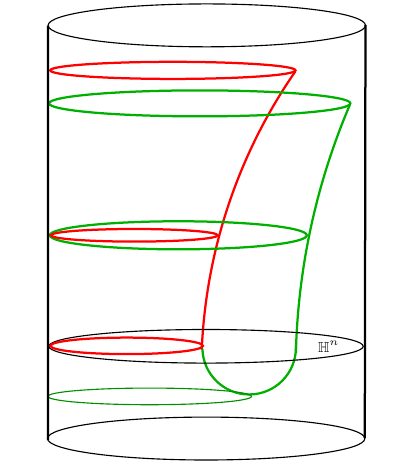}
 \caption{\small The graphs of $\tau_0^-$ and $\tau_0^+$ (left) and the
 parabolic $r$(odd)-translating catenoid $\Sigma_\lambda$ obtained from them (right). For $r>1,$ $\Sigma_\lambda$ is $C^2$-singular
 on the horizontal horosphere of minimal height.}
 \label{fig-parabolicwing}
\end{figure}
In addition, it follows from~\eqref{eq-principalcurvatures} (for $\alpha=1$) that the principal
curvatures $k_i$  of $\Sigma_L$ are all constant and positive, except for $k_n,$ which vanishes
everywhere. In particular, $\Sigma_L$ is convex and it has constant mean curvature,
regardless the value of $r$.

Finally, to prove that $\Sigma_L$ is isoparametric (see Section~\ref{sec-isoparametric}),
for each $u\in\mathbb R$, denote by  $\Sigma_L^u$  the parallel
hypersurface of $\h^n\times\R$ at distance $u$ from $\Sigma_L$.
Recall from~\eqref{eq-immersion} that $\Sigma_L=f(\mathcal H_0\times\R)$, where
$f$ is the immersion
$$f(p,s)=(\gamma_p(s),\phi_L(s)), \,\,\, (p,s)\in\mathcal H_0\times\R,$$
being $\gamma_p$ the geodesic of $\h^n$ given by
$\gamma_p(s):=\exp_p(s\eta_{0}(p))$.
Also, from~\eqref{eq-unitnormal}, the unit normal $N$ of $\Sigma_L$ at $f(p,s)$ is
\begin{equation} \label{eq-N}
N=-\rho_L\eta_s(p)+\theta_L\partial_t.
\end{equation}

Given $s\in\R$, write $\tilde s=s-u\rho_L$. Then, from the linearity of
$\phi_L$, we have that $\phi_L(s)=\phi_L(\tilde s)+u\phi_L(\rho_L)$. From this,~\eqref{eq-theta2},
and~\eqref{eq-phiL}, one easily gets
\begin{equation} \label{eq-000}
\phi_L(s)+u\theta_L=\phi_L(\tilde s)+\frac{u}{\sqrt{1-\rho_L^2}}\cdot
\end{equation}

Now, denoting by $\overbar\exp$ the exponential map of $\h^n\times\R,$
one has that $\Sigma_L^u=f^u(\mathcal H_0\times\R)$, where $f^u$ is the immersion
\begin{equation} \label{eq-001}
f^u(p,s)=\overbar{\exp}_{f(p,s)}(uN(f(p,s)), \,\,\, (p,s)\in\mathcal H_0\times\R.
\end{equation}
Then, observing that $\eta_s=\gamma_p'(s),$ we have from~\eqref{eq-N},~\eqref{eq-000}
and~\eqref{eq-001} that
\begin{eqnarray*}
f^u(p,s)&=&(\exp_{\gamma_p(s)}(-u\rho_L\gamma_p'(s)),\phi_L(s)+u\theta_L)\\
&=&(\exp_p((s-u\rho_L)\eta_{0}(p)),\phi_L(s)+u\theta_L)\\
&=&\left(\exp_p(\tilde s\eta_{0}(p)),\phi_L(\tilde s)+\frac{u}{\sqrt{1-\rho_L^2}}\right),
\end{eqnarray*}
which shows that $\Sigma_L^u$ is nothing but
a vertical translation of $\Sigma_L$. Therefore, for any $u\in\R$, $\Sigma_L^u$ has constant principal curvatures and,
in particular, constant mean curvature, giving that $\Sigma_L$ is indeed isoparametric. This proves (i).

Considering the fact that Proposition~\ref{prop-qualitative} holds for $\alpha=1$,
we conclude that the proofs of (ii) and (iii) are completely  analogous to the ones given for
assertions (ii) and (iii) of Theorem~\ref{th-existence}.
\end{proof}

\begin{figure}[hbt]
 \centering
  \includegraphics[width=4cm]{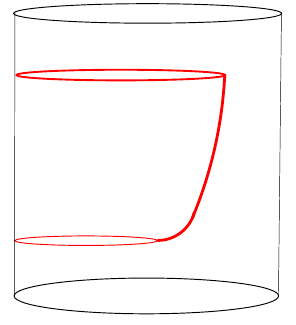}
  \hspace{1cm}
 \includegraphics[width=4.28cm]{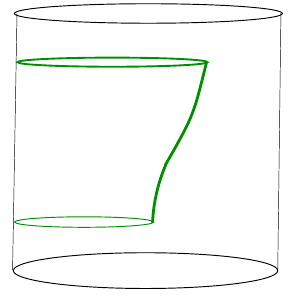}
 \caption{\small Parabolic $r$(even)-translating catenoids $\Sigma_1$ (left) and $\Sigma_2$ (right).}
 \label{fig-parabolictranslatingcatenoid}
\end{figure}

\begin{remark}
For $n=2$ and $r=1$, the parabolic $r$-bowl-soliton was previously obtained in~\cite{dominguezvazquez-manzano}
as an element of a one-parameter family of isoparametric surfaces of $\h^2\times\R$ called \emph{parabolic helicoids}.
\end{remark}

\begin{remark} \label{rem-uniquenessgraphparabolic}
With the notation of Proposition~\ref{prop-qualitative}, we have from~\eqref{eq-parabolicMCFode2} that
\[
\tau_{s_0}(s)=\tau_{0}(s-s_0)\,\,\,\forall s\in(s_0,+\infty), s_0\in\R.
\]
From this equality, we conclude that the height functions of the  $(\mathcal H_s,\phi)$-graphs
associated to two distinct solutions of~\eqref{eq-parabolicMCFode2} differ
by a vertical translation in $\hr,$ so that they are congruent. Notice that this contrasts with
the rotational case considered in the preceding section (cf.~Theorem~\ref{th-existence}, items  (ii) and (iii)).
\end{remark}

Regarding the Cauchy problem~\eqref{eq-IVP-parabolic},
since $\partial\Omega=\R\times\{-1\}\cup\R\times\{1\}$, we have that
the only solutions  are the constant function $\tau_L$,
and those of the type $\tau_{s_0}^{\pm}$.
From this fact, and the considerations of Remark~\ref{rem-uniquenessgraphparabolic},
we conclude that a version of the uniqueness result
for rotational $r(<\hspace{-.1cm}n)$-translators obtained in Proposition~\ref{prop-uniquenessgraphrotational}
holds for parabolic $r(<\hspace{-.1cm}n)$-translators as well. The proof is completely analogous. More precisely, we have

\begin{proposition} \label{prop-uniquenessgraphparabolic}
Let $\Sigma$ be a connected parabolic $r(<\hspace{-.1cm}n)$-translator in $\hr$
which is a vertical graph over an open set of \,$\h^n$.
Then, up to an ambient isometry, $\Sigma$ is an open set of either the parabolic $r$-bowl soliton or
a parabolic $r$-translating catenoid.
\end{proposition}

\section{Hyperbolic Translators to $r(<\hspace{-.1cm}n)$-MCF in $\h^n\times\R$} \label{sec-hyperbolic}

In analogy with the preceding  section, we consider now
hyperbolic $r(<\hspace{-.1cm}n)$-translators
in $\h^n\times\R$, i.e., those which are invariant by
horizontal hyperbolic translations of $\h^n\times\R$.
So, they will be constructed from  $(\mathcal E_s,\phi)$-graphs, where
$$\mathscr E:=\{\mathcal E_s\,;\, s\in I\subset(-\infty,+\infty)\}$$
is a family of  hypersurfaces of $\h^n$ which are equidistant from a
fixed  totally geodesic hyperplane $\mathcal E_0\subset\h^n.$
The open interval $I$ defining the family $\mathscr E$ is:
\[
I:=\left\{
\begin{array}{lcc}
(-\infty,+\infty) & \text{if} & r=1,\\[1ex]
(0,+\infty) & \text{if} & r>1.
\end{array}
\right.
\]

\begin{comment}
For each $s\in(-\infty,+\infty),$ we shall consider the  orientation of
$\mathcal E_s$ in such a way  that its principal curvature function is
$\alpha(s)=-\tanh s.$
From ~\eqref{eq-principalcurvatures}, the  principal curvatures of $\Sigma$ are
\begin{equation} \label{eq-principalcurvatures-hyperbolic}
k_i(s)=\rho(s)\tanh s, \,\,\, i=1,\dots, n-1, \quad\text{and}\quad k_n(s)=\rho'(s),
\end{equation}
so that equation \eqref{eq-rMCFodegeneral} now is:
\begin{equation} \label{eq-MCFode-hyperbolic}
{{n-1}\choose{r}}\tanh^r(s)\rho^r(s)+{{n-1}\choose{r-1}}\tanh^{r-1}(s)\rho^{r-1}(s)\rho'(s)=\sqrt{1-\rho^2(s)}.
\end{equation}

Setting $\tau=\rho^r,$ equation \eqref{eq-MCFode-hyperbolic} takes the form
\end{comment}

In this setting, with the notation of Section~\ref{sec-translators}, we have that $\alpha=\tanh$.
Hence, equation~\eqref{eq-ODEtau} becomes
\begin{equation} \label{eq-hyperbolicMCFode2}
\tau'(s)=C\sqrt{1-\tau^{2/r}(s)}\coth^{r-1}(s)-(n-r)\tanh(s)\tau(s),
\end{equation}
and the associated Cauchy problem is:
\begin{equation} \label{eq-IVP-hyperbolic}
\left\{
\begin{array}{l}
y'(s)=F(s,y(s))\\[1ex]
y(s_0)=y_0,
\end{array}
\right.
\end{equation}
where $(s_0,y_0)\in\Omega:=I\times[-1,1]$  and $F=F_{(n,r)}$ is the function:
\begin{equation} \label{eq-Fhyperbolic}
F(s,y):=C\sqrt{1-y^{2/r}}\coth^{r-1}(s)-(n-r)\tanh(s)y, \,\,\, 1\le r<n,  \,\,\, (s,y)\in\Omega.
\end{equation}

Figure~\ref{fig-hyperbolicfigure1} shows the graphs of some solutions
to the Cauchy problem~\eqref{eq-IVP-hyperbolic} for the cases $r=1$ and $r>1$.

\begin{figure}[htb]
 \centering
 \includegraphics[scale=.27]{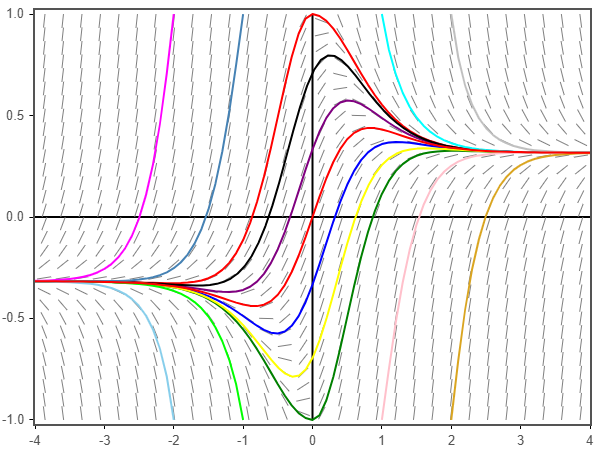}
 \hspace{1cm}
 \includegraphics[scale=.27]{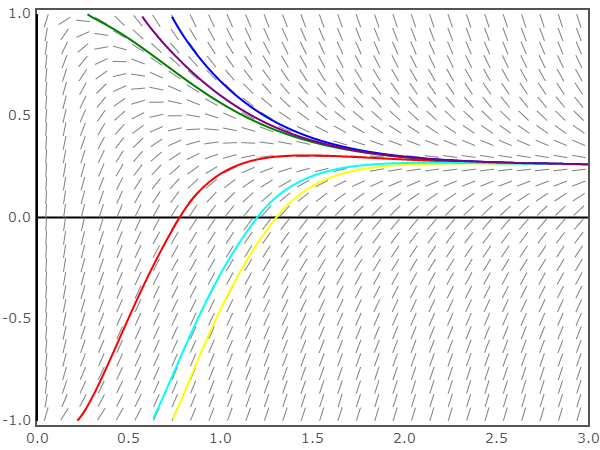}
 \caption{\small Graphs of solutions to~\eqref{eq-IVP-hyperbolic} for $F_{(4,1)}$ (left),
 and $F_{(4,3)}$ (right).}
 \label{fig-hyperbolicfigure1}
\end{figure}

\begin{proposition} \label{prop-hyperbolicsolutions}
 For $r=1$ and $\lambda\in[-1,1],$ the solution $\tau_\lambda$ to~\eqref{eq-IVP-hyperbolic}
   satisfying $\tau_\lambda(0)=\lambda$ is defined
   in $(-\infty,+\infty)$ and satisfies {\rm (cf.~Fig.~{\rm\ref{fig-hyperbolicfigure1}}, left):}
   $$\lim_{s\to\pm\infty}\tau_\lambda(s)=\pm L.$$
\end{proposition}

\begin{proof}
For $r=1,$ the function $F=F_{(n,r)}$ is
\begin{equation} \label{eq-F}
F(s,y)=\sqrt{1-y^{2}}-(n-1)\tanh(s)y,
\end{equation}
which is bounded on any strip $\Omega_\delta=I_\delta\times[-1,1]\subset\Omega$, where
$I_\delta:=(-\delta,\delta)\owns 0.$ So, for a sufficiently small $\delta$,
$\tau_\lambda$ is well defined in $I_\delta$  for all $\lambda\in(-1,1).$

For $\lambda=\pm 1$, we have that $\tau_\lambda'(0)=0$. In addition,
$F$ is bounded in $\Omega_\delta$.  Hence,
as $\lambda\to\pm 1$, the solutions
$\tau_\lambda$, $\lambda\ne\pm1$, converge uniformly to the solutions $\tau_{\pm1}$.
As a consequence,  possibly taking a smaller $\delta$,
$\tau_{\pm 1}$ are both defined in $I_\delta$.

Since $F(s_0,1)<0<F(s_0,-1)$ (resp. $F(s_0,-1)<0<F(s_0,1)$)
for $s_0>0$ (resp. $s_0<0),$ for any $\lambda\in[-1,1],$ we have that $\tau_\lambda(s)\ne\pm1$
for all $s\in I_{\max}$, which implies that
$I_{\max}=(-\infty,+\infty).$

Now, arguing as in the proof of Proposition~\ref{prop-qualitative}, one easily concludes that
each function $\tau_\lambda$ has at most two critical points, so that
\[
L_{\pm}^\lambda:=\lim_{s\to\pm\infty}\tau_\lambda(s)
\]
is well defined and $\lim_{s\to\pm\infty}\tau_\lambda'(s)=0.$

These last two equalities and~\eqref{eq-hyperbolicMCFode2} then yield
\[
\sqrt{1-(L_{\pm}^\lambda)^{2}}\mp(n-1)L_{\pm}^\lambda=0,
\]
which implies that $-L_-^\lambda=L_+^\lambda=L$.
\end{proof}

\begin{proposition} \label{prop-uniquenesshyperbolicsolutions}
The only solutions to the Cauchy problem~{\rm\eqref{eq-IVP-hyperbolic}}
which are defined in a maximal interval are the functions $\tau_{s_0}^{\pm}$
of Proposition~{\rm\ref{prop-qualitative}}, and the
functions $\tau_\lambda$ of Proposition~{\rm\ref{prop-hyperbolicsolutions}} {\rm(}if $r=1${\rm).}
\end{proposition}

\begin{proof}
The result is immediate for $r=1$, for in this case we have
$$\partial\Omega=\R\times\{-1\}\cup\R\times\{1\},$$ so that
the endpoint of a solution to~{\rm\eqref{eq-IVP-hyperbolic}}, if it exists,
has $y$-coordinate $-1$ or $1$.

Now, let us suppose that $1<r<n.$
In this case, it suffices to prove that there is no solution to~\eqref{eq-IVProtational}
whose graph has an endpoint of the form $p:=(0,a)$.
Arguing as in the proof
of Proposition~\ref{prop-uniquenesssolutions}, assume, by contradiction,
that such a solution exists and call it $\tau$. We can  also assume, without loss of generality,
that $a>0$. Extending $\tau$ to $0$
by making $\tau(0)=a$, it is easily seen that the graph $\mathcal G$ of $\tau$ is tangent
to the $y$-axis at $p$.

Choose a small  $c>0$ such that $\tau'>0$ on $(0,c)$,
and set $b:=\tau(c)>0.$
Given a positive $\delta<c$, set $\mathcal S_\delta$ for the line segment from $p=(0,a)$ to $(\delta,b)$
and write $p_\delta=\mathcal S_\delta\cap\mathcal G$.
By Rolle's Theorem,  there exists a point $q_\delta:=(s_\delta,\tau(s_\delta))$ in the
open arc of $\mathcal G$ from $p$ to $p_\delta$ such that the tangent line to
$\mathcal G$ at $q_\delta$ is parallel to
$\mathcal S_\delta$, which yields $\tau'(s_\delta)=(b-a)/\delta$.
Thus, by~\eqref{eq-hyperbolicMCFode2},
\[
C\sqrt{1-\tau^{2/r}(s_\delta)}\coth^{r-1}(s_\delta)-(n-r)\tanh(s_\delta)\tau(s_\delta)=\frac{b-a}{\delta}\cdot
\]

Since $0<s_\delta<\delta$, we have that $\delta\coth^{r-1}(s_\delta)>s_\delta\coth^{r-1}(s_\delta)\ge 1.$
This, together with the last equality above, yields
\begin{equation} \label{eq-limitb-a}
b-a>C\sqrt{1-\tau^{2/r}(s_\delta)}-(n-r)\delta\tanh(s_\delta)\tau(s_\delta).
\end{equation}
Letting $\delta\to0$ on both sides of~\eqref{eq-limitb-a} gives $b-a\ge C\sqrt{1-a^{2/r}}$, which is
a contradiction, since we can choose $c>0$ in such a way that $b=\tau(c)$ is arbitrarily close to $a$.
This finishes the proof.
\end{proof}

\begin{figure}
 \centering
  \includegraphics[width=5.5cm]{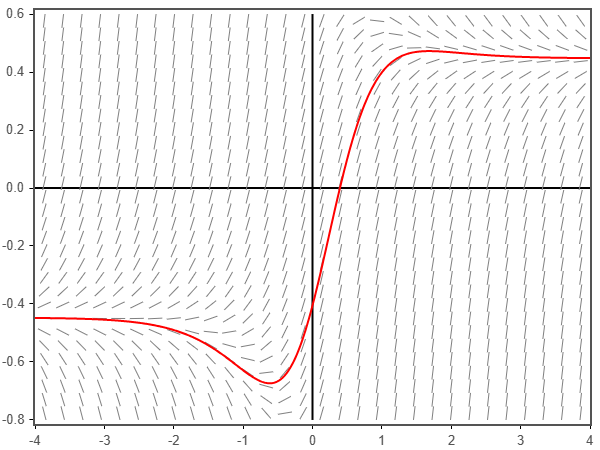}
  \hfill
 \includegraphics[width=6cm]{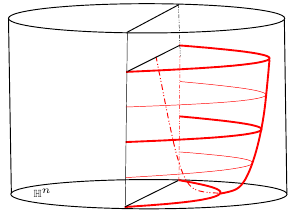}
 \caption{\small The graph of $\tau_\lambda$ (left) and part of the
 hyperbolic bowl soliton $\Sigma_\lambda$ obtained from it (right).}
 \label{fig-hypgrimreaper}
\end{figure}

\begin{theorem} \label{th-existence-hyperbolic}
Given integers $n\ge 2$ and $r\in\{1,\dots, n-1\}$, the following  hold:
\begin{itemize}[parsep=1ex]
  \item[\rm i)] For $r=1,$ there exists a one-parameter family
  $\mathscr B=\{\Sigma_\lambda\,;\, \lambda\in[-1,0]\}$ of
  hyperbolic $r$-translators
  (to be called   \emph{hyperbolic bowl solitons}) with the following properties:
  %%%%%%%%%%%%%%%%%%%%
  \subitem $\bullet$ For each $\lambda\in(-1,0),$  $\Sigma_\lambda$ is
  an entire vertical graph contained in the
  half-space $\h^n\times[0,+\infty)$ with unbounded height function, which is tangent to an equidistant hypersurface
  $\mathcal E_{s(\lambda)}\in\mathscr E$ (Fig.~{\rm \ref{fig-hypgrimreaper}}).
  %%%%%%%%%%%%%%%%%%%%
  \subitem $\bullet$ For $\lambda=-1$, $\Sigma_\lambda$ is
  a complete graph over  a half-space of \,$\h^n$ determined by the hyperplane $\mathcal E_0$, which is asymptotic
  to $\mathcal E_0\times [0,+\infty)$. In addition, $\Sigma_\lambda$
  is contained in the half-space $\h^n\times[0,+\infty)$ with unbounded height function,
  being tangent to an equidistant hypersurface
  $\mathcal E_{s(\lambda)}\in\mathscr E$ (Fig.~{\rm \ref{fig-hypbowl02}}).

   \item[\rm ii)] For $r=1,$ there exists a one-parameter family
  $\mathscr G=\{\Sigma_\mu\,;\, \mu\in(-\infty,0)\}$ of
  hyperbolic $r$-translators
  (to be called   \emph{hyperbolic $1$-grim reapers}). Each
  translator $\Sigma_\mu\in\mathscr G$ is a complete graph over
  a half-space of \,$\h^n$ determined by the hyperplane $\mathcal E_0$, which
  is asymptotic
  to $\mathcal E_0\times [0,-\infty)$, intersects $\h^n\times\{0\}$ along $\mathcal E_\mu$,
   and has unbounded (above and below) height function (Fig.~{\rm \ref{fig-1-hypgrimreaper}}).

  \item[\rm iii)] If \,$r$ is odd, there exists a one-parameter family
  $\mathscr C_r=\{\Sigma_\lambda\,;\, \lambda\in(0,+\infty)\}$ of
  properly embedded hyperbolic $r$-translators in $\h^n\times\R$
  (to be called \emph{hyperbolic} $r$-\emph{translating catenoids})
  which are all homeomorphic to Euclidean space $\R^n.$  In addition, one has that
  (Fig.~{\rm\ref{fig-hyperbolicwing}}):
  %%%%%%%%%%%%%%%%%%%%
  \subitem $\bullet$ For each $\lambda\in(0,+\infty),$  $\Sigma_\lambda$ is the union of two graphs  $\Sigma_\lambda^-$ and
  $\Sigma_\lambda^+,$ both unbounded from above, over one of the connected components of
  the complement of the convex region of $\h^n$ bounded by $\mathcal E_0$ and $\mathcal E_\lambda.$
  %%%%%%%%%%%%%%%%%%%%
  \subitem $\bullet$ Each hyperbolic  $r$-translating catenoid $\Sigma_\lambda\in\mathscr C_r$
  is contained in a half-space  of \,$\h^n\times\R,$
  and its set of points of minimal height is an equidistant hypersurface
  in a horizontal hyperplane $\Pi_t$, $t<0$.
  %%%%%%%%%%%%%%%%%%%%
  \subitem $\bullet$ For $r>1,$ any hyperbolic  $r$-translating catenoid $\Sigma_\lambda\in\mathscr C_r$
  is $C^2$-singular along its equidistant hypersurface of minimal height.
  %%%%%%%%%%%%%%%%%%%%
  \subitem $\bullet$ For any $\lambda\in(0,+\infty),$ the angle functions
  $\theta^-$ and \,$\theta^+$  of   $\Sigma_\lambda^-$ and
  $\Sigma_\lambda^+,$  respectively, satisfy:
  \[ %begin{equation} \label{eq-parabolicthetas}
  \lim_{s\rightarrow+\infty}\theta^-(s)=\lim_{s\rightarrow+\infty}\theta^+(s)=\theta_L,
  \] %end{equation}
  where $\theta_L$ is the limit angle.
  %%%%%%%%%%%%%%%%%%%%%%%%%%%%%%%%%%%%%%%%%%%%%%%%%%%%%%%%%%%%%%%%%%%%%%%%%%%%%
  %%%%%%%%%%%%%%%%%%%%%%%%%%%%%%%%%%%%%%%%%%%%%%%%%%%%%%%%%%%%%%%%%%%%%%%%%%%%%
  \item[\rm iv)] If \,$r$ is even, there  are two one-parameter families
  $\mathscr C_r^i=\{\Sigma_\lambda^i\,;\, \lambda>0\},$ $i=1,2,$ of
  properly embedded hyperbolic $r$-translators in $\h^n\times\R$
  (to be called \emph{hyperbolic} $r$-\emph{translating catenoids}) with nonempty boundary,
  which are all homeomorphic to a half-space $\R^n\times[0,+\infty).$
  In addition, one has (Fig.~{\rm\ref{fig-hyperbolictranslatingcatenoid}}):
  %%%%%%%%%%%%%%%%%%%%
  \subitem $\bullet$ For each $\lambda>0,$  $\Sigma_\lambda^i$ is an unbounded graph
  in $\h^n\times [0,+\infty)$ over one of the connected components of
  the complement of the convex region of $\h^n$ bounded by $\mathcal E_0$ and an equidistant
  $\mathcal E_{\bar\lambda},$ $\bar\lambda=\bar\lambda(\lambda).$
  %%%%%%%%%%%%%%%%%%%%
  \subitem $\bullet$ Along their boundaries, the $r$-translators in $\mathscr C_r^1$ are tangent to
  the horizontal hyperplane $\Pi_0$, whereas those in $\mathscr C_r^2$ are orthogonal to $\Pi_0$.
  %%%%%%%%%%%%%%%%%%%%%%%%%%%%%%%%%%%%%%%%%%%%%%%%%%%%%%%%%%%%%%%
  \subitem $\bullet$ For each $\lambda>0,$  the angle function
  $\theta_\lambda^i$   of $\Sigma_\lambda^i\in\mathscr C_r^i$ satisfies:
  \[%\begin{equation} \label{eq-thetasevenhyperbolic}
  \lim_{s\rightarrow+\infty}\theta_\lambda^i(s)=\theta_L.
  \] %\end{equation}
\end{itemize}
\end{theorem}

\begin{proof}
(i) Given $\lambda\in(-1,0],$ let $\tau_\lambda\colon(-\infty,+\infty)\rightarrow\R$ be as in
Proposition \ref{prop-hyperbolicsolutions}. Set
$\Sigma_\lambda$ for  the $(\mathcal E_s,\phi_\lambda)$-graph
with $\rho$-function  $\rho_\lambda=\tau_\lambda$ and  height function
\[
\phi_\lambda(s)=\int_{s(\lambda)}^{s}\frac{\rho_\lambda(u)}{\sqrt{1-\rho_\lambda^2(u)}}du, \,\,\, s\in(-\infty,+\infty),
\]
where $s(\lambda)$ satisfies  $\rho_\lambda(s(\lambda))=0.$
Then, $\Sigma$ is an entire graph over $\h^n$ and, by Proposition~\ref{prop-main},  is a translator to MCF
in $\h^n\times\R.$ Also, since $\rho_\lambda$ is negative in $(-\infty,s(\lambda))$ and positive in $(s(\lambda),+\infty),$
we have that $\phi_\lambda(s)>0$ for all $s\ne s(\lambda),$ which implies that $\Sigma_\lambda$ is contained in the
half-space $\h^n\times[0,+\infty),$ and is tangent to the equidistant hypersurface
$\mathcal E_{s(\lambda)}\subset\h^n\times\{0\},$ for
$\phi_\lambda(s(\lambda))=\phi_\lambda'(s(\lambda))=0$ (Fig.~\ref{fig-hypgrimreaper}).

\begin{figure}[h]
 \centering
  \includegraphics[width=5.5cm]{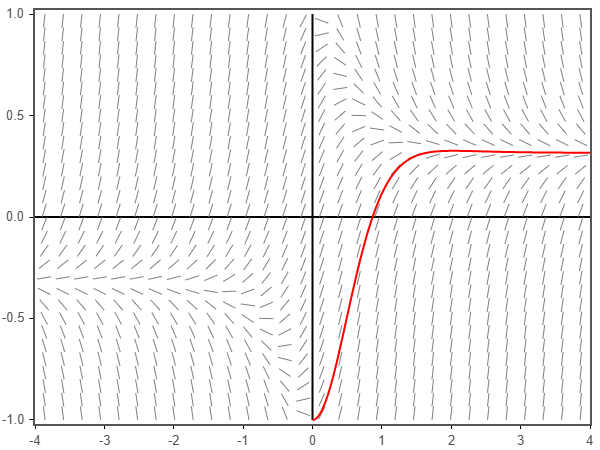}
  \hfill
 \includegraphics[width=6cm]{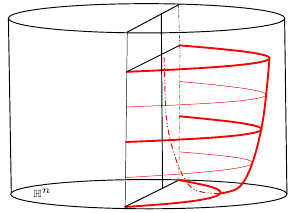}
 \caption{\small The graph of $\tau_{-1}$ on $(0,+\infty)$ (left) and  the
 hyperbolic bowl soliton $\Sigma_{-1}$ obtained from it (right).}
 \label{fig-hypbowl02}
\end{figure}

To prove that $\phi_\lambda$ is unbounded, notice that the
function $\rho_\lambda/\sqrt{1-\rho_\lambda^2}$ is bounded below by a positive constant $C_0$
in any interval $(a,+\infty)$ with $a>s(\lambda)$ sufficiently large,  for
\[
\lim_{s\rightarrow+\infty}\frac{\rho_\lambda(s)}{\sqrt{1-\rho_\lambda^2(s)}}=\frac{L}{\sqrt{1-L^2}}>0.
\]
Then, for any $s\in (a,+\infty),$ one has
\[
\phi_\lambda(s)>\int_{a}^{s}\frac{\rho_0(u)}{\sqrt{1-\rho_0^2(u)^2}}du\ge C_0(s-a),
\]
which implies that $\phi_\lambda$ is unbounded.

Now, assume that $\lambda=-1$ and consider
the $(\mathcal E_s,\phi)$-graph $\Sigma$
defined by $\rho=\tau_{-1}$ with $s>0$.
Then, the height function $\phi$ of $\Sigma$ is
\[
\phi(s)=\int_{s(-1)}^{s}\frac{\rho(u)}{\sqrt{1-\rho^2(u)}}du, \,\,\, s\in(0,+\infty).
\]

\begin{figure}
 \centering
  \includegraphics[width=5.5cm]{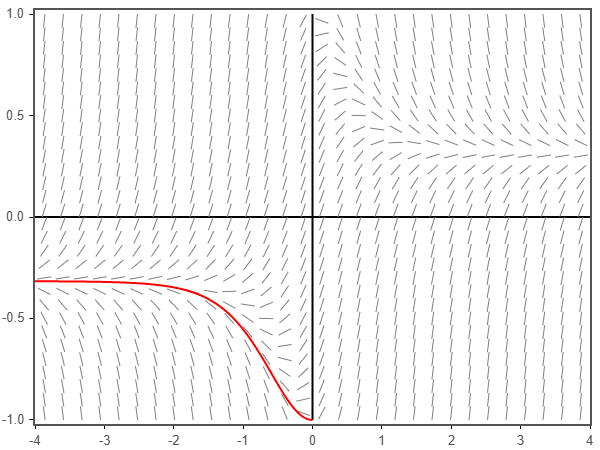}
  \hfill
 \includegraphics[width=6cm]{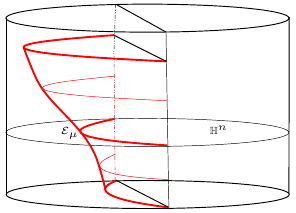}
 \caption{\small The graph of $\tau_{-1}$ on $(-\infty,0)$ (left) and  the
 hyperbolic grim reaper $\Sigma_\mu$ obtained from it (right).}
 \label{fig-1-hypgrimreaper}
\end{figure}

As above, $\Sigma$ has unbounded height function, is contained in
$\h^n\times[0,+\infty),$ and is tangent to the equidistant hypersurface
$\mathcal E_{s(-1)}\subset\h^n\times\{0\}.$ So, it remains to prove that
$\Sigma$ is complete and asymptotic to $\mathcal E_0\times [0,+\infty)$.

We have that $\rho^2(0)=1$, $(\rho^2)'(0)=0$, and
$(\rho^2)''(0)=-2\rho''(0)\le0$, since $s=0$ is a minimum of $\rho$.
So, the second order Taylor's formula of $\rho^2$ around $s=0$ is
\begin{equation} \label{eq-taylor01}
\rho^2(s)=1+\frac12(\rho^2)''(0)s^2+f(s), \quad \lim_{s\to 0}\frac{f(s)}{s^2}=0.
\end{equation}

Setting $a:=(\rho^2)''(s_{\max})/2,$ we have from \eqref{eq-taylor01} that
\begin{equation} \label{eq-limitzero01}
\lim_{s\to0}\frac{\sqrt{1-\rho^2(s)}}{|s|}=\sqrt{-a}>0.
\end{equation}

Now, by successive applications of the l'Hôpital's rule, we have from~\eqref{eq-taylor01} that
\begin{equation} \label{eq-lHopital01}
\lim_{s\to 0}\frac{f(s)}{s^3}=\lim_{s\to 0}\frac{(\rho^2)'(s)-2as}{3s^2}=
\lim_{s\to0}\frac{(\rho^2)''(s)-2a}{6s}=\frac{(\rho^2)'''(0)}{6}\ne\pm\infty.
\end{equation}

Finally, let us check that the function
\[
g(s):=\frac{1}{\sqrt{1-\rho^2(s)}}-\frac{1}{\sqrt{-a}(-s)}
\]
is well defined and bounded in a neighborhood of $0.$ With this purpose,
we first observe that, from~\eqref{eq-taylor01}, we have
\[
f(s)=(\sqrt{-a}(-s)-\sqrt{1-\rho^2(s)})(\sqrt{-a}(-s)+\sqrt{1-\rho^2(s)}).
\]
Therefore, we can write $g$ as
\begin{eqnarray*}
g(s) &=& \frac{\sqrt{-a}(-s)-\sqrt{1-\rho^2(s)}}{\sqrt{-a(1-\rho^2(s))}(-s)}\\
            &=& \frac{1}{\sqrt{-a}\sqrt{1-\rho^2(s)}}\frac{1}{(-s)}\frac{f}{\sqrt{1-\rho^2(s)}+\sqrt{-a}(-s)}\\
            &=& \frac{-s}{\sqrt{-a}\sqrt{1-\rho^2(s)}}\frac{f}{(-s)^3}\frac{1}{\frac{\sqrt{1-\rho^2(s)}}{-s}+\sqrt{-a}}\,\cdot
\end{eqnarray*}

This last equality, together with~\eqref{eq-limitzero01} and~\eqref{eq-lHopital01}, gives that
$\lim_{s\to 0}g(s)$ is well defined and finite, which proves our claim.

To conclude the proof of (i), fix a small $\delta>0$ such that $\delta<s(-1).$
Then, for all $s\in (0,\delta),$ one has
\begin{eqnarray}
\phi(s) &=& \int_{s(-1)}^{\delta}\frac{\rho(u)}{\sqrt{1-\rho^2(u)}}du+\int_{\delta}^s\frac{\rho(u)}{\sqrt{1-\rho^2(u)}}du  \nonumber\\[1ex]
        &\ge& \int_{s(-1)}^{\delta}\frac{\rho(u)}{\sqrt{1-\rho^2(u)}}du+\rho(\delta)\int_{\delta}^s\frac{1}{\sqrt{1-\rho^2(u)}}du. \label{eq-forphi001}
\end{eqnarray}
However, by the definition of $g,$ we have
\begin{eqnarray}
\int_{\delta}^s\frac{1}{\sqrt{1-\rho^2(u)}} &=& \int_{\delta}^sg(u)du+\frac{1}{\sqrt{-a}}\int_{\delta}^s\frac{du}{-u} \nonumber\\[1ex]
&=&\int_{\delta}^sg(u)du+\frac{1}{\sqrt{-a}}\log\left(\frac{\delta}{s}\right).  \label{eq-forphi002}
\end{eqnarray}

Since $g$ is continuous and bounded in $[0,\delta],$
\eqref{eq-forphi001} and~\eqref{eq-forphi002} clearly imply that $\phi(s)\to+\infty$ as $s\to 0$,
which shows that $\Sigma$ is complete. Finally, we have that
\[
\lim_{s\to 0}\phi'(s)=\lim_{s\to 0}\frac{\rho(s)}{\sqrt{1-\rho^2(s)}}=-\infty,
\]
proving that $\Sigma$ is asymptotic to the vertical hyperplane
$\mathcal E_0\times [0,+\infty)$  (Fig.~\ref{fig-hypbowl02}).

\vtt
\noindent
(ii) Consider the function $\rho:=\tau_{-1}|_{(-\infty,0)}$, where
$\tau_{-1}$ is as in~Proposition~\ref{prop-hyperbolicsolutions}.
Given $\mu<0$, let $\Sigma$ be the $(\mathcal E_s,\phi)$-graph
determined by $\rho$ with height function $\phi$ given by
\[
\phi(s)=\int_{\mu}^{s}\frac{\rho(u)}{\sqrt{1-\rho^2(u)}}du, \,\,\, s\in(-\infty, 0).
\]

Since $\rho$ is negative, $\phi$ is strictly decreasing. Moreover, proceeding as
in the proof of (i), one concludes that
\[
\lim_{s\to-\infty}\phi(s)=+\infty \quad\text{and}\quad \lim_{s\to 0}\phi(s)=-\infty,
\]
which implies that $\Sigma$ is complete and asymptotic to $\mathcal E_0\times[0,-\infty)$,
for $\phi'(s)\to-\infty$ as $s\to 0$ (Fig.~\ref{fig-1-hypgrimreaper}). This finishes the proof of (ii).

Since Proposition~\ref{prop-qualitative} holds for $\alpha=\tanh$,
the proofs of (iii) and (iv) are completely  analogous to the ones given for
assertions (ii) and (iii) of Theorem~\ref{th-existence}.
\end{proof}

\begin{figure}[hbt]
 \centering
  \includegraphics[width=5.5cm]{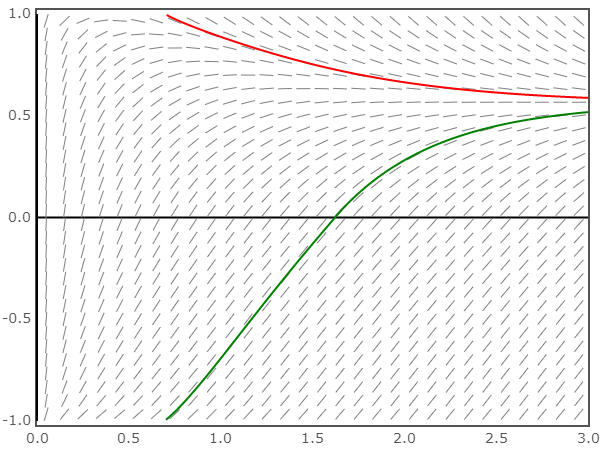}
  \hfill
 \includegraphics[width=5cm]{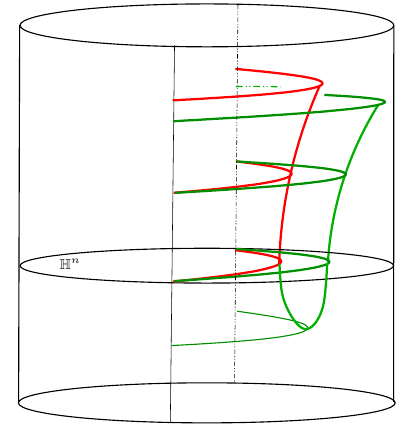}
 \caption{\small The graphs of $\tau_\lambda^-$ and $\tau_\lambda^+$ (left) and the
 hyperbolic $r$(odd)-translating catenoid $\Sigma_\lambda$ obtained from them (right).
 For $r>1,$ $\Sigma_\lambda$ is $C^2$-singular
 on the horizontal equidistant hypersurface of minimal height.}
 \label{fig-hyperbolicwing}
\end{figure}

\begin{remark} \label{rem-symmetryequation}
For $r>1,$ we could have chosen the domain $\Omega$ in the Cauchy problem~\eqref{eq-IVP-hyperbolic}
to be $(\R-\{0\})\times[-1,1].$ However, it is easily checked that the $r$-translators obtained from the solutions
$\tau_{s_0}$ with $s_0<0$ are just the reflections with respect to the vertical hyperplane $\mathcal E_0\times\R$
of the ones obtained from the solutions $\tau_{-s_0}$. The same goes for the one-parameter family of translators
to MCF in Theorem~\ref{th-existence-hyperbolic}-(i). More precisely, for any $\lambda>0,$ the
translator to MCF obtained from the solution $\tau_\lambda$ such that $\tau_\lambda(0)=\lambda$ is the
reflection with respect to $\mathcal E_0\times\R$ of the translator obtained from the solution
$\tau_{-\lambda}.$ Notice that reflections with respect to vertical hyperplanes are isometries of $\hr$ which take
$r$-translators to $r$-translators.
\end{remark}

Proposition~\ref{prop-uniquenesshyperbolicsolutions} and Theorem~\ref{th-existence-hyperbolic}, together with the considerations
of the above remark, give the following uniqueness result, whose proof is completely analogous
to the one given for Proposition~\ref{prop-uniquenessgraphrotational}.

\begin{proposition} \label{prop-uniquenessgraphhyperbolic}
Let $\Sigma$ be a connected hyperbolic $r(<\hspace{-.1cm}n)$-translator in $\hr$
which is a vertical graph over an open set of \,$\h^n$. If $r=1$, up to an ambient isometry,
$\Sigma$ is an open set of  a hyperbolic bowl soliton, a
hyperbolic $1$-grim reaper, or a hyperbolic $1$-translating catenoid. If $r>1$,
up to an ambient isometry, $\Sigma$ is
an open set of a hyperbolic $r$-translating catenoid.
\end{proposition}

\begin{figure}
 \centering
  \includegraphics[width=4.75cm]{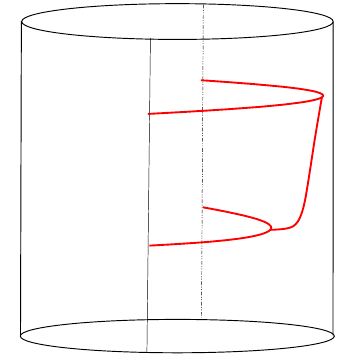}
  \hspace{1cm}
 \includegraphics[width=4.4cm]{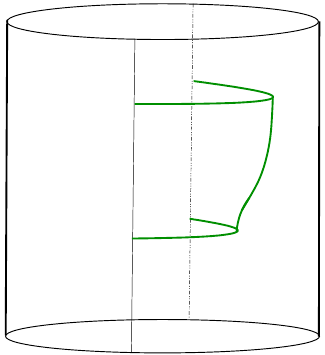}
 \caption{\small Hyperbolic $r$(even)-translating catenoids with boundary, where the one on the left
 belongs to $\mathscr C_r^1$, whereas the one on the right belongs to $\mathscr C_r^2$.}
 \label{fig-hyperbolictranslatingcatenoid}
\end{figure}

\section{Translators to  Gaussian curvature flow in $\qr$} \label{sec-gaussian}

In this section, in analogy with the preceding ones,
we consider invariant  $n$-translators in $\qr$ (i.e., translators to
the Gaussian curvature flow).  With the notation of Section~\ref{sec-translators},
we have that  $\beta:=1/\alpha$ is one of the functions: $\tan_\epsilon$, $\coth$, or the constant $1$.
Hence, when $r=n$, equation~\eqref{eq-ODEtau} becomes
\begin{equation}\label{eq-taur=n}
\tau'(s)=n\sqrt{1-\tau^{\frac 2n}(s)}\beta^{n-1}(s),
\end{equation}
whose associated Cauchy problem is
\begin{equation} \label{eq-CPr=n}
\left\{
\begin{array}{l}
y'(s)=n\sqrt{1-y^{\frac 2n}(s)}\beta^{n-1}(s),\\[1ex]
y(s_0)=y_0, \,\,\, (s_0,y_0)\in\Omega,
\end{array}
\right.
\end{equation}
where $\Omega:=I\times(-1,1)$, being
$$I:=\left\{
\begin{array}{lcl}
(-\infty,+\infty) &\text{if} & \beta=\tan_\epsilon \, \text{or} \,\, \beta=1,\\[1ex]
(0,+\infty)&\text{if} & \beta=\coth.
\end{array}
\right.
$$

\begin{remark}
In the rotational case, the parameter $s$ is the radius of
a geodesic sphere $S_s^{n-1}\subset\mathbb Q_\epsilon^n$,
and so it takes only positive values. However, in this setting, the function
$\beta=\tan_\epsilon$ is well defined in $(-\infty,+\infty)$, which allowed us
to consider the Cauchy problem~\eqref{eq-CPr=n} for the rotational case
in the region $\Omega=(-\infty,+\infty)\times(-1,1)$. This shall give us a better
understanding of the qualitative behavior of the solutions.
\end{remark}

\begin{figure}[htb]
 \centering
 \includegraphics[scale=.29]{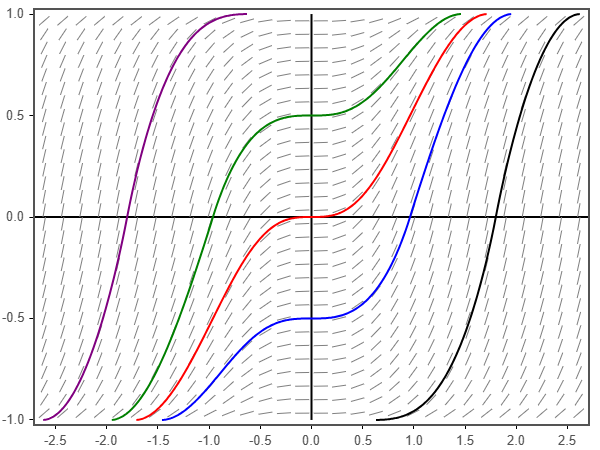}
 \hfill
 \includegraphics[scale=.29]{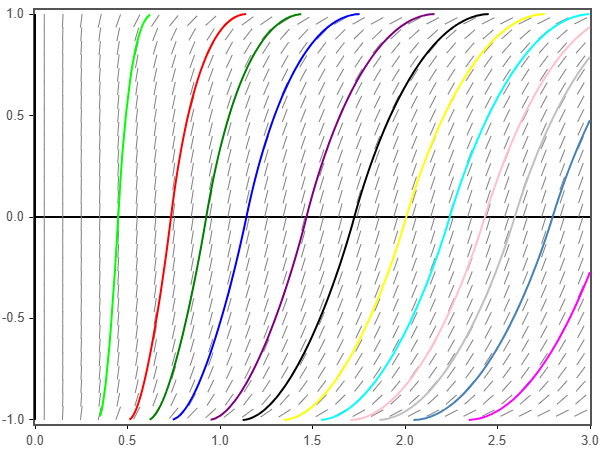}
 \caption{\small Graphs of solutions to~\eqref{eq-CPr=n}
 for $(n,\beta)=(3,\tanh)$ (left), and $(n,\beta)=(3,\coth)$ (right).}
 \label{fig-plotr=n1}
\end{figure}

In the next propositions, we establish some fundamental properties of solutions
to~\eqref{eq-CPr=n} (see Fig.~\ref{fig-plotr=n1}).
To accomplish that, it will be convenient to consider the cases $n$ odd and $n$ even separately.

\begin{proposition}\label{prop-n01}
Given an odd integer $n\ge 3$, and a point $(s_0,y_0)\in\Omega$,
let $\tau$ be the solution of the Cauchy problem~\eqref{eq-CPr=n}.
Then, there exist  $s_{\min}=s_{\min}(\tau)$ and $s_{\max}=s_{\max}(\tau)$
in $I$ such that:
\begin{itemize}[parsep=1ex]
\item $(s_{\min}(\tau),s_{\max}(\tau))$ is the maximal interval of definition of $\tau$,
\item  the following equalities hold:
\[
\lim_{s\to s_{\max}}\tau(s)=1 \quad\text{and}\quad \lim_{s\to s_{\min}}\tau(s)=-1.
\]
\end{itemize}
\end{proposition}

\begin{proof}
  Since $n$ is odd, any solution to~\eqref{eq-CPr=n} is  increasing.
  Also, in the case $\beta=\coth$, arguing just as in the proof of
  Proposition~\ref{prop-uniquenesshyperbolicsolutions}, we conclude that
  the graph of $\tau$ does not intersect the $y$-axis.
  Therefore, it suffices to prove that $\tau$ has no horizontal asymptotic lines.

  Assume, by contradiction, that $[s_0,+\infty)$ is contained in the maximal interval
  on which $\tau$ is defined. Then, there exists $\mathcal L\in [\tau(s_0),1]$ such that
  $\tau(s)\to\mathcal L$ as $s\to+\infty$. In this case, we necessarily have
  \begin{equation} \label{eq-limitstau}
  \lim_{s\to+\infty}\tau'(s)=\lim_{s\to+\infty}\tau''(s)=0.
  \end{equation}

  Considering~\eqref{eq-taur=n} and the fist limit in~\eqref{eq-limitstau}, we easily conclude that $\mathcal L=1.$
  Now, assuming $s$ sufficiently large so that $\tau(s)\ne 0,$ we get from a direct computation that
  \begin{equation}  \label{eq-tausecondderivative}
  \tau''(s)=n\beta^{n-2}(s)\left(-\frac{\beta^n(s)}{\tau^{\frac{n-2}{n}}(s)}+(n-1)\sqrt{1-\tau^{\frac2n}(s)}\beta'(s)\right),
  \end{equation}
  and then
 $$\lim_{s\to+\infty}\tau''(s)=\left\{
\begin{array}{ll}
-\infty & \text{if}\,\,\,\beta=\tan_0, \\[1ex]
-n      & \text{if}\,\,\,\beta\ne\tan_0,
\end{array}
\right.$$
which contradicts the second equality in~\eqref{eq-limitstau}.

In the same way we prove that, in the cases $\beta=\tan_\epsilon$
or $\beta=1$, there is no $\mathcal L\in [-1,\tau(s_0)]$
such that $\lim_{s\to-\infty}\tau(s)=\mathcal L$. This finishes the proof.
\end{proof}

\begin{proposition} \label{prop-unboundedphi}
Let $n\ge 3$ and $\tau$ be as in Proposition~{\rm\ref{prop-n01}}.
Then, the function
\[
\phi(s):=\int_{s_0}^{s}\frac{\rho(u)}{\sqrt{1-\rho^2(u)}}du, \quad \rho=\tau^{1/n},
\]
satisfies
\[
\lim_{s\to s_{\max}}\phi(s)=\lim_{s\to s_{\min}}\phi(s)=+\infty.
\]
\end{proposition}

\begin{proof}
We follow closely the final part of the proof of Theorem~\ref{th-existence-hyperbolic}-(i).
Since
$$\lim_{s\to s_{\max}}\tau(s)=1,$$
we have from~\eqref{eq-taur=n} that
$\lim_{s\to s_{\max}}\tau'(s)=0.$ So, we can extend $\tau$ smoothly to $[s_{\max},+\infty)$
by setting $\tau(s)=1$ for all $s\in[s_{\max},+\infty).$ Considering this extension,
we have that $\rho^2(s_{\max})=1$, $(\rho^2)'(s_{\max})=0$ and, from \eqref{eq-tausecondderivative},
that  $(\rho^2)''(s_{\max})<0.$
In particular, the second order Taylor's formula of $\rho^2$ around $s_{\max}$ reads as
\begin{equation} \label{eq-taylor}
\rho^2(s)=1+\frac12(\rho^2)''(s_{\max})(s-s_{\max})^2+f(s), \quad \lim_{s\to s_{\max}}\frac{f(s)}{(s-s_{\max})^2}=0.
\end{equation}

Setting $a:=(\rho^2)''(s_{\max})/2,$ we have from \eqref{eq-taylor} that
\begin{equation} \label{eq-limitsmax01}
\lim_{s\to s_{\max}}\frac{\sqrt{1-\rho^2(s)}}{|s-s_{\max}|}=\sqrt{-a}>0.
\end{equation}
Also, considering~\eqref{eq-tausecondderivative}, a direct computation
gives that $(\rho^2)'''(s_{\max})$ is well defined, that is, it is finite.
Then, applying the l'Hôpital's rule, we have from~\eqref{eq-taylor} that
\[
\lim_{s\to s_{\max}}\frac{f(s)}{(s-s_{\max})^3}=\lim_{s\to s_{\max}}\frac{(\rho^2)'(s)-2a(s-s_{\max})}{3(s-s_{\max})^2}=
\lim_{s\to s_{\max}}\frac{(\rho^2)''(s)-2a}{6(s-s_{\max})}\,,
\]
so that
\begin{equation} \label{eq-lHopital}
\lim_{s\to s_{\max}}\frac{f(s)}{(s-s_{\max})^3}=\frac{(\rho^2)'''(s_{\max})}{6}\ne\pm\infty.
\end{equation}

Proceeding as  in the proof of Theorem~\ref{th-existence-hyperbolic}-(i), one can verify that
the function
\[
g(s):=\frac{1}{\sqrt{1-\rho^2(s)}}-\frac{1}{\sqrt{-a}(s_{\max}-s)}
\]
is well defined and bounded in a neighborhood of $s_{\max}.$
\begin{comment}
With this purpose,
we first observe that, from \eqref{eq-taylor}, we have
\[
f(s)=(\sqrt{-a}(s_{\max}-s)-\sqrt{1-\rho^2(s)})(\sqrt{-a}(s_{\max}-s)+\sqrt{1-\rho^2(s)}).
\]
Therefore, we can write $g$ as
\begin{eqnarray*}
g(s) &=& \frac{\sqrt{-a}(s_{\max}-s)-\sqrt{1-\rho^2(s)}}{\sqrt{-a(1-\rho^2(s))}(s_{\max}-s)}\\
            &=& \frac{1}{\sqrt{-a}\sqrt{1-\rho^2(s)}}\frac{1}{(s_{\max}-s)}\frac{f}{\sqrt{1-\rho^2(s)}+\sqrt{-a}(s_{\max}-s)}\\
            &=& \frac{s_{\max}-s}{\sqrt{-a}\sqrt{1-\rho^2(s)}}\frac{f}{(s_{\max}-s)^3}\frac{1}{\frac{\sqrt{1-\rho^2(s)}}{s_{\max}-s}+\sqrt{-a}}\,\cdot
\end{eqnarray*}

This last equality, together with \eqref{eq-limitsmax01} and \eqref{eq-lHopital}, gives that
$\lim_{s\to s_{\max}}g(s)$ is well defined and finite, which proves our claim.
\end{comment}

To conclude the proof, fix a small $\delta>0$ such that $s_0<s_{\max}-\delta.$
Then, for all $s\in (s_{\max}-\delta,s_{\max}),$ one has
\begin{eqnarray}
\phi(s) &=& \int_{s_0}^{s_{\max}-\delta}\frac{\rho(u)}{\sqrt{1-\rho^2(u)}}du+\int_{s_{\max}-\delta}^s\frac{\rho(u)}{\sqrt{1-\rho^2(u)}}du  \nonumber\\[1ex]
        &\ge& \int_{s_0}^{s_{\max}-\delta}\frac{\rho(u)}{\sqrt{1-\rho^2(u)}}du+\rho(s_{\max}-\delta)\int_{s_{\max}-\delta}^s\frac{1}{\sqrt{1-\rho^2(u)}}du. \label{eq-forphi}
\end{eqnarray}
But, by the definition of $g,$
\begin{eqnarray}
\int_{s_{\max}-\delta}^s\frac{1}{\sqrt{1-\rho^2(u)}} &=& \int_{s_{\max}-\delta}^sg(u)du+\frac{1}{\sqrt{-a}}\int_{s_{\max}-\delta}^s\frac{du}{s_{\max}-u} \nonumber\\[1ex]
&=&\int_{s_{\max}-\delta}^sg(u)du+\frac{1}{\sqrt{-a}}\log\left(\frac{\delta}{s_{\max}-s}\right),  \label{eq-forphi2}
\end{eqnarray}
which implies that $\phi(s)\to+\infty$ as $s\to s_{\max}$.
The proof that $\phi(s)\to+\infty$ as $s\to s_{\min}$ is analogous.
\end{proof}

In the two preceding propositions, the parts regarding  the limits of
$\tau$ and $\phi$ as $s\to s_{\max}$ have analogous versions for $n$ even.
To establish that, we have just to consider the Cauchy problem~\eqref{eq-CPr=n}
on $\Omega_+:=[0,+\infty)\times[0,1)$. Indeed, in this case, we have from~\eqref{eq-taur=n}
that the solutions to~\eqref{eq-CPr=n} are all increasing. Thus,
we can argue as  in the proofs of Propositions~\ref{prop-n01} and~\ref{prop-unboundedphi}
to obtain the following results.

\begin{proposition}\label{prop-neven01}
Given an even integer $n\ge 2$, and a point $(s_0,y_0)\in\Omega_+$,
let $\tau$ be the solution of the Cauchy problem~\eqref{eq-CPr=n} in $\Omega_+$.
Then, there exists $s_{\max}=s_{\max}(\tau)$
in $I$, $s_{\max}>0$,  such that (see Fig.~{\rm\ref{fig-plotsreven}})
$$\displaystyle\lim_{s\to s_{\max}}\tau(s)=1.$$
\end{proposition}

\begin{proposition} \label{prop-nevenunboundedphi}
Let $n\ge 2$ and $\tau$ be as in Proposition~{\rm\ref{prop-neven01}}.
Then, the function
\[
\phi(s):=\int_{s_0}^{s}\frac{\rho(u)}{\sqrt{1-\rho^2(u)}}du, \quad \rho=\tau^{1/n},
\]
satisfies
$\displaystyle\lim_{s\to s_{\max}}\phi(s)=+\infty.$
\end{proposition}

\begin{figure}
 \centering
 \includegraphics[scale=.29]{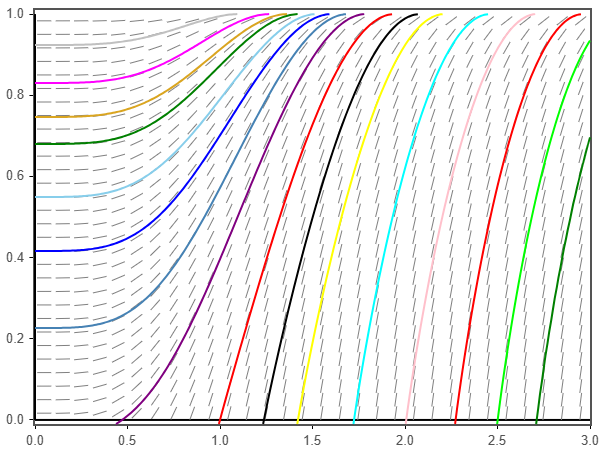}
 \hfill
 \includegraphics[scale=.29]{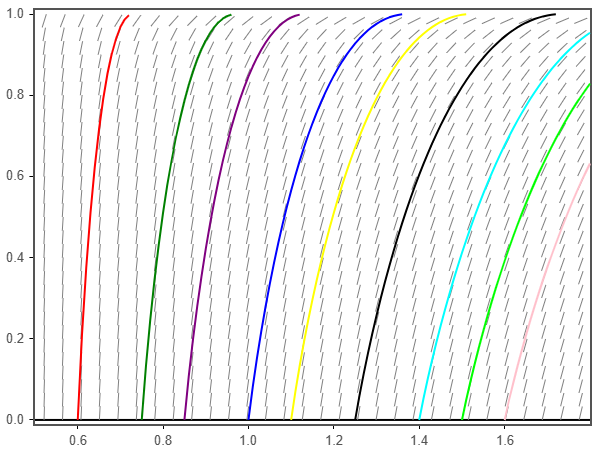}
 \caption{\small Graphs of solutions to~\eqref{eq-CPr=n}
 on $\Omega_+$ for $(n,\beta)=(4,\tanh)$ (left), and $(n,\beta)=(4,\coth)$ (right).}
 \label{fig-plotsreven}
\end{figure}

\subsection{Rotational translators to $n$-MCF in $\mathbb Q^n_{\epsilon}\times\mathbb R$}
Let us consider now rotational $n$-translators in
$\mathbb Q_{\epsilon}^n\times\mathbb R$. In this case,
$\beta=\tan_\epsilon$, so that~\eqref{eq-taur=n} becomes
\begin{equation}\label{eq-taur=nrotational}
\tau'(s)=n\sqrt{1-\tau^{\frac 2n}(s)}\tan_{\epsilon}^{n-1}(s),
\end{equation}
whose associated Cauchy problem is:
\begin{equation} \label{eq-CPr=n-rotational}
\left\{
\begin{array}{l}
y'(s)=n\sqrt{1-y^{\frac 2n}(s)}\tan_{\epsilon}^{n-1}(s)\\[1ex]
y(s_0)=y_0,
\end{array}
\right.
\end{equation}
where $(s_0,y_0)\in\Omega:=(-\infty,+\infty)\times(-1,1).$

\begin{figure}[hbt]
 \centering
 \includegraphics[scale=.35]{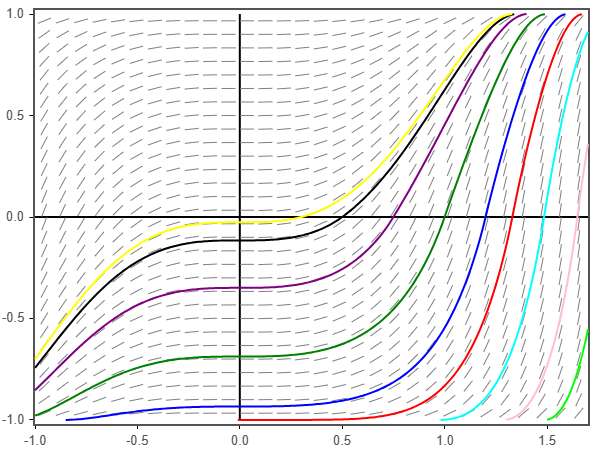}
 \caption{\small Graphs of solutions to~\eqref{eq-CPr=n-rotational}
 for $(n,\beta)=(3,\tan_0)$. The red curve is the graph of the solution $\tau_*$ in the statement
 of Proposition~\ref{prop-taustar}.}
 \label{fig-plotn=r-rotational}
\end{figure}

Next, we establish a special property of the solutions to~\eqref{eq-CPr=n-rotational}
when $n$ is odd.

\begin{proposition} \label{prop-taustar}
Let $n\ge 3$ be an odd integer. Given $s_0>0,$ let $\tau_{s_0}$ be the solution to~\eqref{eq-CPr=n-rotational}
with initial condition  $y(s_0)=0.$ Then,  there exists $s_*>0$ with the following properties (see Fig.~{\rm\ref{fig-plotn=r-rotational}}):
\begin{itemize}[parsep=1ex]
\item[\rm i)] $s_{\min}(\tau)>0$ if and only if $s_0>s_*.$
\item[\rm ii)] The solution $\tau_*$ to~\eqref{eq-CPr=n-rotational} such that $\tau_*(s_*)=0$ satisfies $s_{\min}(\tau_*)=0.$
\end{itemize}
As a consequence, the following holds:
\begin{itemize}[parsep=1ex]
\item[\rm iii)] For any $\mu>0,$ there exists a solution $\tau$ to~\eqref{eq-CPr=n-rotational}
such that $s_{\min}(\tau)=\mu.$
\end{itemize}
\end{proposition}

\begin{proof}
  Consider the set $$\Lambda:=\{s_0>0\,;\, s_{\min}(\tau_{s_0})\le 0\}$$
  and observe that any $s_0>0$ sufficiently close to $0$ is a point of $\Lambda$.
  In addition, since graphs of distinct solutions do not intersect, if
  $\bar s_0\in\R-\Lambda$, then $[\bar s_0,+\infty)\subset\R-\Lambda$.
  Therefore, either occurs: $\Lambda=(0,s_*)$ for some $s_*>0$ or
  $\Lambda=(0,+\infty)$ .

  Assume, by contradiction, that $\Lambda=(0,+\infty)$. For $s_0>1,$
  one has that $\tau_{s_0}(s_0-1)$ is negative and
  stays bounded away from $-1$ as $s_0$ goes to infinity, since we
  are assuming $s_{\min}(\tau)\le 0$ and, by Proposition~\ref{prop-n01}, $\tau$ is
  increasing with limit $-1$  as $s\to s_{\min}(\tau)$. Consequently,
  \begin{equation} \label{eq-limitinfinity}
  \lim_{s_0\to\infty}\tau'(s_0-1)=\lim_{s_0\to\infty}(n\sqrt{1-(\tau(s_0-1))^{2/n}}\tan_\epsilon^{n-1}(s_0-1))=+\infty.
  \end{equation}

  For all $s\in(0,s_0)$, we have that $\tau_{s_0}(s)<0$, which yields
  $(\tau_{s_0}(s))^{\frac{n-2}{n}}<0$, since we are assuming $n$ odd. In addition,
  for $\beta=\tan_\epsilon$, one has $\beta,\beta'>0$ on $(0,+\infty)$. Considering
  these facts and equality~\eqref{eq-tausecondderivative}, we conclude that
  $\tau_{s_0}''>0$ on $(0,s_0).$ In particular,
  $\tau'(s_0-1)<\tau'(s)$ for all $s\in(s_0-1,s_0)\subset(0,s_0)$. Also, from~\eqref{eq-limitinfinity},
  we can assume $s_0$ sufficiently large, so that $\tau'(s_0-1)>1.$ Then, we have
  \[
  1\ge\tau(s_0)-\tau(0)=\int_{0}^{s_0}\tau'(s)ds\ge\int_{s_0-1}^{s_0}\tau'(s)ds\ge\tau'(s_0-1)>1,
  \]
  which is a contradiction. Therefore, $\Lambda=(0,s_*)$ for some $s_*>0$.

  Now, since $s_*:=\sup\Lambda$, it is clear that it
  satisfies (i) and (ii). Assertion (iii)
  follows from (i)-(ii) and the fact that the graphs of
  solutions to~\eqref{eq-CPr=n-rotational} foliate $\Omega$.
  \end{proof}

  \begin{figure}[hbt]
 \centering
  \includegraphics[scale=.25]{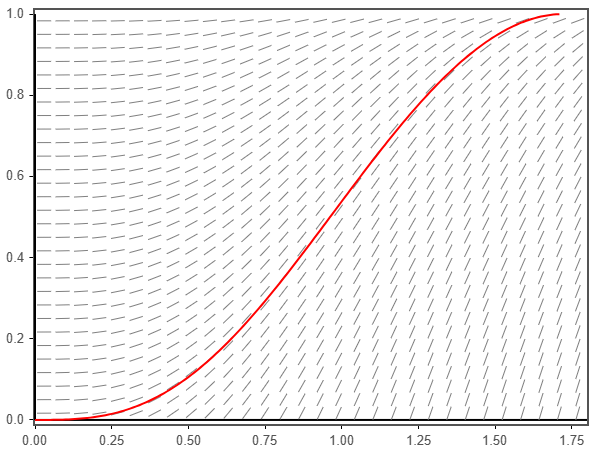}
  \hspace{1cm}
 \includegraphics[scale=.7]{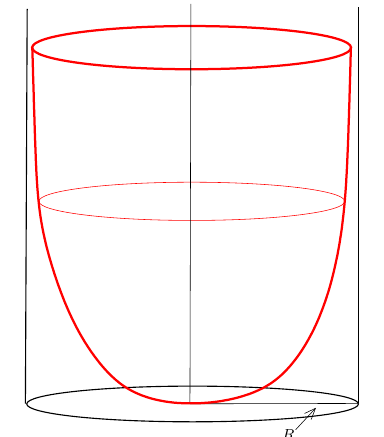}
 \caption{\small The   $n$-bowl soliton (right) and the solution to~\eqref{eq-CPr=n-rotational} that
 generates it (left).}
 \label{fig-n-bowl}
\end{figure}

  \begin{proposition} \label{prop-noblown}
Given an integer $n\ge 2$, let $\tau_0$ be  the solution to~\eqref{eq-CPr=n-rotational} satisfying
$\tau_0(0)=0.$ Then, for $\rho_0=\tau_0^{1/n},$ one has that the limits
\[
L_1:=\lim_{s\rightarrow 0}(\cot_\epsilon(s)\rho_0(s))
\quad\text{and}\quad L_2:=\lim_{s\rightarrow 0}\rho_0'(s)
\]
are both finite.
\end{proposition}
\begin{proof}
Analogous to the proof of Proposition \ref{prop-noblow}.
\end{proof}

\begin{theorem}\label{th-nrot}
Let $n\geq 3$ be an odd integer. Then, the  following assertions hold:
\begin{itemize}[parsep=1ex]
\item[\rm i)] There exists a rotational strictly convex $n$-translator $\Sigma_0$
  in $\qr$ (to be called the $n$-\emph{bowl soliton}) which is a vertical graph
  over an open ball $B_{R}(o)\subset\Pi_0$ of radius $R>0$. Moreover, $\Sigma_0$ is
  contained in the closed half-space $\q_\epsilon^n\times[0,+\infty)$ with unbounded height, and
  is asymptotic to $\partial B_R(0)\times\R$ (Fig.~{\rm\ref{fig-n-bowl}}).

%%%%%%%%%%%%%%%%%%%%%%%%%%%%%%%%%%%%%%%%%%%%%%%%%%%%%%%%%%%%%%%%%%%%%%%%%%%%%%%%%%%%%%%%%%%%%%%%%%%%%%%

\item[\rm ii)] There exists a one-parameter family $\mathscr K_n=\{\mathcal K_\lambda\ ;\ 0<\lambda<1\}$ of
properly embedded rotational cones in $\qr$ with vertex
at $o\in\Pi_0:=\q_\epsilon^n\times\{0\}$ (to be called $n$-\emph{translating cones}),
which are all $n$-translators.
Any $\mathcal K_{\lambda}\in\mathscr K_n$ is the union of two vertical graphs
$\mathcal G_{\pm\lambda}$  defined on  open balls
$B_{R_\pm}(o)\subset\Pi_0$  of  radiuses  $R_{\pm}=R_{\pm}(\epsilon,n,\lambda),$ $R_{-}>R_+>0,$ both
with a conical singularity at $o$
%intersecting the axis of rotation  at $o$ (at which $\mathcal G_i$ has a conical  singularity) with angle function $\Theta=\sqrt{1-\lambda^2}$
(Fig.~{\rm\ref{fig-n-cone}}).
In addition,  $\mathcal K_{\lambda}$ has the following properties:
%%%%%%%%%%%
\subitem $\bullet$ It is contained in a half space and its height function is unbounded from above.
%%%%%%%%%%%
\subitem $\bullet$ Its graphs $\mathcal G_{\pm\lambda}$ are vertically asymptotic to the vertical cylinders $\partial B_{R_\pm}(o)\times\mathbb R,$
respectively, and their angle functions $\theta_{\pm\lambda}$ satisfy $$\displaystyle \lim_{s\to 0}\theta_{\pm\lambda}^2(s)=1-\lambda^2.$$
%%%%%%%%%%%
\subitem $\bullet$ It is $C^2$-smooth, except at the vertex $o,$ where it is singular, and  on its $(n-1)$-sphere of minimal height,
where it is $C^2$-singular.

%%%%%%%%%%%%%%%%%%%%%%%%%%%%%%%%%%%%%%%%%%%%%%%%%%%%%%%%%%%%%%%%%%%%%%%%%%%%%%%%%%%%%%%%%%%%%%%%%%%%%%%%%%%%%%%%%%%%%%%%%%%%%%%%%%%%%%%%%%%%%%%%%%%%%%

\item[\rm iii)] There exists a one-parameter family
  $\mathscr C_n=\{\Sigma_\mu\,;\, \mu\ge0\}$ of
  properly embedded rotational $n$-translators in $\qr$
  (to be called $n$-\emph{grim reapers}).
  Each $\Sigma_\mu\in\mathscr C_n$ is a  vertical graph
  over an annulus $B_{R}(o)-B_{\mu}(o)$,
  $R=R(\epsilon,n,\mu)>\mu,$  and has the following properties (Fig.~{\rm\ref{fig-ngrimreaper}}):
  %%%%%%%%%%%
\subitem $\bullet$ It is contained in a half space, and its height function
is unbounded from above.
%%%%%%%%%%%
\subitem $\bullet$ It is vertically asymptotic to the  cylinders
  $\partial B_{R}(o)\times\mathbb R$ and $\partial B_{\mu}(o)\times\mathbb R,$
  where the latter reduces to a vertical line for $\mu=0.$
%%%%%%%%%%%
\subitem $\bullet$ It is $C^2$-singular along its $(n-1)$-sphere of minimal height.
\end{itemize}
\end{theorem}

\begin{figure}[ht]
 \centering
  \includegraphics[scale=.26]{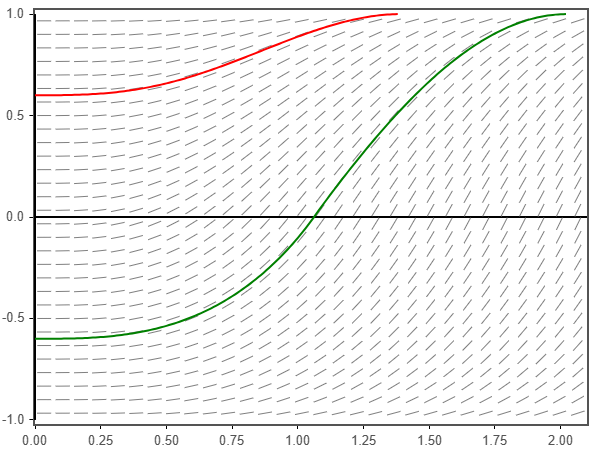}
  \hspace{1cm}
 \includegraphics[scale=1]{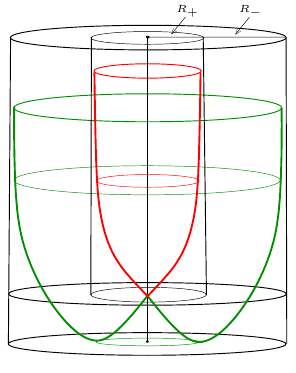}
 \caption{\small An  $n$-translating cone (right) and the solutions to~\eqref{eq-CPr=n-rotational} that
 generate it (left).}
 \label{fig-n-cone}
\end{figure}

\begin{proof}
Given $\lambda\in[0,1),$ let $\tau_{\pm\lambda}\colon[0,s_{\max}(\tau_{\pm\lambda}))\to\R$ be
the solutions to~\eqref{eq-CPr=n-rotational}
satisfying  $\tau_{\pm\lambda}(0)=\pm\lambda.$ Setting
$\rho_{\pm\lambda}:=\tau_{\pm\lambda}^{1/n},$ it follows from  Proposition~\ref{prop-main}
that the rotational graphs $\mathcal G_{\pm\lambda}$ with $\rho$-functions $\rho_{\pm\lambda}$
and height functions
\[
\phi_{\pm\lambda}(s)=\int_{0}^{s}\frac{\rho_{\pm\lambda}(u)}{\sqrt{1-\rho_{\pm\lambda}^2(u)}}du, \,\,\, s\in[0,s_{\max}(\tau_{\pm\lambda})),
\]
are both $n$-translators  in $\qr.$ By Proposition~\ref{prop-n01},
$\rho_{\pm\lambda}(s)\to 1$ as $s\to s_{\max}(\tau_{\pm\lambda})$, which implies that
$\phi_{\pm\lambda}'(s)\to+\infty$ as $s\to s_{\max}(\tau_{\pm\lambda}).$
In addition, Proposition \ref{prop-unboundedphi} gives that $\phi_{\pm\lambda}$ are both unbounded.
Therefore, setting $R_{\pm}:=s_{\max}(\tau_{\pm\lambda})$ and $\ell:=\{o\}\times\R$  for the axis of rotation of
$\mathcal G_{\pm\lambda},$  we have that $\mathcal G_{\pm\lambda}$ are  asymptotic to
$\partial B_{R_{\pm}}(o)\times\R,$ respectively. Notice that $R_+<R_-.$ Otherwise, the graphs
$\mathcal G_{\pm\lambda}$ would intersect.

For $\lambda=0,$ it follows from Proposition \ref{prop-noblown} and
equalities \eqref{eq-principalcurvatures} (for $\alpha=\cot_\epsilon$)
that $\Sigma_0:=\mathcal G_0$ is $C^2$ and strictly convex, which proves (i).

Analogously, for $\lambda>0,$  $\mathcal G_{\lambda}$ is $C^2$-smooth
on $\mathcal G_\lambda-\{o\}.$ Regarding $\mathcal G_{-\lambda}$,
there exists $s_0>0$ such that  $\rho_{-\lambda}(s_0)=0.$ Hence, $\phi_{-\lambda}$
is decreasing in $(0,s_0)$ and increasing in $(s_0,s_{\max}(\tau_{-\lambda})).$
Also, $\rho_\lambda'(s_0)=+\infty,$ since $\tau_{-\lambda}(s_0)=0.$ So,
$\mathcal G_{-\lambda}$ is $C^2$-singular on its $(n-1)$-sphere $S$ of minimal height.
Clearly, $\mathcal G_{-\lambda}$ is $C^2$ on the complement of $S\cup\{o\}.$

Now, observe that
$\phi'_{\pm \lambda}(0)={\pm\lambda}/{\sqrt{1-\lambda^2}},$
which implies that the angle functions $\theta_{\pm\lambda}$ of $\mathcal G_{\pm\lambda}$
satisfy (cf.~\eqref{eq-thetasquared}):
\[
\lim_{s\to 0}\theta_{\pm\lambda}^2(s)=\frac{1}{1+(\phi_{\pm\lambda}'(0))^2}=1-\lambda^2<1,
\]
so that $o$  is a conical singular point of both graphs $\mathcal G_{\pm\lambda}.$

It follows from the above considerations that, for each $\lambda >0,$ the cone
$$\mathcal K_\lambda:=\mathcal G_\lambda\cup\mathcal G_{-\lambda}$$ is
an $n$-translator, as stated. This proves (ii).

Now, to prove (iii), choose $\mu\ge 0$. By Proposition~\ref{prop-taustar}-(iii), there exists
a solution $\tau_\mu$ to~\eqref{eq-CPr=n-rotational} such that $s_{\min}(\tau_\mu)=\mu.$
Setting $s_0>0$ for the point at which $\rho_\mu:=\tau_\mu^{1/n}$ vanishes, and
defining $R:=s_{\max}(\tau_\mu),$ we conclude as above that
\[
\phi_\mu:=\int_{s_0}^s\frac{\rho_\mu(u)}{\sqrt{1-\rho_\mu^2(u)}}du, \,\,\, s\in(\mu,R),
\]
is the height function of a rotational $n$-translator $\Sigma_\mu$  in $\qr.$ Furthermore,
by Proposition~\ref{prop-unboundedphi}, $\phi_\mu$ is unbounded above and $\Sigma_\mu$ is asymptotic to both
$B_\mu(o)$ and $B_R(o).$ Analogously to the $n$-translating cones,
$\Sigma_\mu$ is $C^2$-singular on its $(n-1)$-sphere of minimal height.
This shows (iii) and finishes the proof.
\end{proof}

Taking into account Propositions~\ref{prop-neven01},~\ref{prop-nevenunboundedphi},
and~\ref{prop-noblown}, one can mimic the proof of Theorem~\ref{th-nrot}
and then get the following result.

\begin{theorem}\label{th-nevenrot}
Let $n\geq 2$ be an even integer. Then, the  following assertions hold:
\begin{itemize}[parsep=1ex]
\item[\rm i)] There exists a rotational strictly convex $n$-translator $\Sigma_0$
  in $\qr$ (to be called the $n$-\emph{bowl soliton}) which is a vertical graph
  over an open ball $B_{R}(o)\subset\Pi_0$ of radius $R>0$. Moreover, $\Sigma_0$ is
  contained in the closed half-space $\q_\epsilon^n\times[0,+\infty)$ with unbounded height, and
  is asymptotic to $\partial B_R(0)\times\R$ (Fig.~{\rm\ref{fig-n-bowl}}).

%%%%%%%%%%%%%%%%%%%%%%%%%%%%%%%%%%%%%%%%%%%%%%%%%%%%%%%%%%%%%%%%%%%%%%%%%%%%%%%%%%%%%%%%%%%%%%%%%%%

\item[\rm ii)] There exists a one-parameter family $\mathscr K_n=\{\mathcal K_\lambda\ ;\ 0<\lambda<1\}$ of
properly embedded rotational half-cones  in $\qr$ with vertex
at $o\in\q_\epsilon^n\times\{0\}$ (to be called \emph{peaked} $n$-\emph{bowl soliton}),
which are all $n$-translators.
Any $\mathcal K_{\lambda}\in\mathscr K_n$ is a vertical graph
defined on an open ball
$B_{R}(o)\subset\q_\epsilon^n\times\{0\}$  of  radius  $R=R(\epsilon,n,\lambda)$
with a conical singularity at $o$.
%intersecting the axis of rotation  at $o$ (at which $\mathcal G_i$ has a conical  singularity) with angle function $\Theta=\sqrt{1-\lambda^2}$
In addition,  $\mathcal K_{\lambda}$ has the following properties (Fig.~{\rm\ref{fig-peaked}}):
%%%%%%%%%%%
\subitem $\bullet$ It is contained in a half space and its height function is unbounded from above.
%%%%%%%%%%%
\subitem $\bullet$ It is vertically asymptotic to the cylinder $\partial B_{R}(o)\times\mathbb R$,
and its angle function $\theta_\lambda$ satisfies $$\displaystyle \lim_{s\to 0}\theta_{\lambda}^2(s)=1-\lambda^2.$$
%%%%%%%%%%%%%%%%%%%%%%%%%%%%%%%%%%%%%%%%%%%%%%%%%%%%%%%%%%%%%%%%%%%%%%%%%%%%%%%%%%%%%%%%%%%%%%%%%%%%%%%%%%%%%%%%%%%%%%%%%%%%%%%%%%%%%%%%%%%%%%%%%%%%%%

\item[\rm iii)] There exists a one-parameter family
  $\mathscr G_n=\{\Sigma_\mu\,;\, \mu>0\}$ of
  properly embedded rotational $n$-translators  in $\qr$
  (to be called $n$-\emph{grim reapers}) with nonempty boundary.
  Each $\Sigma_\mu\in\mathscr G_n$ is a  vertical graph
  over an annulus $B_{R}(o)-B_{\mu}(o)$,
  $R=R(\epsilon,n,\mu)>\mu,$  and has the following properties
  (Fig.~{\rm\ref{fig-peaked}}):
  %%%%%%%%%%%
\subitem $\bullet$ It is contained in a half space, and its height function
is unbounded from above.
%%%%%%%%%%%
\subitem $\bullet$ It is vertically asymptotic to the  cylinder
  $\partial B_{R}(o)\times\mathbb R$ and  tangent to $\h^n\times\{0\}$ along its boundary
   $\partial\Sigma_\mu=\partial B_{\mu}(o).$
%%%%%%%%%
\end{itemize}
\end{theorem}

\begin{figure}[h]
 \centering
 \includegraphics[scale=.27]{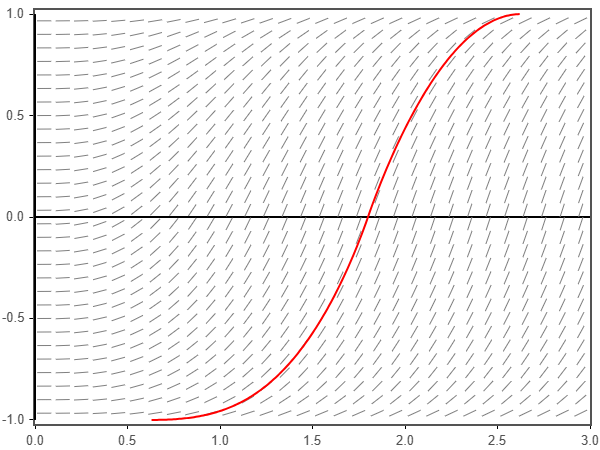}
 \hspace{1cm}
  \includegraphics[scale=.8]{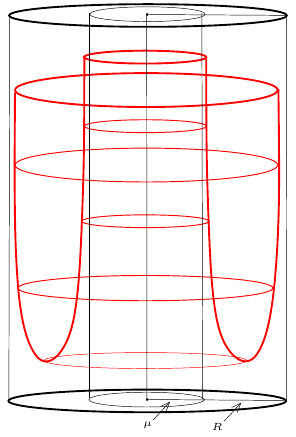}
 \caption{\small An  $n$-grim reaper (right) and the solution to~\eqref{eq-CPr=n-rotational} that
 generates it (left).}
 \label{fig-ngrimreaper}
\end{figure}

\begin{remark} \label{rem-nconvergence}
Given $s_0>0$, let $\tau_{s_0}$ be the solution to~\eqref{eq-CPr=n-rotational}
with initial condition $y(s_0)=0.$ By the continuity of solutions with respect to
initial conditions, we have that the restriction of $\tau_{s_0}$ to the interval where it
is positive converges uniformly to the solution $\tau_0$ satisfying $\tau_0(0)=0$ as $s_0\to 0$.
Consequently, the corresponding subset of the grim reaper converges (in compact sets)
to the $n$-bowl soliton $\Sigma_0$ (compare with Remark~\ref{rem-catenoids}).
\end{remark}

\begin{figure}[htb]
 \centering
  \includegraphics[width=4cm]{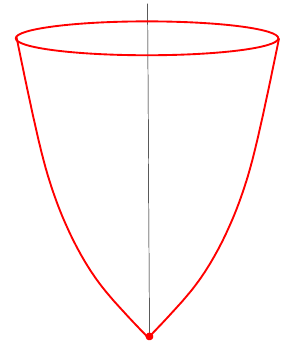}
  \hspace{1cm}
  \includegraphics[width=4.2cm]{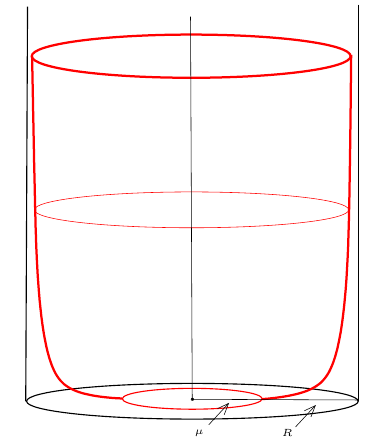}
 \caption{\small Peaked $n$-bowl soliton (left) and $n$-grim reaper with boundary (right).}
 \label{fig-peaked}
\end{figure}

In view of the results of Propositions~\ref{prop-n01}--\ref{prop-noblown}, as well as the proofs
of Theorems~\ref{th-nrot} and~\ref{th-nevenrot}, we see that all solutions to~\eqref{eq-CPr=n-rotational}
have been considered for obtaining the $n$-translators in these theorems. Therefore, as stated below, we have
for rotational $n$-translators a uniqueness result which
is analogous to that of Proposition~\ref{prop-uniquenessgraphrotational}.

\begin{proposition} \label{prop-uniquenessgraphrotational-n}
Let $\Sigma$ be a connected rotational $n$-translator in $\qr$
which is a vertical graph over an open set of \,$\q_\epsilon^n.$
Then, $\Sigma$ is an open set of one of the following hypersurfaces:
an $n$-bowl soliton, an $n$-translating cone,  an
$n$-grim reaper, or a peaked $n$-bowl soliton.
\end{proposition}

\begin{figure}[hbt]
 \centering
  \includegraphics[scale=.35]{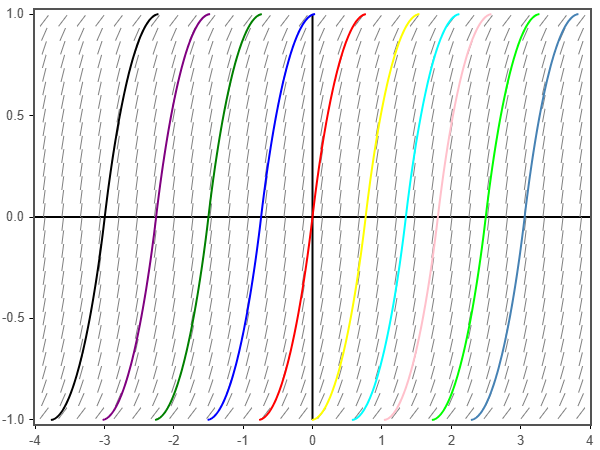}
 \caption{\small Graphs of solutions to~\eqref{eq-IVPparabolicr=n} for $n=3$.}
 \label{fig-parabolicgraphs}
\end{figure}

\subsection{Parabolic translators to $n$-MCF in $\mathbb H^n\times\mathbb R$}
We shall consider now parabolic $n$-translators in
$\h^n\times\mathbb R$, i.e., those invariant by horizontal parabolic translations.
In this case, \eqref{eq-taur=n} becomes
\begin{equation}\label{eq-r=nparabolic}
\tau'(s)=n\sqrt{1-\tau^{\frac 2n}(s)},
\end{equation}
and the associated Cauchy problem is
\begin{equation} \label{eq-IVPparabolicr=n}
\left\{
\begin{array}{l}
y'(s)=n\sqrt{1-y^{\frac 2n}(s)}\\[1ex]
y(s_0)=y_0, \,\,\, (s_0,y_0)\in\Omega,
\end{array}
\right.
\end{equation}
where $\Omega:=(-\infty,+\infty)\times(-1,1)$.

Considering Propositions~\ref{prop-n01}--\ref{prop-nevenunboundedphi} in the parabolic setting,
that is, for $\beta=1$, we obtain the following
result, whose proof is analogous to the one given for Theorem~\ref{th-nrot}.

\begin{theorem} \label{th-nparabolic}
Let $n\ge 2$ be an integer. Then, the following assertions hold:
\begin{itemize}[parsep=1ex]
\item[\rm i)] If $n$ is odd,  there exists a parabolic $n$-translator $\Sigma$ in $\mathbb H^n\times\mathbb R$
(to be called the \emph{parabolic $n$-grim reaper}) which has the following
properties (Fig.~{\rm\ref{fig-Kparabolicgrimreapers}}):
%%%%%%%%%%%%%%%%%%%
\subitem $\bullet$ $\Sigma$ is a vertical graph over the open region of $\h^n$ bounded by two
parallel horospheres $\mathcal H_{\pm}$. In particular, $\Sigma$ is homeomorphic to $\R^n.$
%%%%%%%%%%%%%%%%%%%%%%%%%%%%%%%%%%%%%%%%%%%%%%%%%%%%%%%%%%%%%
\subitem $\bullet$ The height of $\Sigma$ is bounded from below and unbounded from above, and
$\Sigma$ is vertically asymptotic  to both cylinders $\mathcal H_{\pm}\times\R.$
%%%%%%%%%%%%%%%%%%%%%%%%%%%%%%%%%%%%%%%%%%%%%%%%%%%%%%%%%%%%%%%%%%%%%%%%%%%%
\subitem $\bullet$ $\Sigma$ is $C^2$-singular along the horosphere of minimal height.
%%%%%%%%%%%%%%%%%%%%%%%%%%%%%%%%%%%%%%%%%%%%%%%%%%%%%%%%%%%%%
%%%%%%%%%%%%%%%%%%%%%%%%%%%%%%%%%%%%%%%%%%%%%%%%%%%%%%%%%%%%%
\item[\rm ii)] If $n$ is even,
there exists a parabolic $n$-translator $\Sigma$ in $\mathbb H^n\times\mathbb R$
with nonempty boundary (to be called the \emph{parabolic $n$-grim reaper}) which
has the following properties (Fig.~{\rm\ref{fig-Kparabolicgrimreapers}}):
%%%%%%%%%%%%%%%%%%%
\subitem $\bullet$ $\Sigma$ is a vertical graph over the open region of $\h^n$ bounded by two
parallel horospheres $\mathcal H_{0}$ and $\mathcal H_+$.
%%%%%%%%%%%%%%%%%%%%%%%%%%%%%%%%%%%%%%%%%%%%%%%%%%%%%%%%%%%%%
\subitem $\bullet$ The height of $\Sigma$ is bounded from below and unbounded from above.
%%%%%%%%%%%%%%%%%%%%%%%%%%%%%%%%%%%%%%%%%%%%%%%%%%%%%%%%%%%%%%%%%%%%%%%%%%%%
\subitem $\bullet$ $\Sigma$ is tangent to $\mathcal H_0\times\{0\}$, where it reaches  its minimal height, and it
is vertically  asymptotic to $\mathcal H_{+}\times\mathbb R$.
\end{itemize}
\end{theorem}

\begin{figure}[h]
 \centering
  \includegraphics[width=3.7cm]{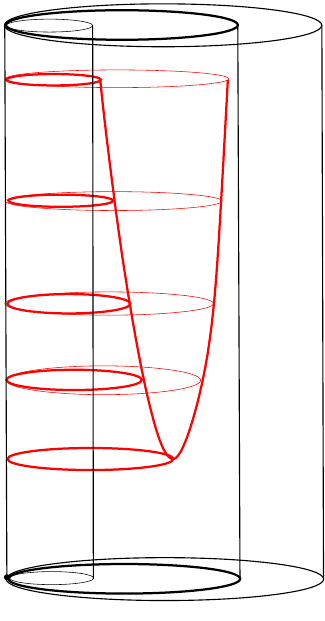}
   \hspace{2cm}
 \includegraphics[width=4cm]{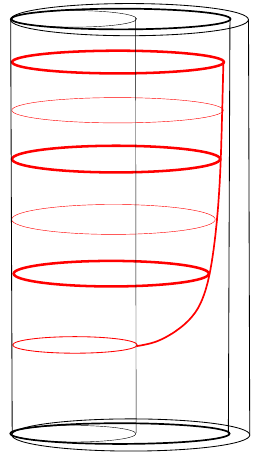}
 \caption{\small The parabolic  $n$-grim reaper for $n$ odd (left) and $n$ even (right).}
 \label{fig-Kparabolicgrimreapers}
\end{figure}

\begin{remark} \label{rem-nparabolic}
The considerations of Remark~\ref{rem-uniquenessgraphparabolic} apply here. More precisely,
given $s_0\in\R$, denote by $\tau_{s_0}$ the solution to~\eqref{eq-IVPparabolicr=n}
with initial condition  $y(s_0)=0.$ It is easily checked that (see Fig.~\ref{fig-parabolicgraphs}):
\[
\tau_{s_0}(s)=\tau_{0}(s-s_0)\,\,\,\forall s\in(s_{\min}(\tau_{s_0}),s_{\max}(\tau_{s_0})),
\]
which implies that all $(\mathcal H_s,\phi)$-graphs obtained from the
solutions $\tau_{s_0}$ are congruent to the one obtained from $\tau_0.$
\end{remark}

As it was for the parabolic $r<n$  case
(cf.~Proposition~\ref{prop-uniquenessgraphparabolic}),
the following uniqueness result holds for parabolic $n$-translators.

\begin{proposition} \label{prop-uniquenessgraphparabolic-n}
Let $\Sigma$ be a connected parabolic $n$-translator in $\hr$
which is a vertical graph over an open set of \,$\h^n$.
Then, up to an ambient isometry, $\Sigma$ is an open set of
a parabolic $n$-grim reaper.
\end{proposition}

\subsection{Hyperbolic translators to $n$-MCF in $\mathbb H^n\times\mathbb R$}
Concluding this section, we shall consider hyperbolic $n$-translators
in $\hr,$ i.e., those invariant by horizontal hyperbolic translations.
In this setting, equation~\eqref{eq-hyperbolicMCFode2} becomes
\begin{equation}\label{eq-hyperbolicMCFode2n}
\tau'(s)=n\sqrt{1-\tau^{\frac 2n}(s)}\coth^{n-1}(s),
\end{equation}
so that the associated Cauchy problem is
\begin{equation} \label{eq-IVPn3}
\left\{
\begin{array}{l}
y'(s)=n\sqrt{1-y^{\frac 2n}(s)}\coth^{n-1}(s),\\[1ex]
y(s_0)=y_0, \,\,\, (s_0,y_0)\in\Omega,
\end{array}
\right.
\end{equation}
where $\Omega:=(0,+\infty)\times(-1,1)$.

Analogously to the parabolic case in the preceding subsection,
applying Propositions~\ref{prop-n01}--\ref{prop-nevenunboundedphi} as in the proof
of Theorem~\ref{th-nrot} yields the following result.

\begin{theorem} \label{th-nhyperbolic}
Let $n\ge 2$ be an integer. Then, the following assertions hold:
\begin{itemize}[parsep=1ex]
\item[\rm i)] If $n$ is odd,  there exists a one parameter family
$\mathscr G_n:=\{\Sigma_\lambda\,;\, \lambda>0\}$ of
hyperbolic $n$-translators in $\mathbb H^n\times\mathbb R$
(to be called  \emph{hyperbolic $n$-grim reapers}) which has the following
properties (Fig.~{\rm\ref{fig-hyperbolicn-grimreaper}}):
%%%%%%%%%%%%%%%%%%%
\subitem $\bullet$ For each $\lambda>0,$ $\Sigma_\lambda$
is a vertical graph over the open region of \,$\h^n$ bounded by two
parallel equidistant hypersurfaces $\mathcal E_{\pm}$. In particular, $\Sigma_\lambda$
is homeomorphic to $\R^n.$
%%%%%%%%%%%%%%%%%%%%%%%%%%%%%%%%%%%%%%%%%%%%%%%%%%%%%%%%%%%%%
\subitem $\bullet$ The height of $\Sigma_\lambda$ is bounded from below and unbounded from above, and
$\Sigma_\lambda$ is vertically asymptotic  to both cylinders $\mathcal E_{\pm}\times\R.$
%%%%%%%%%%%%%%%%%%%%%%%%%%%%%%%%%%%%%%%%%%%%%%%%%%%%%%%%%%%%%%%%%%%%%%%%%%%%
\subitem $\bullet$ $\Sigma_\lambda$ is $C^2$-singular along its equidistant
hypersurface  of minimal height.
%%%%%%%%%%%%%%%%%%%%%%%%%%%%%%%%%%%%%%%%%%%%%%%%%%%%%%%%%%%%%
%%%%%%%%%%%%%%%%%%%%%%%%%%%%%%%%%%%%%%%%%%%%%%%%%%%%%%%%%%%%%
\item[\rm ii)] If $n$ is even,
there exists a one parameter family
$\mathscr G_n:=\{\Sigma_\lambda\,;\, \lambda>0\}$ of
hyperbolic $n$-translators in $\mathbb H^n\times\mathbb R$
(to be called  \emph{hyperbolic $n$-grim reapers}) which has the following
properties (Fig.~{\rm\ref{fig-hyperbolicn-grimreaper}}):
%%%%%%%%%%%%%%%%%%%
\subitem $\bullet$ For each $\lambda>0,$ $\Sigma_\lambda$ is a vertical graph over the open region
of $\h^n$ bounded by two
parallel equidistant hypersurfaces  $\mathcal E_{\pm}$.
%%%%%%%%%%%%%%%%%%%%%%%%%%%%%%%%%%%%%%%%%%%%%%%%%%%%%%%%%%%%%
\subitem $\bullet$ The height of $\Sigma_\lambda$ is bounded from below and unbounded from above.
%%%%%%%%%%%%%%%%%%%%%%%%%%%%%%%%%%%%%%%%%%%%%%%%%%%%%%%%%%%%%%%%%%%%%%%%%%%%
\subitem $\bullet$ $\Sigma_\lambda$ is tangent to $\h^n\times\{0\}$,
where it reaches  its minimal height, and it
is vertically asymptotic to $\mathcal E_{+}\times\mathbb R$.
\end{itemize}
\end{theorem}

\begin{figure}[htb]
 \centering
 \includegraphics[width=5cm]{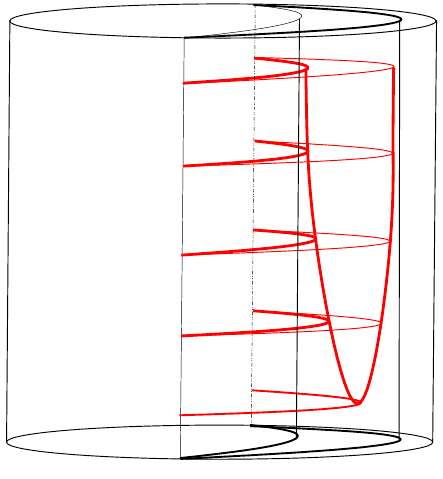}
  \hspace{1cm}
  \includegraphics[width=5cm]{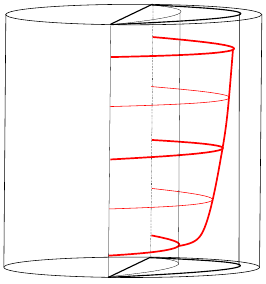}
 \caption{\small Hyperbolic  $n$-grim reaper for $n$ odd (left) and $n$ even (right).}
 \label{fig-hyperbolicn-grimreaper}
\end{figure}

\begin{remark} \label{rem-nsymmetryequation}
As in the case $r<n$ (cf. Remark~\ref{rem-symmetryequation}),
the $n$-translators obtained from the subfamily
$\{\mathcal E_s\,;\, s<0\}$ (of the family of equidistant hypersurfaces of
a totally geodesic hypersurface $\mathcal E_0$ of $\h^n$) are congruent to those
obtained in Theorem~\ref{th-nhyperbolic}.
\end{remark}

As it was for the rotational and parabolic cases of the preceding subsections,
the following uniqueness result for hyperbolic $n$-translators holds.

\begin{proposition} \label{prop-uniquenessgraphhyperbolic-n}
Let $\Sigma$ be a connected hyperbolic $n$-translator in $\hr$
which is a vertical graph over an open set of \,$\h^n$.
Then, up to an ambient isometry, $\Sigma$ is an open set of
a hyperbolic grim reaper.
\end{proposition}

\section{Uniqueness Results} \label{sec-uniqueness}

In this final section, we shall
classify the invariant $r$-translators of $\qr$.
In fact, we shall prove that, up to ambient isometries,
any invariant $r$-translator of $\qr$ is one of those obtained in the preceding sections.

\begin{definition}
Given integers $n\ge 2$ and $r\in\{1,\dots, n\},$ an $r$-translator
of $\qr$ will be called \emph{fundamental} if it is congruent
(see Remarks~\ref{rem-reflectedtranslator},~\ref{rem-congruentreven}--\ref{rem-symmetryequation}, and~\ref{rem-nparabolic}--\ref{rem-nsymmetryequation})
to one of the following invariant hypersurfaces:
\begin{itemize}[parsep=1ex]
\item a vertical hyperplane, i.e., a cylinder over a totally geodesic hypersurface of $\q_\epsilon^n$ (cf.~Example~\ref{exem-stationary});
\item a cylinder over a sphere of $\mathbb Q_\epsilon^n$, a horosphere of $\h^n$, or an equidistant hypersurface of $\h^n$ (for $r=n$, cf.~Example~\ref{exem-stationary});
\item a grim reaper (cf.~Example~\ref{exam-grimreaper}, item (ii) of Theorem~\ref{th-existence-hyperbolic};
item (iii) of Theorems~\ref{th-nrot} and~\ref{th-nevenrot}, and Theorems~\ref{th-nparabolic} and~\ref{th-nhyperbolic});
\item a bowl soliton (cf.~item (i) of Theorems~\ref{th-existence}--\ref{th-nevenrot});
\item a translating catenoid (cf.~items (ii) and (iii) of Theorems~\ref{th-existence} and~\ref{th-existence-parabolic}, and items
(iii) and (iv) of Theorem~\ref{th-existence-hyperbolic});
\item a translating cone (cf.~Theorem~\ref{th-nrot}-(ii));
\item a peaked bowl soliton (cf.~Theorem~\ref{th-nevenrot}).
\end{itemize}
\end{definition}

\begin{remark} \label{rem-asymptotic}
Regarding the asymptotic behavior of the fundamental translators of $\qr$,
an interesting phenomenon occurs;
\emph{up to an ambient isometry, any two fundamental $r$-translators
are asymptotic to each other, independently of the isometry groups that fix them}. Indeed, by the results
of the theorems in the preceding sections, in all fundamental $r$-translators,
the angle function ``at infinity'' is equal to the limit angle $\theta_L$
(recall that $\theta_L=0$ if $\epsilon=0$ or $r=n$).
\end{remark}

Gathering all the uniqueness results obtained in the preceding sections; namely,
Propositions~\ref{prop-uniquenessgraphrotational},~\ref{prop-uniquenessgraphparabolic},~\ref{prop-uniquenessgraphhyperbolic},
and~\ref{prop-uniquenessgraphrotational-n}--\ref{prop-uniquenessgraphhyperbolic-n}, gives the following result.

\begin{proposition} \label{prop-generaluniquenessgraph}
Let $\Sigma$ be a connected invariant $r$-translator in $\qr$
which is a vertical graph over an open set of \,$\q_\epsilon^n.$
Then, $\Sigma$ is an open set of one of the fundamental $r$-translators.
\end{proposition}

By means of this last proposition,  we
classify now all invariant $r$-translators of $\qr$.

\begin{theorem} \label{th-uniquenesssymetric}
Let $\Sigma$ be a connected invariant  $r$-translator
of \,$\qr.$ Then, $\Sigma$ is an open set of a fundamental $r$-translator.
\end{theorem}
\begin{proof}
Let us suppose first that $\theta$ never vanishes on $\Sigma.$ Then,
$\Sigma$ is given by a union of invariant vertical graphs.
By Proposition~\ref{prop-generaluniquenessgraph}, each such graph is contained in one and only one
of the fundamental $r$-translators. Then, since $\Sigma$ is connected,
the same is true for $\Sigma.$

Suppose now that $\theta$ vanishes on an open set of $\Sigma.$ Then, since $\Sigma$ is
invariant and connected, it must be contained in a cylinder over one of the following
hypersurfaces of $\q_\epsilon^n$: a totally geodesic hyperplane,
a geodesic sphere,  a horosphere, or a equidistant hypersurface.
These cylinders, of course, are all $r$-translators.

Finally, assume that the set $\mathcal\Sigma'\subset\Sigma$ on which $\theta$ never vanishes
is open and dense in $\Sigma.$ Then, from the first part of the proof,
any connected component of $\mathcal\Sigma'$ is contained
in one and only one fundamental translator to $r$-{\rm MCF}.
The result, then, follows from  the connectedness of  $\Sigma.$
\end{proof}

\begin{remark} \label{rem-liramartin}
In \cite{lira-martin}, Lira and Martín considered translators to MCF in products
$M\times\R,$ where $M$ is a Hadamard manifold endowed
with a rotationally invariant metric. In their Theorem 12, they obtained a one-parameter family
of translators which, for $M=\h^n$, coincides with the one-parameter family $\mathscr B$ of
hyperbolic bowl-solitons we obtained in Theorem~\ref{th-existence-hyperbolic}-(i). Their methods, though,
are different from ours. In addition, in their Theorem 13, they aim to list
all possible invariant translators to MCF in the products $M\times\R$. However, for $M=\h^n$,
the hyperbolic translating catenoids we obtained in Theorem~\ref{th-existence-hyperbolic}-(ii) seem to be
missing in their statement.
\end{remark}

In our next two results, we classify
the translators to $r$-MCF in $\qr$ whose $r$-th mean curvature is constant,
as well as those which are isoparametric.

\begin{theorem} \label{th-CMC}
Let $\Sigma$ be a connected $r$-translator  in
$\qr$ whose $r$-th mean curvature $H_r$ is constant. Then, $\Sigma$ is
necessarily an open set of one of the following hypersurfaces:
\begin{itemize}[parsep=1ex]
\item a vertical cylinder  over an $r$-minimal hypersurface of \,$\q_\epsilon^n,$
\item a vertical cylinder over an arbitrary hypersurface of \,$\q_\epsilon^n$ (if $r=n$),
\item the parabolic $r$-bowl soliton of \,$\h^n\times\R.$
\end{itemize}
\end{theorem}

\begin{proof}
It follows from the hypothesis that the angle function $\theta$ of $\Sigma$ is a constant which,
without loss of generality, we can assume nonnegative.
Then, from~\cite[Corollary 4]{delima-roitman}, $\Sigma$ is locally an $(M_s,\phi)$-graph
(if $\theta>0$),  an open set of a vertical  cylinder $\Gamma\times\R$ over a hypersurface
$\Gamma$ of $\q_\epsilon^n$ (if \,$\theta=0$), or an open set of a  horizontal hyperplane (if $\theta=1$).
Clearly, horizontal hyperplanes are not translators and, if $r<n,$ $\Gamma\times\R$ is
an $r$-translator if and only if $\Gamma$ is $r$-minimal in \,$\q_\epsilon^n.$ For $r=n,$
$\Gamma\times\R$ is an $r$-translator for any hypersurface $\Gamma\subset\q_\epsilon^n.$

So, we can assume that $\Sigma$ is locally an $(M_s,\phi)$-graph. Then,
by~\eqref{eq-theta2}, the associated
$\rho$ function  is a positive constant. From this and \eqref{eq-principalcurvatures},
we have that the $r$-th mean curvature $H_r^s$ of $M_s$ is given by $H_r^s=(-\rho)^{-r}H_r,$
which implies that $H_r^s$  is constant.
Hence, by~\cite[Theorem 1.1]{jin-ge}, the family $M_s$ is isoparametric.
Since $H_r^s$ is not zero and is independent of $s,$
each parallel $M_s$ must be an open set of a  horosphere of $\h^n,$
which implies that $\Sigma$ is an invariant parabolic $r$-translator.
Then, from Theorem~\ref{th-uniquenesssymetric}, $\Sigma$ is an open set of
a parabolic fundamental $r$-translator. However, the only such translator
having constant $r$-th mean curvature is the parabolic $r$-bowl soliton
of $\h^n\times\R$ (cf.~Theorem~\ref{th-existence-parabolic}-(i)).
\end{proof}

\begin{theorem} \label{th-isoparametric}
Let $\Sigma$ be a connected isoparametric $r$-translator  in
$\qr$ . Then, $\Sigma$ is
necessarily an open set of one of the following hypersurfaces:
\begin{itemize}[parsep=1ex]
\item a vertical hyperplane of \,$\qr$,
\item a vertical cylinder over an isoparametric hypersurface of \,$\q_\epsilon^n$ (if $r=n$),
\item the parabolic $r$-bowl soliton of \,$\h^n\times\R.$
\end{itemize}
\end{theorem}

\begin{proof}
It is immediate that a cylinder $\Gamma\times\R$ over a connected hypersurface
$\Gamma$ of $\q_\epsilon^n$ is isoparametric in $\qr$ if and only if
$\Gamma$ is isoparametric in $\q_\epsilon^n$.
Considering the classification of isoparametric hypersurfaces of
$\q_\epsilon^n$ (see Section~\ref{sec-isoparametric}), we conclude that the
only such hypersurfaces which are $r$-minimal are the totally geodesic ones.
Thus, since isoparametric hypersurfaces are necessarily CMC,
assuming $\Gamma\times\R$ is an isoparametric $r$-translator, we have from Theorem~\ref{th-CMC} that
$\Gamma$ must be either a totally geodesic hyperplane of $\q_\epsilon^n$ or,
if $r=n$, an isoparametric hypersurface of $\q_\epsilon^n$.

Now, as proved in Theorem~\ref{th-existence-parabolic}-(i), for any $r\in\{1,\dots, n-1\}$, the
parabolic $r$-bowl soliton $\Sigma_L$ of $\h^n\times\R$ is isoparametric. In addition, by Theorem~\ref{th-CMC},
$\Sigma_L$ is the only noncylindrical $r$-translator which has constant mean curvature. The result, then,
follows from this fact and the considerations of the preceding paragraph.
\end{proof}

Given a hypersurface $\Sigma\subset\qr,$ we will call a transversal intersection
$$\Sigma^t:=\Sigma\transv(\q_\epsilon^n\times\{t\})$$
a \emph{horizontal section} of $\Sigma.$

An evident property of any fundamental $r$-translator is that its angle function
is constant along its horizontal sections.
In our last result, as stated below, we characterize the translators to MCF in
$\qr$ which have this property.

\begin{theorem} \label{th-characterizationr=1}
Let $\Sigma\subset\qr$ be a connected  translator to
{\rm MCF} whose angle function $\theta$ is constant on each horizontal section $\Sigma^t\subset\Sigma.$
Then,  one the following occurs:
\begin{itemize}[parsep=1ex]
\item[\rm i)] $\Sigma$ is an open set of a vertical cylinder over a minimal
hypersurface of \,$\mathbb Q_\epsilon^n$.
\item[\rm ii)] $\Sigma$ is given locally by an
$(M_s,\phi)$-graph whose level hypersurfaces are isoparametric.
In particular, if the parallels $M_s$ are umbilical, then $\Sigma$ is an open set of
a fundamental translator to {\rm MCF} in $\qr.$
\end{itemize}
\end{theorem}
\begin{proof}

Let $\Sigma'$ be the open set of $\Sigma$ on which $\theta T$ does not vanish,
where $T$ is the gradient of the height function of $\Sigma$
in $\qr$ (see Section \ref{sec-preliminaries}). If $\Sigma'$ is empty, then
$\theta$ vanishes on $\Sigma,$ since $T$ cannot vanish on an open set of a translator in $\qr.$
In this case, $H=\theta=0$ on $\Sigma$. Hence, by Theorem \ref{th-CMC}, (i) occurs.

Assume now that $\Sigma'$ is nonempty.
Since $\theta$ is constant on each $\Sigma^t,$ we have that $\nabla\theta$ is parallel to $T$
on $\Sigma'.$ This, together with the identity $AT=-\nabla\theta,$ gives that $T$
is a principal direction of $\Sigma'.$
Hence, by \cite[Theorem 6]{delima-roitman}, $\Sigma'$ is given locally by
an $(M_s,\phi)$-graph.
Let us  show  that, for any such $(M_s,\phi)$-graph, the parallels $M_s$ are isoparametric.
With this purpose, consider a horizontal section $\Sigma^t\subset\Sigma.$
From \cite[Lemma 1]{delima-roitman},
the mean curvature $H_t$ of $\Sigma^t$ (as a hypersurface of $\q_\epsilon^n$) and the mean
curvature $H$ of $\Sigma$ relate as
\begin{equation}  \label{eq-Ht}
H_t=-\frac{1}{\sqrt{1-\theta^2}}\left(H-\frac{1}{\|T\|^2}\langle AT,T\rangle\right).
\end{equation}
In addition, we have from \cite[Theorem 6]{delima-roitman} that  $\langle AT,T\rangle/\|T\|^2$ is constant on $\Sigma^t.$
Since $H=\theta$ on $\Sigma,$ it follows from  \eqref{eq-Ht} that $H_t$ is constant
on $\Sigma^t,$ which gives that  the parallels $M_s$ are indeed isoparametric, and so (ii) occurs.
If, in addition, the parallels $M_s$ are totally umbilical, then $\Sigma$ is invariant. In this case, by
Theorem~\ref{th-uniquenesssymetric}, $\Sigma$ is an open set of a fundamental translator to MCF.
\end{proof}

\noindent
{\bf Acknowledgments.} We would like to thank the referee for the careful reading and comments.
We are also indebted to Álvaro Ramos and João P. dos
Santos for their valuable suggestions. G. Pipoli is partially supported by INdAM-GNSAGA
and PRIN 20225J97H5.

\bigskip

The  authors declare no conflict of interest.

\end{document}